%% file: LPV20240106V2.tex
\documentclass[review,onefignum,onetabnum]{siamart190516}
\nolinenumbers

\input{ex_shared}

\ifpdf
\hypersetup{
  pdftitle={ rational approximation of exponential clustered poles},
  pdfauthor={Shuhuang Xiang, and Shunfeng Yang}
}
\fi

\begin{document}

\maketitle

\begin{abstract}
This paper builds further rigorous analysis on the root-exponential convergence for lightning schemes approximating corner singularity problems. By utilizing  Poisson summation formula, Runge's approximation theorem and Cauchy's integral theorem, the optimal rate is obtained for efficient lightning plus polynomial schemes, newly developed by Herremans, Huybrechs and Trefethen 
\cite{Herremans2023},  for approximation of  $g(z)z^\alpha$ or  $g(z)z^\alpha\log z$ in a sector-shaped domain with tapered exponentially clustering poles, where $g(z)$ is analytic on the sector domain. From these results, Conjecture 5.3  in \cite{Herremans2023}  on the root-exponential convergence rate is confirmed and the choice of the parameter $\sigma_{opt}=\frac{\sqrt{2(2-\beta)}\pi}{\sqrt{\alpha}}$ may achieve the fastest convergence rate
among all $\sigma>0$. Furthermore, based on Lehman and Wasow's study of  corner singularities \cite{Lehman1954DevelopmentsIT, Wasow}, together with the decomposition of Gopal and Trefethen \cite{Gopal2019}, root-exponential rates for lightning plus polynomial schemes in corner domains $\Omega$ are validated, and the best choice of lightning clustering parameter $\sigma$ for $\Omega$ is also obtained explicitly. The thorough analysis provides a solid foundation for lightning schemes.

\end{abstract}


\section{Introduction}
\label{sec:Int}
In the study of  partial differential equations in  corner domains, the solutions may possess  isolated branch points 
at the corner points \cite{Wasow}.
The standard techniques for solving such  problems face  challenges in calculating accurate solutions \cite{GopTre2019}. In recent years, efficient and powerful lightning schemes have been developed via rational functions
\begin{equation}\label{eq:rat}
f(z)\approx r_N(z)=\frac{p(z)}{q(z)}=\sum_{j=1}^{N_1}\frac{a_j}{z-p_j}+\sum_{j=0}^{N_2} b_jz^j:=r_{N_1}(z)+b_{N_2}(z),\, N=N_1+N_2
\end{equation}
for corner singularities
 \cite{Brubeck2022,GopTre2019,Gopal2019,Herremans2023,Nakatsukasa2021,Trefethen2021}, which achieve root-exponential
convergence by extensive numerical experiments for solving  Laplace, Helmholtz,
and biharmonic equations (Stokes flow).


For the prototype  $f(x)=x^{\alpha}$ on $[0,1]$ with $0<\alpha<1$, to achieve the best convergence rate $\mathcal{O}(e^{-2\pi\sqrt{\alpha N}})$ \cite{Stahl2003},  Herremans et al. \cite{Herremans2023} introduced a lightning $+$ polynomial approximation (LP) supported by the tapered exponentially clustering poles
\begin{equation}\label{eq:tapered2}
p_j =-C\exp(-\sigma(\sqrt{N_1}-\sqrt{j})),\quad 1\leq j\leq N_1
\end{equation}
with $\sigma>0$ and $C$ a positive number. Especially,
there exist coefficients $\{a_j\}_{j=1}^{N_1}$ and a polynomial $b_{N_2}$ with
$N_2 = \mathcal{O}(\sqrt{N_1})$, for which $r_N(x)$ \cref{eq:rat} having
tapered lightning poles \cref{eq:tapered2}
with $\sigma = \frac{2\pi}{\sqrt{\alpha}}$
satisfies that
\begin{equation}\label{eq: brate}
|r_N(x)-x^\alpha|=\mathcal{O}(e^{-2\pi\sqrt{\alpha N}})
\end{equation}
as $N \rightarrow \infty$, which leads to a significant increase in the achievable accuracy as well as the optimal convergence
rate as
shown by Stahl  \cite{Stahl2003}. Furthermore, the choice of the parameter $\sigma=\frac{2\pi}{\sqrt{\alpha}}$ achieves the fastest convergence rate
among all $\sigma>0$. For more details,
see Herremans et al. \cite{Herremans2023} and Xiang, Yang and Wu \cite{XY2023}.

For turning to problems of scientific computing,  the V-shaped domain, as depicted in {\sc Fig}. \ref{Vsector} (left),  is  pivotal. Herremans et al. \cite{Herremans2023} conjectured on the V-shaped domain, which  has been substantiated through  ample delicate numerical experiments in \cite{Herremans2023}.

{\bf Conjecture 5.3} \cite{Herremans2023}.
There exist coefficients $\{a_j\}_{j=1}^{N_1}$ and a polynomial $b_{N_2}$ with
$N_2 = \mathcal{O}(\sqrt{N_1})$, for which the LP $r_N(z)$ \cref{eq:rat} to $z^\alpha$ endowed with tapered lightning poles \cref{eq:tapered2} parameterized by
\begin{align}\label{eq:optV}
\sigma=\frac{\sqrt{2(2-\beta)}\pi}{\sqrt{\alpha}}
\end{align}
satisfies
\begin{align}\label{eq:rateV}
|r_N(z)-z^\alpha|=\mathcal{O}(e^{-\pi\sqrt{2(2-\beta) N\alpha}})
\end{align}
uniformly for $z\in V_\beta=\{z=xe^{\pm \frac{\beta\pi}{2}i},\ x\in [0,1]\}$ for arbitrary fixed $\beta\in [0,2)$.

\begin{figure}[htbp]
\centerline{\includegraphics[width=14cm]{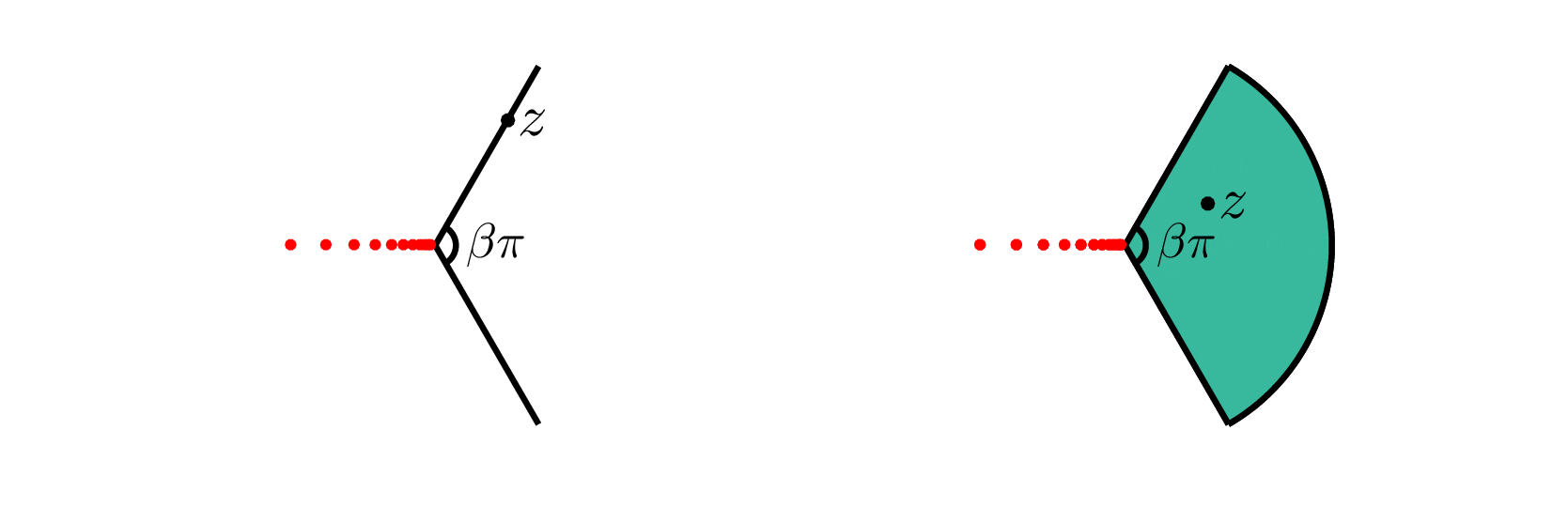}}
\caption{V-shaped domain (left):
$V_\beta=\left\{z: \, z=xe^{\pm \frac{\beta\pi}{2}i}
\mbox{\, with\, $x\in[0,1]$}\right\}$ and
sector domain (right):
$S_{\beta}=\left\{z: \, z=xe^{\pm \frac{\theta\pi}{2}i} \mbox{\, with\, $x\in [0,1]$ and $\theta\in [0,\beta]$}\right\}$
for fixed $\beta\in [0,2)$.
The red points illustrate the distributions of the clustering poles \cref{eq:tapered2}.}
\label{Vsector}
\end{figure}


The conjecture states that the LP, based on the specific  $\sigma=\frac{\sqrt{2(2-\beta)}\pi}{\sqrt{\alpha}}$ for $z^\alpha$ on V-shaped domain $V_\beta$, exhibits a root-exponential convergence rate, which aligns with the best rational approximation in the sense of Stahl \cite{Stahl2003} in the special case $\beta=0$. 
Additionally, we will see that the value $\frac{\sqrt{2(2-\beta)}}{\sqrt{\alpha}}$ is the optimal value for lightning parameter $\sigma$, and hence denoted by $\sigma_{opt}$.

To study the convergence of LPs, it is vital to consider the approximation on the sector-shaped domain $S_\beta$ (see {\sc Fig}. \ref{Vsector} (right)), which includes $V_\beta$ as a special subset.
By employing integral representations of $z^\alpha$ and $z^\alpha\log z$ and along with Runge's approximation theorem, Poisson summation formula \cite{Henrici} and  Cauchy's integral theorem, in this paper, the root-exponential convergence rates of the LPs on  $S_\beta$ is established,
from which  the fastest convergence rates in the uniform norm sense can be attained when the parameter $\sigma$ is chosen as the optimal value $\sigma_{opt}$.

\begin{theorem}\label{mainthm}
There exist coefficients $\{a_j\}_{j=1}^{N_1}$ and a polynomial $b_{N_2}$ with
$N_2 = \mathcal{O}(\sqrt{N_1})$, for which the LP approximation $r_N(z)$ \cref{eq:rat}  to $z^\alpha$ endowed with the
tapered lightning poles \eqref{eq:tapered2} parameterized by $\sigma>0$  satisfies

\begin{equation}\label{eq: rate1}
|r_N(z)-z^\alpha|=\left\{\begin{array}{ll}
\mathcal{O}(e^{-\sigma\alpha\sqrt{N}}),&\sigma\le \sigma_{opt},\\
\mathcal{O}(e^{-\pi\eta\sqrt{2(2-\beta)N\alpha}}),&\sigma> \sigma_{opt},
\end{array}\right.\quad \eta:=\frac{\sigma_{opt}}{\sigma}
\end{equation}
as $N \rightarrow \infty$, uniformly for $z\in S_{\beta}$.
\end{theorem}

From \cref{mainthm}, we see that  Conjecture 5.3 holds in the special case $\sigma_{opt} =\frac{\sqrt{2(2-\beta)}\pi}{\sqrt{\alpha}}$ and $z\in V_\beta$, by which $r_N(z)$ also achieves the fastest rate among all
$\sigma>0$. Furthermore, the similar result holds for $z^\alpha\log z$ in $S_{\beta}$ too. 

\begin{theorem}\label{mainthm2}
There exist coefficients $\{\widetilde{a}_j\}_{j=1}^{N_1}$ and a polynomial $\widetilde{b}_{N_2}$ with
$N_2 = \mathcal{O}(\sqrt{N_1})$, for which the LP $\widetilde{r}_N(z)$ \cref{eq:rat}  to $z^\alpha\log z$ with
tapered lightning poles \eqref{eq:tapered2} parameterized by $\sigma>0$  satisfies

\begin{equation}\label{eq: rate2}
|\widetilde{r}_N(z)-z^\alpha\log{z}|=\left\{\begin{array}{ll}
\mathcal{O}(\sqrt{N\sigma^2\alpha^2}e^{-\sigma\alpha\sqrt{N}}),&\sigma\le \sigma_{opt}\\
\mathcal{O}(e^{-\pi\eta\sqrt{2(2-\beta)N\alpha}}),&\sigma> \sigma_{opt}
\end{array}\right.
\end{equation}
as $N \rightarrow \infty$, uniformly for $z\in S_{\beta}$.
\end{theorem}

Theorems \ref{mainthm} and \ref{mainthm2} can be readily extended to $g(z)z^{\alpha}$ and $g(z)z^{\alpha}\log{z}$ for $g(z)$  analytic on $S_\beta$ by applying Runge's approximation theorem \cite{Gaier1987}.

Furthermore, following the rigorous decompositions  
in Gopal and Trefethen  \cite{Gopal2019}, together with  Lehman and Wasow's contributions on  corner singularities of solutions of
partial differential equations \cite{Lehman1954DevelopmentsIT, Wasow}, these results on $S_\beta$ can be extended to  the case in
which the domain $\Omega$ is a polygon (with every internal angle $<2\pi$, see {\sc Fig}. \ref{general_domain} (first row) for example) for  solving Laplace boundary problems by LPs  with  root-exponential convergence rates in domains with corners, which attests the presume ``in fact we believe convexity is not necessary'' \cite{Gopal2019}.

\begin{theorem}\label{corner}
Let $\Omega$ be a  polygon with corners $w_1,\ldots,w_m$, and let $f$ be
a holomorphic function in $\Omega$ that is analytic on the interior of each side segment and can
be analytically continued to a disk around each $w_k$ with a slit along the exterior bisector
there. Assume $f$ satisfies $f(z)-f(w_k) = (z -w_k)^{\alpha_k}h_k(z)$ as $z\rightarrow w_k$ for each $k$ with some
$\alpha_k\in(0,1)$ and $h_k(z)$ analytic in a neighborhood of $w_k$.
Then there exists a rational approximation $r_n(z)=\sum_{k=1}^mr^{(k)}_{N}(z)+\mathcal{T}(z)$ with $r^{(k)}_{N}(z)$ being the LP approximation around the corner $w_k$ with $\sigma=\frac{\sqrt{2(2-\beta)}}{\sqrt{\alpha}}\pi$, where $\mathcal{T}(z)$ a polynomial of  degree $N_2 = \mathcal{O}(\sqrt{N})$, uniformly for $z\in \Omega$ satisfies
\begin{equation}\label{eq: corrate1}
|r_n(z)-f(z)|=\mathcal{O}\left(e^{-\pi\sqrt{2(2-\beta)N\alpha}}\right)
\end{equation}
as $N\rightarrow \infty$, where $\alpha=\min_{1\le k\le m}\alpha_k$,  $\beta=\max_{1\le k\le m}\beta_k$ and $\beta_k\pi$ denotes the internal angle at $w_k$.
\end{theorem}


The rigorous analysis laid out in this paper provides a solid foundation on the root-exponential convergence for the LPs on corner domains even with curved lines such as pentagram, curvy square, L-shaped and moon-shaped domains, etc. See {\sc Fig}. \ref{general_domain} for example.

\begin{figure}[htbp]
\centerline{\includegraphics[width=16cm]{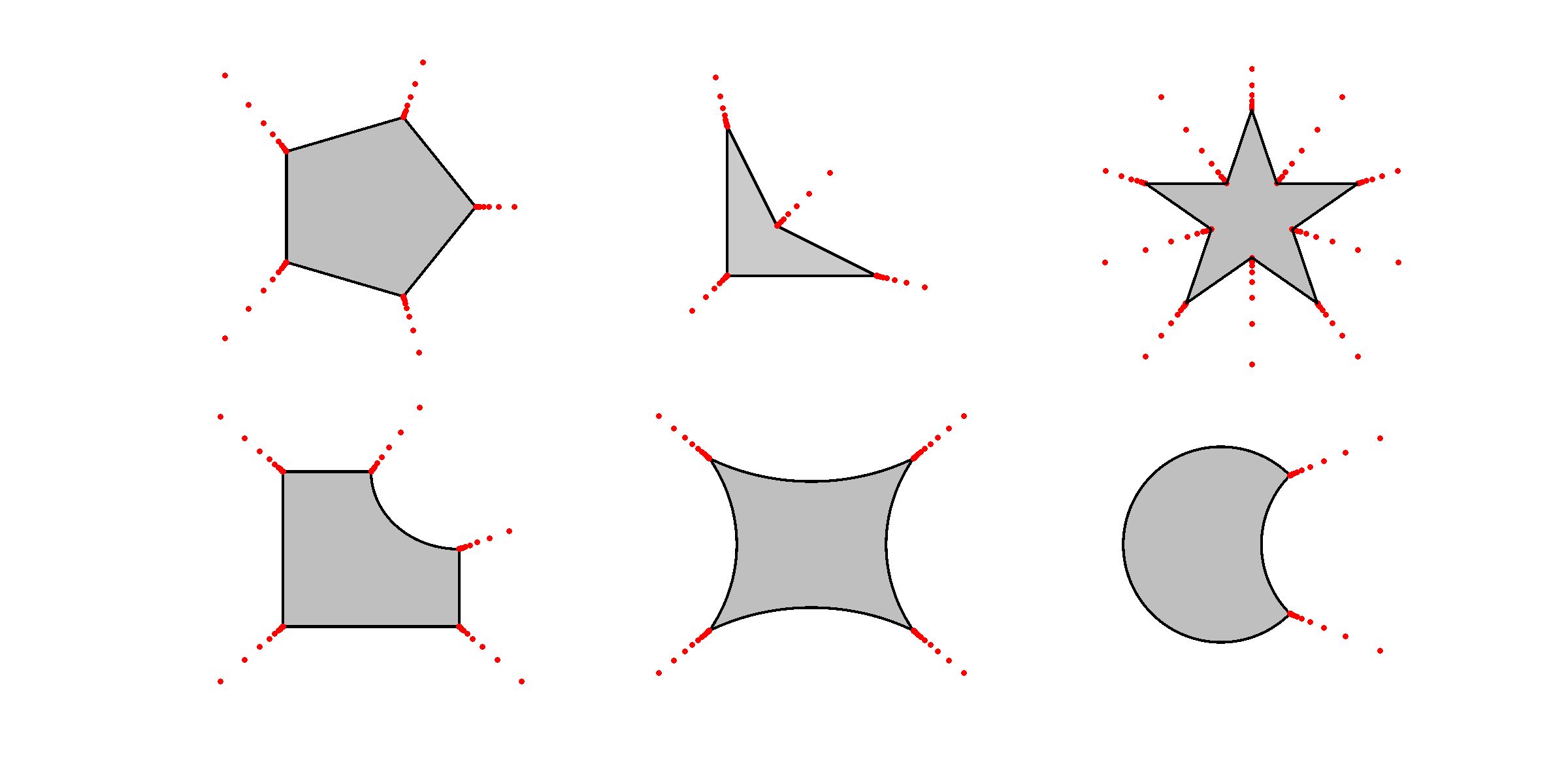}}
\caption{Various corner domains: convex pentagon (first), concave quadrilateral (second), pentagram (third), curvy L-shaped (forth), curvy square (fifth) and moon-shaped (sixth) domains. 
The red points illustrate the distributions of the clustering poles \cref{eq:tapered2}.}
\label{general_domain}
\end{figure}

The rest of this paper is organized as follows. Initially, \cref{sec:2}
is devoted to the generalization of integral representation of $x^{\alpha}$ on interval $[0,1]$ to those of $z^{\alpha}$ and $z^{\alpha}\log{z}$ on the sector domain $S_{\beta}$, whose LP schemes are constructed and the truncated errors are presented. In \cref{sec:3} we present thorough analysis for the convergence rates of numerical quadratures of the integrals for $z^{\alpha}$ and $z^{\alpha}\log{z}$, which play a crucial role in deriving the root-exponential decay rates of LPs.
Then the main results Theorems \ref{mainthm} and \ref{mainthm2} are showed in \cref{sec:convergence}. In \cref{app_cornerdomain} we extend these discuss to corner singularity problems, which shows the root-exponential convergence for LPs  and the best choice of parameter $\sigma$. Finally, some conclusions are presented in \cref{conclusion} and four useful Lemmas are proven in \cref{AppendixA}.
Meanwhile, we include some numerical experiments along with theoretical analysis to illustrate the sharpness of the estimated error bounds and the optimality of parameter choice.

Lots of numerical experiments in this paper are implemented based on the Matlab function \texttt{laplace} developed by Gopal and Trefethen in \cite{Gopal2019}.




\section{Preparatories: integral representations and LP schemes}
\label{sec:2}
In this section, based on Cauchy's residue theorem we first generalise the integral formula of $x^{\alpha}$ in $[0,1]$ to that of $z^{\alpha}$ on the slit disk. Additionally, the formula for $z^{\alpha}\log{z}$ is presented in a similar approach. Then, the LP schemes for them are constructed and the truncated errors are presented based on these preparatories.

\subsection{Integral formula and LP for $z^{\alpha}$}\label{sec:21}
According to \cite[(3.222), p. 319]{GR2014}, $x^\alpha$ on $[0,1]$ can be represented by
 \begin{align*}
x^\alpha&=\frac{\sin(\alpha\pi)}{\alpha\pi}\int_0^{+\infty} \frac{x}{y^{\frac{1}{\alpha}}+x}dy.
\end{align*}
We generalize the integral representation in the complex plane.

\begin{lemma}\label{complex_int_repre}
The following holds for all $z\in\mathbb{C}\setminus(-\infty,0)$ and $\alpha\in(0,1)$
\begin{align}\label{eq:cint}
z^{\alpha}=\frac{\sin{(\alpha\pi)}}{\pi}
\int_0^{+\infty}\frac{zy^{\alpha-1}}{y+z}dy
=\frac{\sin{(\alpha\pi)}}{\alpha\pi}
\int_0^{+\infty}\frac{z}{y^{\frac{1}{\alpha}}+z}dy.
\end{align}
\end{lemma}
\begin{proof}
Consider the integral
\begin{align*}
\int_0^{+\infty}\frac{y^{\alpha-1}}{y+z}dy,\ z\in\mathbb{C}\setminus(-\infty,0],\, \alpha\in(0,1).
\end{align*}
With the help of the residue theorem, we have an integral along a closed Jordan contour $\mathfrak{S}:\,\epsilon\rightarrow R\rightarrow\gamma_R\rightarrow R\rightarrow\epsilon\rightarrow\gamma^{-}_\epsilon$ (see {\sc Fig.} \ref{path_for_representation}) in the complex plane split by the positive real line, which reads as
\begin{align}\label{general integral repres111}
\bigg\{\int_{\epsilon}^{R}+\int_{\gamma_R}
+e^{2i\alpha\pi}\int_{R}^{\epsilon}+\int_{\gamma_\epsilon^{-}}\bigg\}
\frac{y^{\alpha-1}}{y+z}dy=2i\pi{\rm Res}\bigg(\frac{y^{\alpha-1}}{y+z},-z\bigg).
\end{align}
We used in \cref{general integral repres111} the fact
$\log{y}|_{y\in[R\rightarrow\epsilon]}
=\log{y}|_{y\in[\epsilon\rightarrow R]}+2i\pi$, which implies that
$$y^{\alpha-1}|_{y\in[R\rightarrow\epsilon]}
=e^{(\alpha-1)\log{y}}|_{y\in[R\rightarrow\epsilon]}
=e^{2i\alpha\pi}y^{\alpha-1}|_{y\in[\epsilon\rightarrow R]}.$$
Here the radii $R$ and $\epsilon$ of $\gamma_R$ and $\gamma_\epsilon$ are chosen as sufficiently large and small, respectively, such that $0<\epsilon<1<R$ and $-z$ locates inside $\mathfrak{S}$.

\begin{figure}[htbp]
\centerline{\includegraphics[width=5cm]{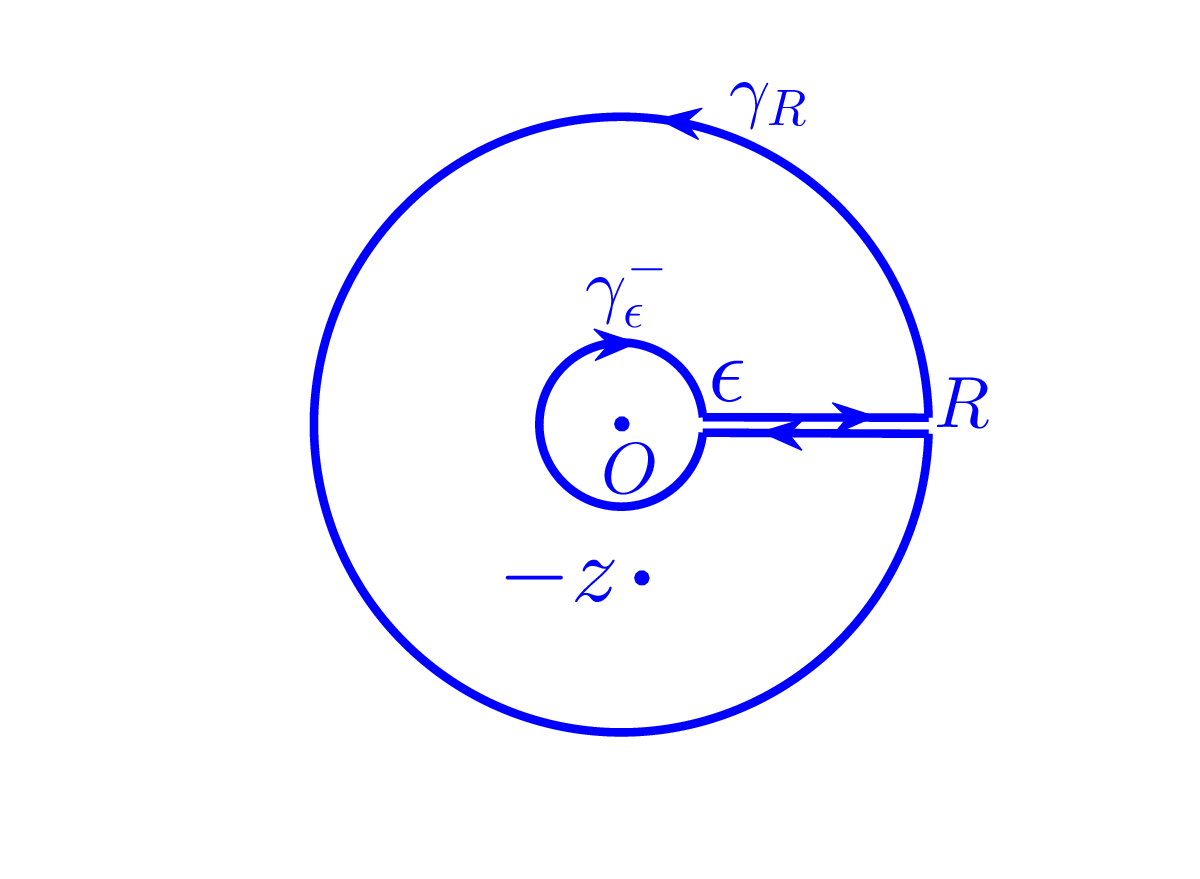}}
\caption{The integral contour $\mathfrak{S}$ of \cref{general integral repres111}.}
\label{path_for_representation}
\end{figure}

Let in \cref{general integral repres111} $R$ tend to $+\infty$ and $\epsilon$ to $0$, we have
\begin{align*}
(1-e^{2i\alpha\pi})\int_0^{+\infty}\frac{y^{\alpha-1}}{y+z}dy
=2i\pi{\rm Res}\bigg(\frac{y^{\alpha-1}}{y+z},-z\bigg)
\end{align*}
due to that
$$\bigg|\int_{\gamma_R}\frac{y^{\alpha-1}}{y+z}dy\bigg|
\le\int_{\gamma_R}\frac{|e^{(\alpha-1)\log{y}}|}{|y|-|z|}ds
\le\frac{R^{\alpha-1}}{R-|z|}2\pi R
$$
approaches to $0$ as $R\rightarrow+\infty$, and
$$\bigg|\int_{\gamma_\epsilon}\frac{y^{\alpha-1}}{y+z}dy\bigg|
\le\frac{\epsilon^{\alpha-1}}{|z|-\epsilon}2\pi\epsilon
$$
tends to $0$ as $\epsilon\rightarrow0$.

By substituting the residue
$${\rm Res}\bigg(\frac{y^{\alpha-1}}{y+z},-z\bigg)
=-e^{i\alpha\pi}z^{\alpha-1}$$
into \cref{general integral repres111}, it follows that
\begin{align*}
\int_0^{+\infty}\frac{y^{\alpha-1}}{y+z}dy
=\frac{-2i\pi e^{i\alpha\pi}z^{\alpha-1}}{1-e^{2i\alpha\pi}}
=\frac{\pi z^{\alpha-1}}{\sin{(\alpha\pi)}},
\end{align*}
then we arrive at the conclusion \cref{eq:cint} for $z\in\mathbb{C}\setminus(-\infty,0)$.  For $z=0$, \cref{eq:cint} is also satisfied. Additionally, we obtain the second equality of \cref{eq:cint} by a change of integral variable $y^\alpha$ by $y$.
\end{proof}

Let $\kappa=\frac{\alpha}{1-\alpha}$.  Following \cite{Herremans2023} and from \cref{eq:cint},  by applying $y=C^{\alpha}e^{t}$, $z^\alpha$ can be rewritten as
\begin{align*}
z^\alpha&=\frac{\sin(\alpha\pi)}{\alpha\pi}\int_0^{+\infty} \frac{z}{y^{\frac{1}{\alpha}}+z}dy
=\frac{\sin(\alpha\pi)}{\alpha\pi}
\left\{\int_{-\infty}^{-T}+\int_{-T}^{\kappa T}+\int_{\kappa T}^{+\infty}\right\}
\frac{zC^{\alpha}e^t}{Ce^{\frac{1}{\alpha}t}+z}dt.\notag
 \end{align*}

For $z=xe^{\pm\frac{\theta\pi}{2}i}\in S_{\beta}$, by $\big{|}Ce^{\frac{1}{\alpha}t}+z\big{|}
=\sqrt{C^2e^{\frac{2}{\alpha}t}+2Cxe^{\frac{1}{\alpha}t}\cos\frac{\theta\pi}{2} +x^2}$ it follows
\begin{align}\label{eq:inequ_neg}
\big{|}Ce^{\frac{1}{\alpha}t}+z\big{|}
\ge \left\{
\begin{array}{ll}
x,&0\le \theta\le 1,\\
x\sin{\frac{\theta\pi}{2}}
\ge x\sin\frac{\beta\pi}{2},&1< \theta\le \beta<2,
\end{array}\right.
 \end{align}
which implies
\begin{align*}
\left|\int_{-\infty}^{-T}
\frac{zC^{\alpha}e^t}{Ce^{\frac{1}{\alpha}t}+z}dt\right|\le C^{\alpha}\int_{-\infty}^{-T}
\frac{xe^t}{x\sin\frac{\beta\pi}{2}}dt=\frac{C^{\alpha}}{\sin\frac{\beta\pi}{2}}e^{-T}.
\end{align*}
While for $t\ge\alpha(2\log{2}-\log{C})$, by $$\big{|}Ce^{(\frac{1}{\alpha}-1)t}+e^{-t}z\big{|}\ge \sqrt{C^2e^{2(\frac{1}{\alpha}-1)t}-2Ce^{(\frac{1}{\alpha}-1)t}e^{-t}}\ge \frac{C}{\sqrt{2}}e^{(\frac{1}{\alpha}-1)t},$$
it derives
\begin{align}\label{eq:inequ}
\frac{|z|}{\big{|}Ce^{(\frac{1}{\alpha}-1)t}+e^{-t}z\big{|}}\le \frac{\sqrt{2}x}{Ce^{(\frac{1}{\alpha}-1)t}}
\le \frac{\sqrt{2}}{Ce^{\frac{1}{\kappa}t}}.
\end{align}
Thus,
it yields for $T=\sqrt{\frac{N_th}{(\kappa+1)^2}}\ge (1-\alpha)(2\log{2}-\log{C})$ that
\begin{align*}
\left|\int_{\kappa T}^{+\infty}
\frac{zC^{\alpha}e^t}{Ce^{\frac{1}{\alpha}t}+z}dt\right|\le \frac{\sqrt{2}}{C^{1-\alpha}}\int_{\kappa T}^{+\infty}
\frac{1}{e^{\frac{1}{\kappa}t}}dt=\frac{\sqrt{2}\kappa}{C^{1-\alpha}}e^{-T}
\end{align*}
and then
\begin{align}\label{eq:intC1}
z^\alpha=&\frac{C^{\alpha}\sin(\alpha\pi)}{\alpha\pi}\int_{-T}^{\kappa T}
\frac{ze^t}{Ce^{\frac{1}{\alpha}t}+z}dt+\mathcal{O}(e^{-T})\notag\\
=&\frac{C^{\alpha}\sin(\alpha\pi)}{\alpha\pi}\int_0^{(\kappa+1)^2T^2}
 \frac{1}{2\sqrt{u}}\frac{ze^{\sqrt{u}-T}}{Ce^{\frac{1}{\alpha}(\sqrt{u}-T)}+z}du+\mathcal{O}(e^{-T})\\
 \approx & r_{N_t}(z) +\mathcal{O}(e^{-T}), \notag
 \end{align}
where $r_{N_t}(z)$ is the discretization of the integral in \cref{eq:intC1} using the trapezoidal rule in $N_t$ quadrature points with  stepsize $h$ given by
\begin{align}\label{eq:ECrat1C}
 r_{N_t}(z)&=\frac{\sin(\alpha\pi)}{2\alpha\pi}h
 \sum_{j=1}^{N_t}\frac{1}{\sqrt{jh}}
\frac{zC^{\alpha}e^{\sqrt{jh}-T}}{Ce^{\frac{1}{\alpha}(\sqrt{jh}-T)}+z}\notag\\
&=\frac{\sin(\alpha\pi)}{2\alpha\pi}\left[
\sum_{j=1}^{N_1}\sqrt{\frac{h}{j}}\frac{p_j|p_j|^\alpha}{z-p_j}
+\left(\sum_{j=N_1+1}^{N_t}\sqrt{\frac{h}{j}}\frac{p_j|p_j|^\alpha}{z-p_j}
+\sum_{j=1}^{N_t}\sqrt{\frac{h}{j}}|p_j|^\alpha\right)\right]\\
&=:r_{N_1}(z)+r_2(z)\notag
 \end{align}
with $N_1={\rm ceil}\big(\frac{N_t}{(\kappa+1)^2}\big)$ and
\begin{align}\label{eq:polcoe}
p_j&=-Ce^{\frac{1}{\alpha}(\sqrt{jh}-T)}
=-Ce^{-\frac{\sqrt{h}}{\alpha}
\big(\sqrt{N_t}/(\kappa+1)-\sqrt{j}\big)},&1\le j\le N_t\\
a_j&=\frac{\sqrt{h}p_j|p_j|^\alpha\sin(\alpha\pi)}{2\sqrt{j}\alpha\pi},&1\le j\le N_1.
\end{align}
It is worth mentioning that in $r_{N_1}$  only  the first $N_1$ poles $p_j\ (1\le j\le N_1)$ are considered. In particular,
compared with  \cref{eq:tapered2}, it implies $h=\sigma^2\alpha^2$.

In addition, $r_2(z)$ in \cref{eq:ECrat1C} can be efficiently approximated  with an exponential convergence rate by a polynomial $b_{N_2}(z)$ with $N_2=\mathcal{O}(\sqrt{N_1})$ from the proof of Runge's approximation theorem \cite[pp. 76-77]{Gaier1987}. Runge's approximation theorem marks the beginning of complex approximation theory.
\begin{theorem}\label{Runge}\cite[1895, Runge]{Gaier1987}
Suppose $K\subset \mathbb{C}$ is compacted, $K^C=\mathbb{C}\setminus K$ is connected, and $f$ is analytic on $K$. Then there exist polynomials $p_n$ such that
\begin{align*}
\lim_{n\rightarrow \infty}\max_{z\in K}|f(z)-p_n(z)|=0.
\end{align*}
\end{theorem}

Based on a sequence of finitely connected domains \cite[pp. 8-9]{Walsh1965}, $p_n$ ($n=1,2\ldots$) are chosen as the interpolation polynomials constructed for the $n+1$ Fekete points on $K$ satisfying
\begin{align*}
\max_{z\in K}|f(z)-p_n(z)|=\mathcal{O}(q^n)
\end{align*}
for some $q\in(0,1)$ independent of $n$.
Then analogous to \cite[pp. 5]{Herremans2023} there is a polynomial $b_{N_2}(z)$ with $N_2=\mathcal{O}(\sqrt{N_1})$
such that
\begin{align}\label{polynomial app}
r_2(z)-b_{N_2}(z)=\mathcal{O}(e^{-T})
\end{align}
uniformly for $z\in S_{\beta}$. Consequently, it holds uniformly for $z\in S_{\beta}$ that
\begin{align}\label{polynomial app1}
r_{N_1}(z)+b_{N_2}(z)=r_{N_t}(z)+\mathcal{O}(e^{-T}).
\end{align}

\subsection{Integral formula and LP for $z^{\alpha}\log{z}$}\label{sec:22}
Furthermore, we may present the integral representation for $z^{\alpha}\log{z}$ similar to Lemma \ref{complex_int_repre}.
\begin{lemma}\label{complex_int_log_repre}
The following holds for all $z\in\mathbb{C}\setminus(-\infty,0)$ and $\alpha\in(0,1)$
\begin{align}\label{eq:cint_log}
z^{\alpha}\log{z}=\frac{\sin{(\alpha\pi)}}{\alpha^2\pi}
\int_{0}^{+\infty}\frac{z\log{y}}{y^{\frac{1}{\alpha}}+z}dy
+\frac{\cos{(\alpha\pi)}}{\alpha}\int_{0}^{+\infty}\frac{zdy}{y^{\frac{1}{\alpha}}+z}.
\end{align}
\end{lemma}
\begin{proof}
Analogous to the proof of \cref{complex_int_repre} for $z\in\mathbb{C}\setminus(-\infty,0]$, we have
\begin{align*}
\frac{1-e^{2i\alpha\pi}}{2i\pi}\int_{0}^{+\infty}
\frac{y^{\alpha-1}\log{y}}{y+z}dy
=&{\rm Res}\bigg(\frac{y^{\alpha-1}\log{y}}{y+z},-z\bigg)+ e^{2i\alpha\pi}\int_{0}^{+\infty}\frac{y^{\alpha-1}}{y+z}dy\\
=&-e^{i\alpha\pi}[z^{\alpha-1}\log{z}+i\pi z^{\alpha-1}]
+e^{2i\alpha\pi}\int_{0}^{+\infty}\frac{y^{\alpha-1}}{y+z}dy,
\end{align*}
which implies from \cref{eq:cint} that
\begin{align*}
z^{\alpha-1}\log{z}
=&\frac{e^{i\alpha\pi}-e^{-i\alpha\pi}}{2i\pi}
\int_{0}^{+\infty}\frac{y^{\alpha-1}\log{y}}{y+z}dy
+e^{i\alpha\pi}\int_{0}^{+\infty}\frac{y^{\alpha-1}}{y+z}dy
-i\pi z^{\alpha-1}\\
=&\frac{\sin{(\alpha\pi)}}{\pi}
\int_{0}^{+\infty}\frac{y^{\alpha-1}\log{y}}{y+z}dy
+\cos{(\alpha\pi)}\int_{0}^{+\infty}\frac{y^{\alpha-1}}{y+z}dy\\
=&\frac{\sin{(\alpha\pi)}}{\alpha^2\pi}
\int_{0}^{+\infty}\frac{\log{y}}{y^{\frac{1}{\alpha}}+z}dy
+\frac{\cos{(\alpha\pi)}}{\alpha}\int_{0}^{+\infty}\frac{dy}{y^{\frac{1}{\alpha}}+z}.
\end{align*}
Thus we arrive at \cref{eq:cint_log} for $z\in\mathbb{C}\setminus(-\infty,0]$. Obviously, \cref{eq:cint_log} also holds for $z=0$.
\end{proof}

Similarly to \cref{eq:intC1}, using \cref{eq:inequ_neg} and \cref{eq:inequ}, and denoting $\chi=\frac{C^{\alpha}\sin{(\alpha\pi)\log C}}{\alpha\pi}
+\frac{C^{\alpha}\cos{(\alpha\pi)}}{\alpha}$, we may rewrite
$z^\alpha\log z$ as
\allowdisplaybreaks
\begin{align}\label{eq:intCPerr}
z^\alpha\log{z}
=&\frac{\sin(\alpha\pi)}{\alpha^2\pi}
\int_{-\infty}^{+\infty}\frac{zC^{\alpha}te^t}{Ce^{\frac{1}{\alpha}t}+z}dt
+\chi\int_{-\infty}^{+\infty}\frac{ze^t}{Ce^{\frac{1}{\alpha}t}+z}dt
\notag\\
=&\frac{\sin(\alpha\pi)}{\alpha^2\pi}
\left\{\int_{-\infty}^{-T}+\int_{-T}^{\kappa T}+\int_{\kappa T}^{+\infty}\right\} \frac{zC^{\alpha}te^t}{Ce^{\frac{1}{\alpha}t}+z}dt\notag\\
&+\chi\left\{\int_{-\infty}^{-T}+\int_{-T}^{\kappa T}+\int_{\kappa T}^{+\infty}\right\}
\frac{ze^t}{Ce^{\frac{1}{\alpha}t}+z}dt\notag\\
=&\frac{\sin(\alpha\pi)}{\alpha^2\pi}\int_{-T}^{\kappa T} \frac{zC^{\alpha}te^t}{Ce^{\frac{1}{\alpha}t}+z}dt
+\chi\int_{-T}^{\kappa T}
\frac{ze^t}{Ce^{\frac{1}{\alpha}t}+z}dt+\mathcal{O}(Te^{-T})\notag\\
=&\frac{\sin(\alpha\pi)}{\alpha^2\pi}\int_0^{(\kappa+1)^2T^2}
 \frac{\sqrt{u}-T}{2\sqrt{u}}\frac{zC^{\alpha}e^{\sqrt{u}-T}}
 {Ce^{\frac{1}{\alpha}(\sqrt{u}-T)}+z}du\\
& +\chi\int_0^{(\kappa+1)^2T^2}
 \frac{1}{2\sqrt{u}}\frac{ze^{\sqrt{u}-T}}{Ce^{\frac{1}{\alpha}(\sqrt{u}-T)}+z}du
 +\mathcal{O}(Te^{-T})\notag\\
=&\frac{\sin(\alpha\pi)}{2\alpha^2\pi}\int_0^{(\kappa+1)^2T^2}
\frac{zC^{\alpha}e^{\sqrt{u}-T}}{Ce^{\frac{1}{\alpha}(\sqrt{u}-T)}+z}du
+\mathcal{O}(Te^{-T})\notag\\
& +\left(\chi-\frac{T\sin(\alpha\pi)}{\alpha^2\pi C^{-\alpha}}\right)\int_0^{(\kappa+1)^2T^2}
\frac{1}{2\sqrt{u}}\frac{ze^{\sqrt{u}-T}}{Ce^{\frac{1}{\alpha}(\sqrt{u}-T)}+z}du\notag\\
\approx & \widetilde{r}_{N_t}(z) +\mathcal{O}(Te^{-T}), \notag
\end{align}
where  $\widetilde{r}_{N_t}(z)$ is the discretization  of the integral in \cref{eq:intCPerr} using the trapezoidal rule in $N_t$ quadrature points with stepsize $h$ given by
\allowdisplaybreaks
\begin{align}\label{eq:ECrat2C}
\widetilde{r}_{N_t}(z)=&\frac{h\sin(\alpha\pi)}{2\alpha^2\pi}
  \sum_{j=1}^{N_t}
\frac{zC^{\alpha}e^{\sqrt{jh}-T}}{Ce^{\frac{1}{\alpha}(\sqrt{jh}-T)}+z}\notag\\
&+ \frac{1}{2}\left(\chi-\frac{T\sin(\alpha\pi)}{\alpha^2\pi C^{-\alpha}}\right)
 \sum_{j=1}^{N_t}\sqrt{\frac{h}{j}}
\frac{ze^{\sqrt{jh}-T}}{Ce^{\frac{1}{\alpha}(\sqrt{jh}-T)}+z}\notag\\
=&\left[\frac{h\sin(\alpha\pi)}{2\alpha^2\pi}
\sum_{j=1}^{N_1}\frac{p_j|p_j|^\alpha}{z-p_j}+
\frac{1}{2}
\left(\frac{\chi}{C^{\alpha}}-\frac{T\sin(\alpha\pi)}{\alpha^2\pi}\right)
\sum_{j=1}^{N_1}\sqrt{\frac{h}{j}}\frac{p_j|p_j|^\alpha}{z-p_j}\right]\\
&+\left[\frac{h\sin(\alpha\pi)}{2\alpha^2\pi}\bigg(\sum_{j=N_1+1}^{N_t}
\frac{p_j|p_j|^\alpha}{z-p_j}
+\sum_{j=1}^{N_t}|p_j|^\alpha\bigg)\right.\notag\\
&+\left.
\frac{1}{2}
\left(\frac{\chi}{C^{\alpha}}-\frac{T\sin(\alpha\pi)}{\alpha^2\pi}\right)
\bigg(\sum_{j=N_1+1}^{N_t}\sqrt{\frac{h}{j}}\frac{p_j|p_j|^\alpha}{z-p_j}
+\sum_{j=1}^{N_t}\sqrt{\frac{h}{j}}|p_j|^\alpha\bigg)\right]\notag\\
=&:\widetilde{r}_{N_1}(z)+\widetilde{r}_2(z),\notag
\end{align}
and in exactly the same manner, $\widetilde{r}_2(z)$ in \cref{eq:ECrat2C} can also be efficiently approximated by a polynomial $\widetilde{b}_{N_2}(z)$ with $N_2=\mathcal{O}(\sqrt{N_1})$ such that
\begin{align}\label{polynomial app_log}
\widetilde{r}_2(z)-\widetilde{b}_{N_2}(z)=\mathcal{O}(e^{-T})
\end{align}
uniformly for $z\in S_{\beta}$. Then it also holds uniformly for $z^\alpha \log z$ and $z\in S_{\beta}$ that
\begin{align}\label{polynomial app1_log}
\widetilde{r}_{N}(z):=\widetilde{r}_{N_1}(z)+\widetilde{b}_{N_2}(z)
=\widetilde{r}_{N_t}(z)+\mathcal{O}(e^{-T}).
\end{align}

\subsection{LPs extend to $g(z)z^{\alpha}$ and  $g(z)z^{\alpha}\log{z}$}\label{sec:23}
Suppose $g(z)$ is an analytic function on $S_\beta$. From \cref{eq:intC1} and \cref{eq:intCPerr}, we see that
\begin{align}\label{polynomial app1_ge}
g(z)z^\alpha  &= g(z)r_{N_1}(z)+g(z)r_2(z)+\mathcal{O}(e^{-T}),\\
g(z)z^\alpha\log z  &= g(z)\widetilde{r}_{N_1}(z)+g(z)\widetilde{r}_2(z)+\mathcal{O}(Te^{-T}).
\end{align}
Similarly from the proof of Runge's approximation theorem \cite{Gaier1987},
$g(z)$, $g(z) r_2(z)$ and $g(z)\widetilde{r}_2(z)$ can be efficiently approximated with exponential convergence rates by  polynomials $b^g(z)$, $b^g_{N_2}(z)$ and $\widetilde{b}^g_{N_2}(z)$ with error bound $\mathcal{O}(e^{-T})$ and degree $N_2=\mathcal{O}(\sqrt{N_1})$ similar to \cite[pp. 5]{Herremans2023}, respectively.

Moreover, notice that $b^g(z)r_{N_1}(z)$ and $b^g(z)\widetilde{r}_{N_1}(z)$ can be written in the form of $\sum_{j=1}^{N_1}\frac{a_j}{z-z_j}$ for some $a_j$ due to $N_2<N_1$. Then Theorems \ref{mainthm} and \ref{mainthm2} also hold for  $g(z)z^{\alpha}$ and  $g(z)z^{\alpha}\log{z}$, respectively.

\section{Convergence rates of quadratures on $r_{N_t}(z)$ and $\widetilde{r}_{N_t}(z)$ for $z\in S_\beta$}\label{sec:3}
From \cref{eq:intC1}, \cref{eq:ECrat1C}, \cref{polynomial app1}, \cref{eq:intCPerr}, \cref{eq:ECrat2C} and \cref{polynomial app1_log}, we are only necessary to focus on the quadrature errors on $r_{N_t}(z)$ and $\widetilde{r}_{N_t}(z)$, from which we may establishes Theorem \ref{mainthm} and Theorem \ref{mainthm2}.

Let $T =\frac{\sqrt{N_t h}}{\kappa+1}$ and for $z\in S_{\beta}$ define
\begin{align}
f(u,z)=&\frac{\sin(\alpha\pi)}{\alpha\pi}
\frac{1}{2\sqrt{u}}\frac{zC^{\alpha}e^{\sqrt{u}-T}}
{Ce^{\frac{1}{\alpha}(\sqrt{u}-T)}+z},\label{eq:func}\\
I(z)=&
\int_0^{(\kappa+1)^2T^2}f(u,z)du,\label{eq:quadraturec}\\
I_{log}(z)=&\frac{1}{\alpha}
\int_0^{(\kappa+1)^2T^2}(\sqrt{u}-T)f(u,z)du+\frac{\alpha\pi\chi}{\sin(\alpha\pi)} \int_0^{(\kappa+1)^2T^2}f(u,z)du\label{eq:quadratureclog}
\end{align}
In the following, we show that the quadrature errors satisfy  uniformly for $z\in S_\beta$ that
\begin{align}\label{eq:err}
I(z)-r_{N_t}(z)=\left\{\begin{array}{ll}
\mathcal{O}(e^{-T}),&\sigma\le \sigma_{opt},\\
\mathcal{O}(e^{-\pi\eta\sqrt{2(2-\beta)N\alpha}}),&\sigma> \sigma_{opt},
\end{array}\right.
\end{align}
\begin{align}\label{eq:errlog}
I_{log}(z)-\widetilde{r}_{N_t}(z)=\left\{\begin{array}{ll}
\mathcal{O}(Te^{-T}),&\sigma\le \sigma_{opt},\\
\mathcal{O}(e^{-\pi\eta\sqrt{2(2-\beta)N\alpha}}),&\sigma> \sigma_{opt},
\end{array}\right.
\end{align}
where $\eta=\frac{\sigma_{opt}}{\sigma}$ and the constants in the $\mathcal{O}$ terms are independent of $T$ and $z\in S_\beta$.

Set $z^{\pm}=xe^{\pm\frac{\theta\pi}{2}i}$ with $x\in[0,1]$ and $\theta\in[0,\beta]$.
In most settings, we discuss only the case of $z^+=xe^{\frac{\theta\pi}{2}i}$, and conclusions related to the case $z^-=xe^{-\frac{\theta\pi}{2}i}$ can be obtained by the same approach.

Notice that
$f(u,z^{\pm})=f(u,xe^{\pm\frac{\theta\pi}{2}i})$ has the simple poles
\begin{equation}\label{eq:all_poles_fux_C}
u_k(z^{\pm})=\big[T+\alpha\log{\frac{x}{C}}
+i\alpha\pi\big(2k-1\pm\frac{\theta}{2}\big)\big]^2,
\end{equation}
where $k=0,\pm1,\ldots$, among which the first two closest to the real axis are
$u_0(z^+)$, $u_1(z^+)$ (for $f(u,z^+)$) and $u_1(z^{-})$, $u_0(z^{-})$ (for $f(u,z^-)$).
For brevity, we denote them by $u_0$, $u_1$ (and $u^{-}_0$, $u^{-}_1$, respectively) with
\begin{align}\label{eq:u_0andu_1}
u_0=v_0-ia_0,\ u_1=v_1+ia_1,\ u_0^{-}=v_1-ia_1,\ u_1^{-}=v_0+ia_0,
\end{align}
where
\begin{align}
v_0=&(\alpha\log{\frac{x}{C}}+T)^2-\frac{1}{4}(2-\theta)^2\alpha^2\pi^2,\hspace{.5cm}
a_0=(2-\theta)\alpha\pi\big(\alpha\log{\frac{x}{C}}+T\big),\label{eq:real_imag_part_v0a0}\\
v_1=&(\alpha\log{\frac{x}{C}}+T)^2-\frac{1}{4}(2+\theta)^2\alpha^2\pi^2,\hspace{.5cm}
a_1=(2+\theta)\alpha\pi\big(\alpha\log{\frac{x}{C}}+T\big).\label{eq:real_imag_part_v1a1}
\end{align}


\subsection{Uniform bounds of quadrature errors near $z=0$}\label{subsec:3.1}
At first we show that both $I(z)$ and $r_{N_t}(z)$ are bounded by $\mathcal{O}(e^{-T})$, while $I_{log}(z)$ and $\widetilde{r}_{N_t}(z)$  by $\mathcal{O}(Te^{-T})$  for $z=xe^{\pm\frac{\theta\pi}{2}i}\in S_\beta$ in the vicinity of the origin.

Let $M_0$  be a positive integer such that
\begin{equation}\label{eq:real}
\alpha\pi\sqrt{M_0h}\ge \max\left\{h,\sqrt{2}\alpha\pi, 2\sqrt{6}\alpha^2\pi^2, \alpha\pi\big(\sqrt{(4+\beta)\alpha\pi/2}+\sqrt[4]{4h}\big)^2\right\}.
\end{equation}
and define
\begin{equation}\label{eq:realM0}
c_0=\sqrt{M_0h+\frac{1}{4}(2-\beta)^2\alpha^2\pi^2+\delta_0},\quad  x^*=Ce^{\frac{1}{\alpha}(c_0-T)},
\end{equation}
where $\delta_0$ is a nonnegative number such that $c_0^2\not=jh$ for $j=1,2,\ldots$.

\begin{lemma}\label{la31}
Let $z=xe^{\pm\frac{\theta\pi}{2}i}$ and $0\le\theta\le \beta<2$. Then the quadrature errors satisfy that
\begin{align}\label{eq:qerr}
\big{|} I(z)-r_{N_t}(z)\big{|}=\mathcal{O}(e^{-T}),\quad \big{|} I_{log}(z)-\widetilde{r}_{N_t}(z)\big{|}=\mathcal{O}(Te^{-T})
\end{align}
hold uniformly  for $x\in[0,x^*]$ and  $\theta\in[0,\beta]$.
\end{lemma}
\begin{proof}
For the case $x=0$, \cref{eq:qerr} holds obviously.
For $z=xe^{\frac{\theta\pi}{2}i}$, $x\in(0,x^*]$ and $0\le \theta\le \beta$,  setting $u_0=v_0-ia_0=r_0e^{i\Theta_0}$ with $r_0=\sqrt{v_0^2+a_0^2}$ and $\Theta_0\in[0,2\pi)$,
it is easy to show by the Euler formula and the half angle formulae that
\begin{align}\label{eq:realp}
\left|\Re\big(\sqrt{u_0}\big)\right|
=&\sqrt{r_0}\left|\cos{\frac{\Theta_0}{2}}\right|=\sqrt{\frac{\sqrt{v_0^2+a_0^2}
+v_0}{2}}
=\left|T+\alpha\log{\frac{x}{C}}\right|,
\\
\left|\Im\big(\sqrt{u_0}\big)\right|
=&\sqrt{r_0}\sin{\frac{|\Theta_0|}{2}}=\sqrt{\frac{\sqrt{v_0^2+a_0^2}-v_0}{2}}
=\pi\alpha\big(1-\frac{\theta}{2}\big).\label{eq:imagp}
\end{align}
Noticing that $x\in(0,x^*]$, then $T+\alpha\log{\frac{x}{C}}\le c_0$,  and by \cref{eq:realp} it leads to
\begin{align}\label{bound_sq_disc}
\sqrt{jh}-\Re\big(\sqrt{u_0}\big)\ge \left\{\begin{array}{ll}
\sqrt{jh},&\Re\big(\sqrt{u_0}\big)\le 0\\
\sqrt{jh}-c_0,&\Re\big(\sqrt{u_0}\big)> 0\end{array}\right\}
\ge \sqrt{jh}-c_0,
\end{align}
while for $t\ge c_0^2$
\begin{align}\label{bound_sqr_t}
\left|e^{\frac{1}{\alpha}(\sqrt{t}-\sqrt{u_0})}-1\right|\ge e^{\frac{1}{\alpha}(\sqrt{t}-c_0)}-1.
\end{align}
Thus by \cref{eq:ECrat1C}, \cref{eq:all_poles_fux_C} and \cref{bound_sq_disc},
together with the monotonicity of $\frac{e^t}{e^{\frac{1}{\alpha}(t-c_0)}-1}$ for $t\in (c_0,+\infty)$, we get
\begin{align}
|r_{N_t}(z)|=&\frac{\sin{(\alpha\pi)}}{\alpha\pi}\left|h\sum_{j=1}^{N_t}\frac{1}{2\sqrt{jh}}
\frac{C^{\alpha}e^{-T}e^{\sqrt{jh}}}{e^{\frac{1}{\alpha}(\sqrt{jh}-\sqrt{u_0})}-1}\right|\notag\\
\le&\sum_{jh<(c_0+2h)^2}\frac{he^{-T}}{2\sqrt{jh}}
\frac{C^{\alpha}e^{\sqrt{jh}}}{|1-e^{\frac{1}{\alpha}(\sqrt{jh}-\sqrt{u_0})}|}
+\sum_{jh\ge(c_0+2h)^2}\frac{he^{-T}}{2\sqrt{jh}}
\frac{C^{\alpha}e^{\sqrt{jh}}}{e^{\frac{he^{-T}}{\alpha}(\sqrt{jh}-c_0)}-1}\notag\\
\le&\sum_{jh<(c_0+2h)^2}\frac{he^{-T}}{4\sqrt{jh}}
\frac{C^{\alpha}e^{\sqrt{jh}}}{e^{\frac{1}{2\alpha}\Re(\sqrt{jh}-\sqrt{u_0})}\sin{\frac{(2-\theta)\pi}{4}}}
+\int_{(c_0+h)^2}^{+\infty}
\frac{e^{-T}C^{\alpha}e^{\sqrt{u}}}{e^{\frac{1}{\alpha}(\sqrt{u}-c_0)}-1}d\sqrt{u}\label{eq:trapC}\\
\le& e^{-T}\left[\sum_{jh<(c_0+2h)^2}\frac{h}{4\sqrt{jh}}
\frac{C^{\alpha}e^{\sqrt{jh}}}{e^{-c_0/(2\alpha)}\sin{\frac{(2-\beta)\pi}{4}}}+e^{c_0}\int_h^{+\infty}
\frac{C^{\alpha}e^{t}}{e^{\frac{1}{\alpha}t}-1}dt\right],\notag
\end{align}
and by \cref{eq:func} and \cref{bound_sqr_t},
\allowdisplaybreaks
\begin{align}\label{eq:intappC}
|I(z)|
=&\bigg|\int_{0}^{(\kappa+1)^2T^2}f(u,z)du\bigg|\notag\\
\le&\frac{\sin(\alpha\pi)}{\alpha\pi}\left\{\int_0^{(c_0+h)^2}
+\int_{(c_0+h)^2}^{(\kappa+1)^2T^2}\right\}\bigg{|}\frac{C^{\alpha}}{2\sqrt{u}}
\frac{e^{\sqrt{u}-T}}{e^{\frac{1}{\alpha}(\sqrt{u}-\sqrt{u_0})}-1}\bigg{|}du\notag\\
\le &\int_0^{(c_0+h)^2}
\frac{C^{\alpha}e^{\sqrt{u}-T}}{\big|e^{\frac{1}{\alpha}(\sqrt{u}-\sqrt{u_0})}-1\big|}d\sqrt{u}
+\int_{(c_0+h)^2}^{(\kappa+1)^2T^2}
\frac{C^{\alpha}e^{\sqrt{u}-T}}{e^{\frac{1}{\alpha}(\sqrt{u}-c_0)}-1}d\sqrt{u}\\
\le&e^{-T}
\int_0^{(c_0+h)^2}\frac{C^{\alpha}e^{\sqrt{u}}}
{2e^{\frac{1}{2\alpha}\Re(\sqrt{u}-\sqrt{u_0})}\sin{\frac{(2-\theta)\pi}{4}}}d\sqrt{u}
+e^{c_0-T}\int_{h}^{+\infty}
\frac{C^{\alpha}e^{t}}{e^{\frac{1}{\alpha}t}-1}dt\notag\\
\le&e^{-T}\left[
\frac{C^{\alpha}}{2e^{-c_0/(2\alpha)}\sin{\frac{(2-\beta)\pi}{4}}}
\int_0^{c_0+h}e^tdt+e^{c_0}\int_{h}^{+\infty}
\frac{C^{\alpha}e^t}{e^{\frac{1}{\alpha}t}-1}dt\right],\notag
\end{align}
which together yield $|I(z)-r_{N_t}(z)|=\mathcal{O}(e^{-T})$ uniformly for $z=xe^{\frac{\theta\pi}{2}i}$ with $x\in(0,x^*]$ and $\theta\in [0, \beta]$. Here in \cref{eq:trapC} and \cref{eq:intappC}, we used \cref{eq:imagp} and the fact that 
\begin{align*}\label{eq:dominant}
&\big|e^{\frac{1}{\alpha}(\sqrt{u}-\sqrt{u_0})}-1\big|
=\big{|}e^{\frac{1}{\alpha}(\sqrt{u}-\Re(\sqrt{u_0}))}
e^{-\frac{i}{\alpha}\Im(\sqrt{u_0})}-1\big{|}\notag\\
=&\sqrt{\big[e^{\frac{1}{\alpha}(\sqrt{u}-\Re(\sqrt{u_0}))}-1\big]^2
+2e^{\frac{1}{\alpha}(\sqrt{u}-\Re(\sqrt{u_0}))}
\left[1-\cos\left(\frac{1}{\alpha}\Im(\sqrt{u_0})\right)\right]}\\
\ge& 2e^{\frac{1}{2\alpha}\big(\sqrt{u}-\Re(\sqrt{u_0})\big)}
\sin\frac{\big|\Im\big(\sqrt{u_0}\big)\big|}{2\alpha}
=2e^{\frac{1}{2\alpha}\big(\sqrt{u}-\Re(\sqrt{u_0})\big)}
\sin{\frac{(2-\theta)\pi}{4}}\notag
\end{align*}
and $\sqrt{u}-\Re(\sqrt{u_0})\ge -c_0$ for  $u\in[0,(c_0+h)^2]$,  and $\sqrt{u}-\Re(\sqrt{u_0})\ge \sqrt{u}-c_0$ and $\big{|}e^{\frac{1}{\alpha}(\sqrt{u}-\Re(\sqrt{u_0}))}-1\big{|}\ge \big{|}e^{\frac{1}{\alpha}(\sqrt{u}-\Re(\sqrt{u_0}))}\big{|}-1\ge e^{\frac{1}{\alpha}(\sqrt{u}-c_0)}-1$
for  $u\in[(c_0+h)^2,+\infty)$.

Analogously, \cref{eq:qerr} holds for $z=xe^{-\frac{\theta\pi}{2}i}$ with
$x\in[0,x^*]$ and $\theta\in[0,\beta]$.

On the other hand, from \cref{eq:quadratureclog} together with the above analysis on $f(u,z)$,  we have for $z=xe^{\pm\frac{\theta\pi}{2}i}$ with $x\in[0,x^*]$ and $\theta\in[0,\beta]$ that
\begin{align*}
|I_{log}(z)|=&\bigg|\frac{1}{\alpha}
\int_0^{(\kappa+1)^2T^2}(\sqrt{u}-T)f(u,z)du\bigg|+\mathcal{O}(e^{-T})\\
=&\bigg|\frac{1}{\alpha}\int_0^{(\kappa+1)^2T^2}
\sqrt{u}f(u,z)du\bigg|+\mathcal{O}(Te^{-T})\\
\le &\frac{(\kappa+1)T}{\alpha}\int_0^{(\kappa+1)^2T^2}
|f(u,z)|du+\mathcal{O}(Te^{-T})=\mathcal{O}(Te^{-T}).
\end{align*}
Similarly from \cref{eq:ECrat2C}, by analogous arguments to \cref{eq:trapC}, we obtain $\widetilde{r}_{N_t}(z)=\mathcal{O}(Te^{-T})$, which leads to the desired uniform bound $I_{log}(z)-\widetilde{r}_{N_t}(z)=\mathcal{O}(Te^{-T})$ too.
\end{proof}

{\sc Fig.} \ref{Trapzoid_rule_nearorigin_vdomains} illustrates the behaviors of quadrature errors $\|I-r_{N_t}\|_{\infty}$ and $\|I_{log}-\widetilde{r}_{N_t}\|_{\infty}$ for $z$ in the vicinity of the original point.

\begin{figure}[htbp]
\centerline{\includegraphics[width=12cm]{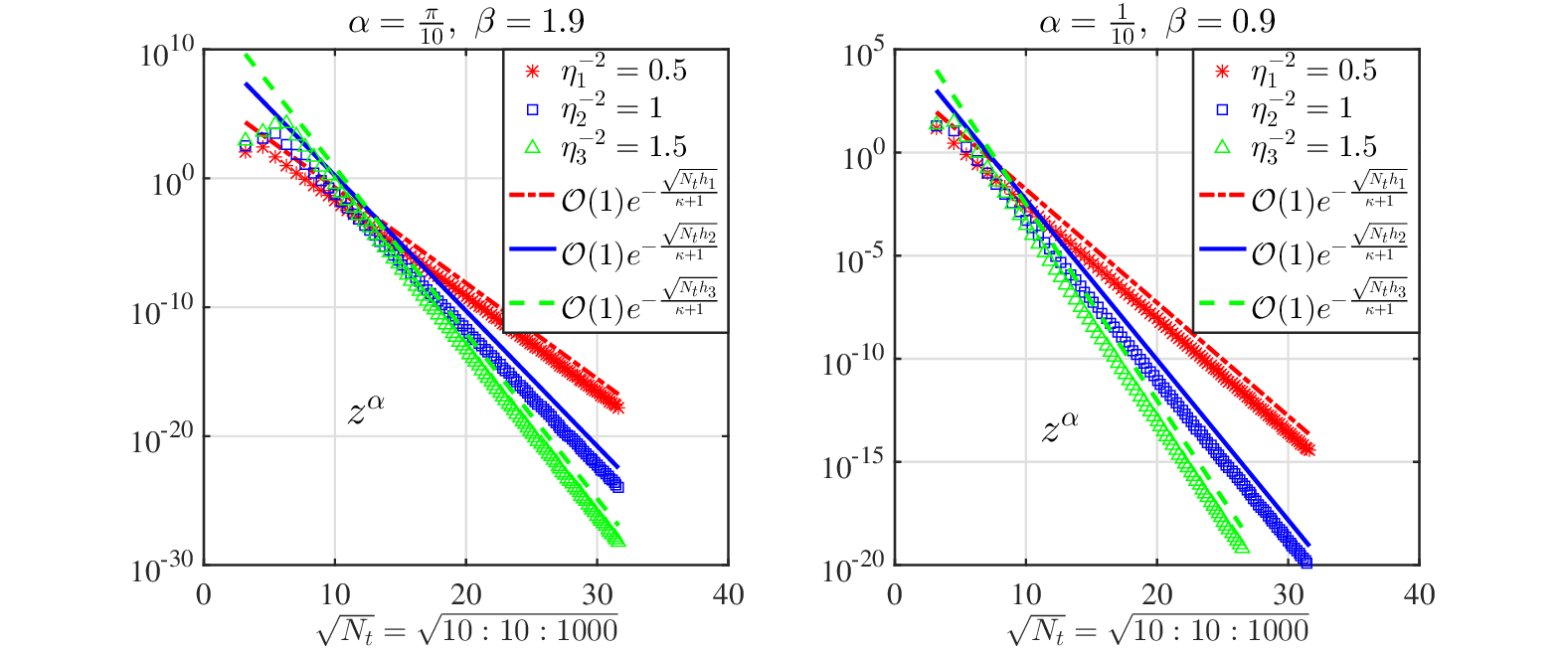}}
\centerline{\includegraphics[width=12cm]{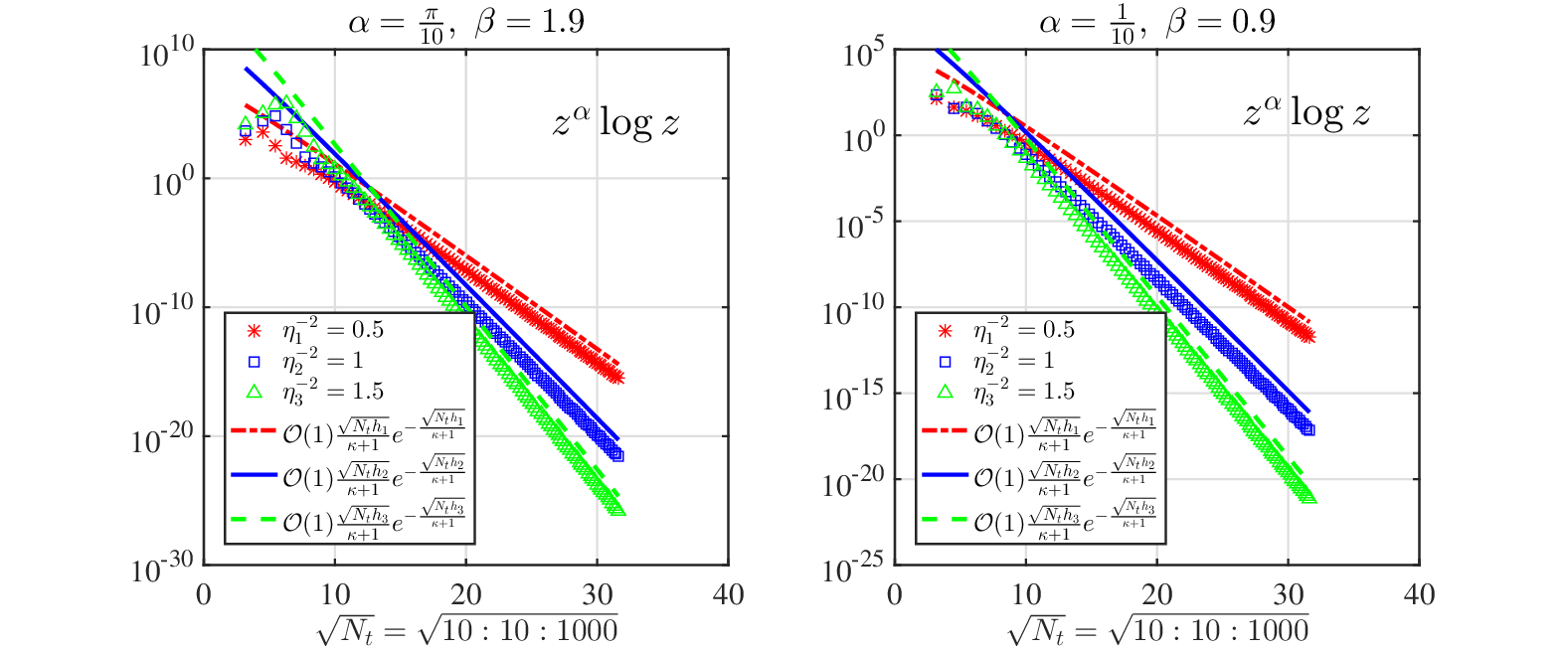}}
\caption{The decay behaviors of  $\|I-r_{N_t}\|_{\infty}$ and $\|I_{log}-\widetilde{r}_{N_t}\|_{\infty}$ for $z=xe^{\pm\frac{\theta\pi}{2}}$ with $x\in[0,x^*]$ and $\theta\in[0,\beta]$ with various step length $h_{\ell}=\eta^{-2}_{\ell}h_{opt}$,
where $h_{opt}=4\pi^2\alpha$ and we set $\eta^{-2}_{\ell}=0.5\ell,\ \ell=1,2,3$. Additionally, $x^*=e^{\frac{1}{\alpha}(c_0-T)}$, $c_0=\sqrt{M_0h+\frac{1}{4}(2-\beta)^2\alpha^2\pi^2+\delta_0}$ and $M_0$ are defined by \cref{eq:realM0} and \cref{eq:real}, respectively. We specific $\delta_0=0$ here.}
\label{Trapzoid_rule_nearorigin_vdomains}
\end{figure}

\subsection{Uniform bounds of quadrature errors for $z=xe^{\pm\frac{\theta\pi}{2}i}\in S_{\beta}$ with $x\in [x^*,1]$}
\label{subsec:3.2}

To obtain the uniform quadrature errors, we consider the error on each V-shaped domain
$$
\mathcal{A}^*_\theta=\big{\{}z=xe^{\pm\frac{\theta\pi}{2}i}:\ x\in[x^*,1]\big{\}},\ x^*=Ce^{\frac{1}{\alpha}(c_0-T)}
$$
for  $\theta\in[0,\beta]$, wherein the quadrature error for $\int_{0}^{+\infty}f(u,z)du$ is analysed in detail. 
Based on these analysis, we establish the uniform bound  independent of $\theta\in [0,\beta]$ and $x\in [x^*,1]$ and  directly apply to $I_{log}(z)-\widetilde{r}_{N_t}(z)$.

Recall the definition of  the simple pole $u_l$ or $u^-_l$ ($l=0,1)$ of $f(u,z)$, we see that $0<(2-\theta)\alpha\pi\frac{\sqrt{h}}{2}\le a_0<a_1$ for $x\in[x^*,1]$.
Then as the functions of $u$, $f(u,z^+)$ and $f(u,z^-)$ are analytic in the strip domains
$$\big{\{}u\in\mathbb{C}:\Re(u)\ge h,\ -a_0<\Im(u)< a_1\big{\}},\ \
\big{\{}u\in\mathbb{C}:\Re(u)\ge h,\ -a_1<\Im(u)< a_0\big{\}}
$$
except for  $u_l$ or $u^-_l$ ($l=0,1)$ on their boundaries,
respectively.
Additionally, we see that all the remaining poles of $f(u,z^+)$ and $f(u,z^-)$ locate outside of the strip domain
\begin{align}\label{extended_stripdomain}
\{u\in\mathbb{C}: \Re(u)>0,\ |\Im(u)|<a_0+a=:A_0\},\quad a:=2\pi\alpha(\alpha\log{\frac{x}{C}}+T).
\end{align}

Furthermore, from the definitions of $c_0$ and $x^*$ \cref{eq:realM0}, we see that $T+\alpha\log{\frac{x}{C}}\ge c_0$ for $x\in [x^*,1]$ and then $v_\ell$ in \cref{eq:real_imag_part_v0a0} and \cref{eq:real_imag_part_v1a1} satisfy $v_\ell=\Re(u_\ell)>M_0h$ ($\ell=0,1$).

In order to get the exponentially convergent rates \cref{eq:err} and \cref{eq:errlog} of the
trapezoidal rules \cref{polynomial app1} and \cref{polynomial app1_log}, respectively, along the way \cite{Trefethen2014SIREV} and \cite{XY2023} it is necessary to introduce the Poisson summation formula (cf. \cite[(10.6-21)]{Henrici} and \cite{Trefethen2014SIREV})
\begin{align}\label{Poisson_summation_formula}
h\sum_{j=-\infty}^{+\infty}\hat f(jh,z)
=\sum_{n=-\infty}^{+\infty}\mathfrak{F}[\hat f]
\big(\frac{2n\pi}{h}\big),
\end{align}
where $\mathfrak{F}[\hat f]
\big(\frac{2n\pi}{h}\big)$ is the $n$-th discrete Fourier transform of

\begin{align}\label{eq:extension_of_f}
\hat f(u,z)=
\begin{cases}
f(u,z), & \Re(u)\ge h,\\
f(h+i\Im(u),z), & -h\le \Re(u)\le h,\\
f(-u,z), & \Re(u)\le-h,
\end{cases}\quad z\in\mathcal{A}^*_\theta.
\end{align}
Consequently,
\begin{align}\label{err_formula_by_fourier}
\int_{-\infty}^{+\infty}\hat f(u,z)du
-h\sum_{j=-\infty}^{+\infty}\hat f(jh,z)
=-\sum_{n\ne0}\mathfrak{F}[\hat f]
\big(\frac{2n\pi}{h}\big).
\end{align}

\bigskip
We first show the decay behavior of the discrete Fourier transform of $\hat f(u,z)$ for fixed $z\in \mathcal{A}^*_\theta$.
For simplicity, we establish the conclusion by leveraging several lemmas that are sketched in \cref{AppendixA}.

\begin{lemma}\label{eq:thm}
Let
$a_{0,\beta}=(2-\beta)\alpha\pi(T+\alpha\log{\frac{x}{C}})$, 
$\hat f(u,z)$ be defined in \cref{eq:extension_of_f} with $z\in \mathcal{A}^*_\theta$
and its discrete Fourier transform be
\begin{align*}
\mathfrak{F}[\hat f]\big{(}\frac{2 n\pi }{h}\big{)}
=\int_{-\infty}^{+\infty}\hat f(u,z)e^{-i\frac{2n\pi}{h}u}du,\quad
n=0,1,\cdots.
\end{align*}
Then the sum of the discrete Fourier transform decays at an exponential rate
\begin{align}\label{eq:conclusionOfFouriersum}
\sum_{n\ne0}\mathfrak{F}[\hat f]\big{(}\frac{2 n\pi }{h}\big{)}=\mathcal{O}(e^{-T})+\mathcal{O}\left(\frac{x^\alpha}{e^{\frac{2\pi }{h}a_{0,\beta}}-1}\right)
\end{align}
and
\begin{align}\label{errquad}
I(z)-r_{N_t}(z)=\mathcal{O}(e^{-T})+\mathcal{O}\left(\frac{x^\alpha}{e^{\frac{2\pi }{h}a_{0,\beta}}-1}\right),
\end{align}
where all the constants in $\mathcal{O}$ terms are independent of $n$, $T$, $\theta$ and $x$ for $z\in \mathcal{A}^*_\theta$.
\end{lemma}

\begin{proof}
 To avoid repetition, only the case $z=z^+=xe^{\frac{\theta\pi}{2}i}$ is proved here, and the other case $z=z^-=xe^{-\frac{\theta\pi}{2}i}$ can be checked in the exactly same manner.

From the definition of \eqref{eq:extension_of_f}, $\hat f(u,z)$ is continuous and piecewise smooth for $u\in (-\infty,+\infty)$ and arbitrarily fixed $z\in S_{\beta}$. Moreover,
from \cref{eq:inequ_neg} and \cref{eq:inequ},
one can check readily that there exists some positive number $M$ such that
$$
\int_0^{+\infty}|f(u,z)|du\le \int_{-\infty}^{\alpha(2\log{2}-\log{C})}\frac{C^{\alpha}}{\sin\frac{\beta\pi}{2}}e^{t}dt+\int_{\alpha(2\log{2}-\log{C})}^{+\infty} \frac{\sqrt{2}}{C^{1-\alpha}e^{\frac{1}{\kappa}t}}dt<M$$
holds uniformly for all $z\in S_{\beta}$, wherein
$$\int_{-\infty}^{+\infty}|\hat{f}(u,z)|du=2h\max_{z\in S_\beta}|f(h,z)|+2\int_{h}^{+\infty}|f(u,z)|du<M'$$
also holds uniformly for constant $M'$. Then $\hat{f}(u,z)$ satisfies \cite[(10.6-12)]{Henrici}.

Define an $h$-periodic function in $v$
\begin{equation}\label{eq:per}
F(v,z)=\sum_{k=-\infty}^{\infty}\hat{f}(kh+v,z),\ v\in[0,h],
\end{equation}
whose uniform convergence can be checked readily. For convenient narration, we also denote $A_1=a_1+a$, similarly to $A_0=a_0+a$ in \cref{extended_stripdomain}.
Then following \cite[pp. 270]{Henrici} and with the help of Cauchy's integral theorem, the $n$th Fourier $(n\ge1)$ coefficient
of $F(v,z)$ satisfies that
\allowdisplaybreaks[4]
\begin{align}
c_n=&\frac{1}{h}\mathfrak{F}[\hat f]\big{(}\frac{2n\pi}{h}\big{)}
=\frac{1}{h}\int_0^hF(v,z)e^{-i\frac{2n\pi}{h}v}dv\label{eq:fourier_c}\\
=&\frac{1}{h}\sum_{k=-\infty}^{\infty}
\int_{kh}^{(k+1)h}\hat f(u,z)e^{-i\frac{2n\pi}{h}u}du\notag\\
=&\frac{1}{h}
\int_{h}^{+\infty} f(u,z)e^{-i\frac{2n\pi}{h}u}du+\frac{1}{h}
\int_{h}^{+\infty} f(u,z)e^{i\frac{2n\pi}{h}u}du
\notag\\
&+\frac{2}{h}\int_{0}^{h}f(h,z)\cos{\big{(}\frac{2n\pi}{h}u\big{)}}du\notag\\
=&\frac{1}{h}
\int_{\Gamma^{-}_{\rho,h}} f(u,z)e^{-i\frac{2n\pi}{h}u}du+\frac{1}{h}
\int_{\Gamma^{+}_{\rho,h}}f(u,z)e^{i\frac{2n\pi}{h}u}du
\label{eq:fourier_positive11}\\
=&-\frac{i}{h}\int_{0}^{A_0}f(h-it,z)e^{-\frac{2n\pi}{h}t}dt+\frac{i}{h}\int_{0}^{A_1}f(h+it,z)e^{-\frac{2n\pi}{h}t}dt
\label{eq:fourier_positive11111}\\
&+\frac{1}{h}\left\{
\int_{h-iA_0}^{u_0-ia} +\int_{C^{-}_{\rho}}+\int_{u_0-ia}^{+\infty-iA_0}\right\}f(u,z)e^{-i\frac{2n\pi}{h}u}du\notag\\
&+\frac{1}{h}\left\{
\int_{h+iA_1}^{u_1+ia} +\int_{C^{+}_{\rho}}+\int_{u_1+ia}^{+\infty+iA_1}\right\}f(u,z)e^{i\frac{2n\pi}{h}u}du,\notag\\
=&
-\frac{i}{h}\int_{0}^{A_0}f(h-it,z)e^{-\frac{2n\pi}{h}t}dt+\frac{i}{h}\int_{0}^{A_1}f(h+it,z)e^{-\frac{2n\pi}{h}t}dt\label{eq:fourier_positive}\\
&+\frac{1}{h}\left\{
\int_{h-iA_0}^{+\infty-iA_0} +\int_{C^{-}_{\rho}}\right\}f(u,z)e^{-i\frac{2n\pi}{h}u}du\notag\\
&+\frac{1}{h}\left\{
\int_{h+iA_1}^{+\infty+iA_1} +\int_{C^{+}_{\rho}}\right\}f(u,z)e^{i\frac{2n\pi}{h}u}du,\notag
\end{align}
where we used $\frac{2}{h}\int_{0}^{h}f(h,z)\cos{\big{(}\frac{2n\pi}{h}u\big{)}}du=\frac{2f(h,z)}{h}\int_{0}^{h}\cos{\big{(}\frac{2n\pi}{h}u\big{)}}du=0$,
and
$$C^{-}_{\rho}=\{z=u_0+\rho e^{i\vartheta}\big|\ \vartheta:0\rightarrow-2\pi\},
\quad C^{+}_{\rho}=\{z=u_1+\rho e^{i\vartheta}\big|\ \vartheta:0\rightarrow2\pi\}
\footnote{The valid range of $\vartheta$ for $C_{\rho}^{\pm}$ is actually $\pm\frac{\pi}{2}\rightarrow\pm\frac{5\pi}{2}$. However, it is equivalent to express it as $0\rightarrow\pm2\pi$ due to the invariance of the integral $\int_{C_{\rho}^{\pm}}f(u,z)e^{\pm i\frac{2n\pi}{h}u}du$.}$$
with
$0<\rho=\frac{1}{2}\min\{\alpha^2\pi^2,a_{0,\beta},\frac{1}{N_0}\}$ for some fixed sufficiently large  $N_0$ independent of $z$ and $T$. We also used in \cref{eq:fourier_positive11} Cauchy's integral theorem on the holomorphic function $f(u,z)$, then the integrals on $[h,+\infty)$ are converted to those on the paths (see {\sc Fig}. \ref{integral_contour}):
\begin{align*}
&\Gamma^{-}_{\rho,h}:\ h\rightarrow h-iA_0\rightarrow u_0-ia\rightarrow C_{\rho}^{-}
\rightarrow u_0-ia\rightarrow+\infty-iA_0\rightarrow+\infty,\\
&\Gamma^{+}_{\rho,h}:\ h\rightarrow h+iA_1\rightarrow u_1+ia\rightarrow C_{\rho}^{+}
\rightarrow u_1+ia\rightarrow+\infty+iA_1\rightarrow+\infty.
\end{align*}

\begin{figure}[htp]
\centerline{\includegraphics[width=14cm]{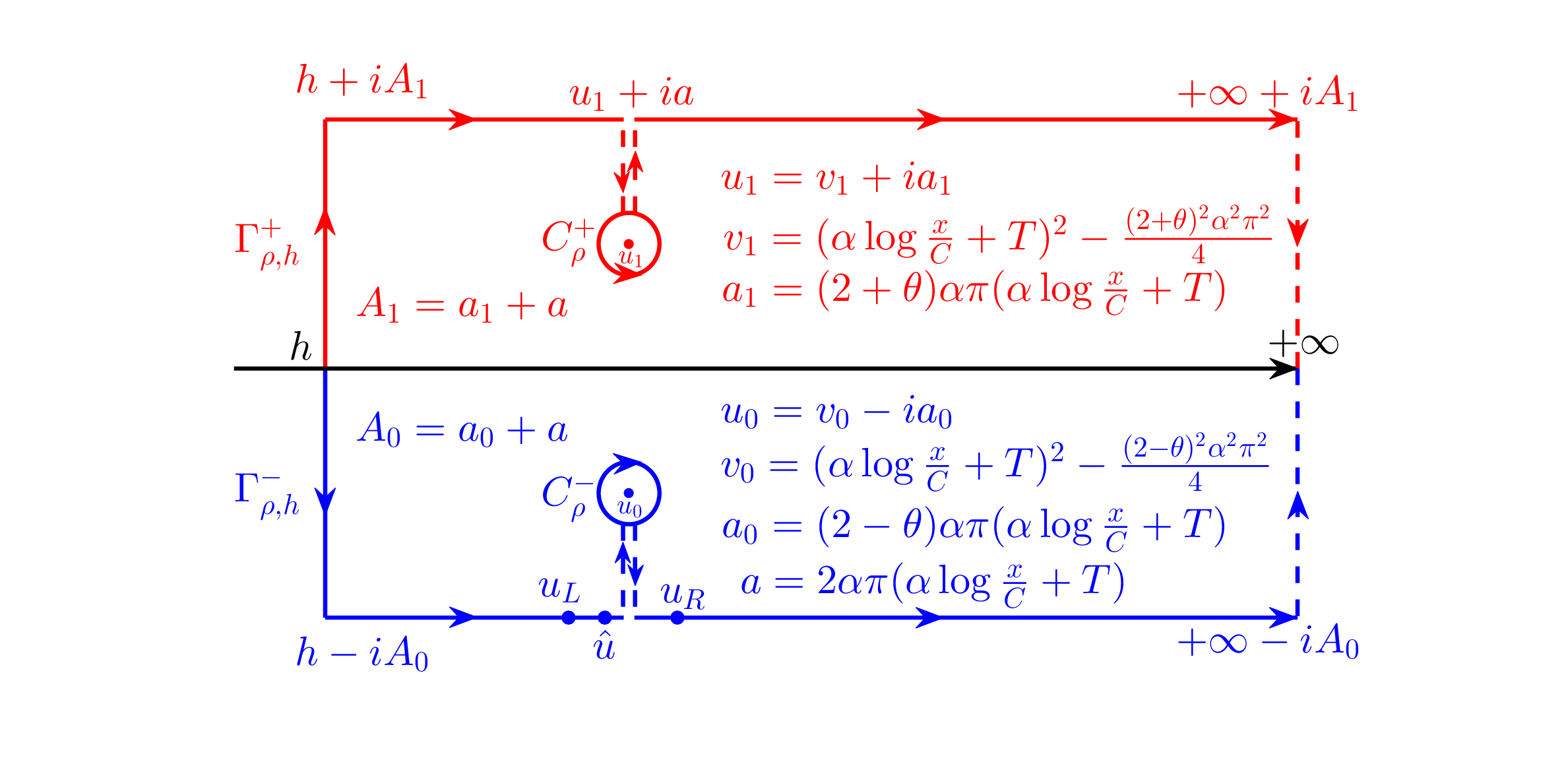}}
\caption{The integral contours $\Gamma^{-}_{\rho,h}$ (blue) and $\Gamma^{+}_{\rho,h}$ (red). The first two nearest poles to the real line of $f(u,z)$ are $u_0$ and $u_1$.
Together with the straight line $[h,+\infty]$ in the opposite direction, they form two closed circuits, wherein $f(u,z)$ is holomorphic.}
\label{integral_contour}
\end{figure}

It is obvious that the integrals on the vertical line segments $(u_0-ia)\rightleftharpoons C^{-}_{\rho}$ and $(u_1+ia)\rightleftharpoons C^{+}_{\rho}$ (not include $C^{\pm}_{\rho}$) can be canceled.
Meanwhile, we used in \cref{eq:fourier_positive11111} the fact that
$$
\lim_{U\rightarrow +\infty}\int_{U-iA_0}^{U}
f(u,z)e^{-i\frac{2n\pi}{h}u}du=0,\
\lim_{U\rightarrow +\infty}\int_{U+iA_1}^{U}
f(u,z)e^{i\frac{2n\pi}{h}u}du=0.
$$

The integrals in \cref{eq:fourier_positive} can be bounded uniformly for $z\in \mathcal{A}^*_\theta$ as follows
\begin{align*}
&\bigg{|}\int_{0}^{A_1}f(h+it,z)
e^{-\frac{2n\pi}{h}t}dt-\int_{0}^{A_0}f(h-it,z)
e^{-\frac{2n\pi}{h}t}dt\bigg{|}\\
&=\mathcal{O}(e^{-T})
\int_0^{\infty}te^{\sqrt{t}}e^{-\frac{2n\pi t}{h}}
dt,
\hspace{3.5cm} \mbox{(see \cref{la3})} \notag\\
&\bigg{|}\int_{C^{-}_{\rho}}
f(u,z)e^{-i\frac{2n\pi}{h}u}du\bigg{|}
=e^{-\frac{2n\pi}{h}a_0}x^{\alpha}\mathcal{O}(1),\notag\\
&\bigg{|}\int_{C^{+}_{\rho}}
f(u,z)e^{i\frac{2n\pi}{h}u}du\bigg{|}
=e^{-\frac{2n\pi}{h}a_1}x^{\alpha}\mathcal{O}(1),
\hspace{1.7cm} \mbox{(see \cref{lemma_inte_circ})}\\
&\bigg{|}
\int_{h-iA_0}^{+\infty-iA_0}f(u,z)e^{-i\frac{2n\pi}{h}u}du\bigg{|}
=e^{-\frac{2n\pi}{h}A_0}x^{\alpha}\mathcal{O}(1),\\
&\bigg{|}
\int_{h+iA_1}^{+\infty+iA_1}f(u,z)e^{i\frac{2n\pi}{h}u}du\bigg{|}
=e^{-\frac{2n\pi}{h}A_1}x^{\alpha}\mathcal{O}(1),
\hspace{.5cm} \mbox{(see \cref{infty})}
\end{align*}
respectively, where the constants in $\mathcal{O}$ terms are independent of $n$, $x$, $\theta$ and $T$ for $z\in \mathcal{A}^*_\theta$.
Then it follows
by  $a_{0,\beta}\le a_i< A_i$ ($i=0,1$) that
\begin{align}\label{eq:fourier_minus}
h|c_n|=\big{|}\mathfrak{F}[\hat f]\big{(}\frac{2 n\pi }{h}\big{)}\big{|}
=\mathcal{O}(e^{-T})\int_0^{+\infty}ue^{\sqrt{u}}e^{-\frac{2n\pi u}{h}}du+e^{-\frac{2n\pi}{h}a_{0,\beta}}x^{\alpha}\mathcal{O}(1).
\end{align}

In the analogous way, from \cref{eq:fourier_c}, \cref{eq:fourier_minus} still holds for $n\le-1$. Thus we get
\begin{align*}
&h\bigg|\sum_{n\ne 0}c_n\bigg|=\bigg|\sum_{n\ne0}\mathfrak{F}[\hat f]\big{(}\frac{2 n\pi }{h}\big{)}\bigg|
=\mathcal{O}(e^{-T})\int_0^{+\infty}\frac{ue^{\sqrt{u}}}{e^{\frac{2\pi u}{h}}-1}du
+\frac{x^{\alpha}\mathcal{O}(1)}{e^{\frac{2\pi}{h}a_{0,\beta}}-1},
\end{align*}
which leads to the desired result \cref{eq:conclusionOfFouriersum}, where we used that $\int_{0}^{+\infty}\frac{ue^{\sqrt{u}}}{e^{\frac{2\pi u}{h}}-1}du$ is convergent and dependent only on $h$.

It is worthy of noting that
\begin{align*}
\int_{-\infty}^{+\infty}\hat f(u,z)du
=&2hf(h,z)+2\int_h^{+\infty}f(u,z)du\\
=&2I(z)+\mathcal{O}(e^{-T}),
\end{align*}
where we applied $\int_{(\kappa+1)^2T^2}^{+\infty}f(u,z)du=\mathcal{O}(e^{-T})$, $f(h,z)=\mathcal{O}(e^{-T})$ and
$$
\bigg|\int_0^hf(u,z)du\bigg|\le \frac{e^{\sqrt{h}-T}C^{\alpha}\sin(\pi\alpha)}
{2\delta_{\theta}\pi\alpha}\int_0^h\frac{1}{\sqrt{u}}du=\mathcal{O}(e^{-T})
$$
with $\delta_\theta=1$ for $0\le\theta\le1$ and $\delta_\theta=\sin{\frac{\beta\pi}{2}}$ for $1<\theta\le\beta<2$ by \cref{eq:inequ_neg}. 
By utilizing the Poisson summation formula (cf. \cite[(10.6-21)]{Henrici} and \cite{Trefethen2014SIREV})
\begin{align}\label{PossionSum}
h\sum_{k=-\infty}^{+\infty}F(kh+\tau)={\rm P.V.}\sum_{k=-\infty}^{+\infty}\mathfrak{F}[\hat f]\left(\frac{2\pi n}{h}\right)e^{\frac{2i\pi n\tau}{h}}
\end{align}
with $\tau=0$, together with
\cref{err_formula_by_fourier} it deduces
\begin{align*}
\int_{-\infty}^{+\infty}\hat f(u,z)du
-h\sum_{j=-\infty}^{+\infty}\hat f(jh,z)
=-h\sum_{n\ne0}c_n,
\end{align*}
then together with \cref{eq:conclusionOfFouriersum}, it implies
\begin{align*}
&\bigg{|}\int_{0}^{+\infty} f(u,z)du
-h\sum_{j=1}^{N_t} f(jh,z)\bigg{|}\notag\\
\le & \frac{h}{2}\sum_{n\ne0}|c_n|+\mathcal{O}(e^{-T})+h
\sum_{j=N_t+1}^{+\infty}|f(jh,z)|\\
\le & \frac{h}{2}\sum_{n\ne0}|c_n|+\mathcal{O}(e^{-T})
+h\frac{\sin{(\alpha\pi)}}{\alpha\pi C^{1-\alpha}}
\sum_{j=N_t+1}^{+\infty}\frac{1}{\sqrt{jh}}
e^{-\frac{1}{\kappa}(\sqrt{jh}-T)}\\
\le&\mathcal{O}(e^{-T})+\mathcal{O}\left(\frac{x^\alpha}{e^{\frac{2\pi}{h}a_{0,\beta}}-1}\right)
+\frac{\sin{(\alpha\pi)}}{\alpha\pi C^{1-\alpha}}
\int_{(\kappa+1)^2T^2}^{+\infty}
\frac{e^{-\frac{1}{\kappa}(\sqrt{u}-T)}}{\sqrt{u}}du\notag\\
=&\mathcal{O}(e^{-T})+\mathcal{O}\left(\frac{x^\alpha}{e^{\frac{2\pi }{h}a_{0,\beta}}-1}\right)\notag
\end{align*}
by applying the fact for $T\ge(1-\alpha)\log{\frac{2x}{C}}$ and $u\ge(\kappa+1)^2T^2$ that
\begin{align*}
|f(u,z)|
\le&\frac{\sin{(\alpha\pi)}}{\alpha\pi}\frac{x}{2\sqrt{u}}
\frac{C^{\alpha}e^{\sqrt{u}-T}}{Ce^{\frac{1}{\alpha}(\sqrt{u}-T)}-x}
=\frac{\sin{(\alpha\pi)}}{\alpha\pi}\frac{x}{2\sqrt{u}}
\frac{C^{\alpha}e^{-\frac{1}{\kappa}(\sqrt{u}-T)}}
{C\big[1-e^{\log{\frac{x}{C}}-\frac{1}{\alpha}(\sqrt{u}-T)}\big]}\\
\le&\frac{\sin{(\alpha\pi)}}{\alpha\pi}\frac{e^{-\frac{1}{\kappa}(\sqrt{u}-T)}}{C^{1-\alpha}\sqrt{u}}.
\end{align*}
Therefore, by $\int_{0}^{+\infty} f(u,z)du=I(z)+\mathcal{O}(e^{-T})$ we obtain the desired result \cref{errquad}.
\end{proof}

\bigskip
Moreover, since $I_{log}(z)=\frac{1}{\alpha}
\int_0^{(\kappa+1)^2T^2}(\sqrt{u}-T)f(u,z)+\frac{\chi\alpha\pi}{\sin(\alpha\pi)} \int_0^{(\kappa+1)^2T^2}f(u,z)du$ \cref{eq:quadratureclog}, by \cref{eq:thm} we are only concerned with the quadrature error on the integrand
$$f_{log}(u,z)=\frac{1}{\alpha}(\sqrt{u}-T)f(u,z).$$

Through the same procedure, we obtain the following result from \Cref{LemmalogA}.

\begin{lemma}\label{eq:thmlog}
Let $a_{0,\beta}=(2-\beta)\alpha\pi(T+\alpha\log{\frac{x}{C}})$ and
$\hat f_{log}(u,z)$ be defined in \cref{eq:extension_of_f} with $f(u,z)$ replaced by $f_{log}(u,z)$ for
$z\in\mathcal{A}^*_{\theta}$.
Then the sum of the discrete Fourier transform decays at an exponential rate
\begin{align}\label{eq:conclusionOfFouriersumlog}
\sum_{n\ne0}\mathfrak{F}[\hat f_{log}]\big{(}\frac{2 n\pi }{h}\big{)}=\mathcal{O}(Te^{-T})+\mathcal{O}\left(\frac{x^\alpha}{e^{\frac{2\pi }{h}a_{0,\beta}}-1}\right)
\end{align}
and
\begin{align}\label{errquadlog}
I_{log}(z)-\widetilde{r}_{N_t}(z)=\mathcal{O}(Te^{-T})+\mathcal{O}\left(\frac{x^\alpha}{e^{\frac{2\pi }{h}a_{0,\beta}}-1}\right),
\end{align}
where all the constants in $\mathcal{O}$ terms are independent of $n$, $T$, $\theta\in [0,\beta]$ and $x\in [x^*,1]$.
\end{lemma}
\begin{proof}
Estimate \cref{eq:conclusionOfFouriersumlog} follows from analogous argument of \Cref{eq:thm} by \Cref{LemmalogA}.
Quadrature error \cref{errquadlog} follows from  \cref{eq:conclusionOfFouriersumlog}
similar to the proof of  \Cref{eq:thm}.
\end{proof}

\section{Proof of Theorems \ref{mainthm} and \ref{mainthm2}}\label{sec:convergence}
\bigskip
To show  \cref{mainthm} and \cref{mainthm2}, we introduce the following estimates.
\begin{lemma}\label{eq:PalyC}
Let $h=\sigma^2\alpha^2=\sigma_{opt}^2\alpha^2/\eta^2$ ($\eta=\frac{\sigma_{opt}}{\sigma}$) and $Q(x)=\frac{x^\alpha}{e^{\frac{2\pi}{h}a_{0,\beta}}-1}$
for $x\in[x^*,1]$ and $a_{0,\beta}=(2-\beta)\alpha\pi(T+\alpha\log{\frac{x}{C}})$. Then it holds uniformly for $T\ge 0$ that  
\begin{align}\label{eq:qerr1}
\left\{\begin{array}{ll}
\frac{1}{e^{\eta^2 T}-1}\le  Q(x)\le \frac{e^{\frac{\sqrt{h}}{2}} e^{-T}}{e^{\eta^2\frac{\sqrt{h}}{2}}-1},&\eta\ge 1\\
\frac{(1-\eta^2)^{1-\frac{1}{\eta^2}}e^{-T}}{\eta^2}\le  Q(x)\le \max\left\{\frac{1}{e^{\eta^2 T}-1},\frac{e^{\frac{\sqrt{h}}{2}} e^{-T}}{e^{\eta^2\frac{\sqrt{h}}{2}}-1}\right\},&\eta<1.\end{array}\right.
\end{align}
\end{lemma}
\begin{proof}
From the definition of $a_{0,\beta}$, $Q(x)$ can be written as $Q(x)=\frac{x^{\alpha}}{(\frac{x}{C})^{\eta^2\alpha} e^{\eta^2 T}-1}$, then it directly follows from the monotonicity of $Q(x)$ by
\begin{align*}
\frac{d}{dx}Q(x)=\frac{(1-\eta^2)\left(\frac{x}{C}\right)^{\eta^2\alpha} e^{\eta^2 T}-1}{\left[(\frac{x}{C})^{\eta^2\alpha} e^{\eta^2 T}-1\right]^2}C^{\alpha-1}\alpha \left(\frac{x}{C}\right)^{\alpha-1}.
\end{align*}
\end{proof}

Hence, from Lemmas \ref{eq:thm}, \ref{eq:thmlog} and \ref{eq:PalyC}, we observe that the quadrature error \cref{errquad} for $I(z)$ (and \cref{errquadlog} for $I_{log}(z)$) is dominated by $\frac{e^{\sqrt{h}/2} e^{-T}}{e^{\eta^2\sqrt{h}/2}-1}$ when $\eta\ge1$, and $\frac{1}{e^{\eta^2 T}-1}$ when $\eta<1$ (by $Te^{-T}$ when $\eta\ge1$, and $\frac{1}{e^{\eta^2 T}-1}$ when $\eta<1$, respectively).
We illustrate the sharpness of these order estimates by {\sc Fig.} \ref{decay-behavior-of-Integral}.

\begin{figure}[htp]
\centerline{\includegraphics[width=12cm]{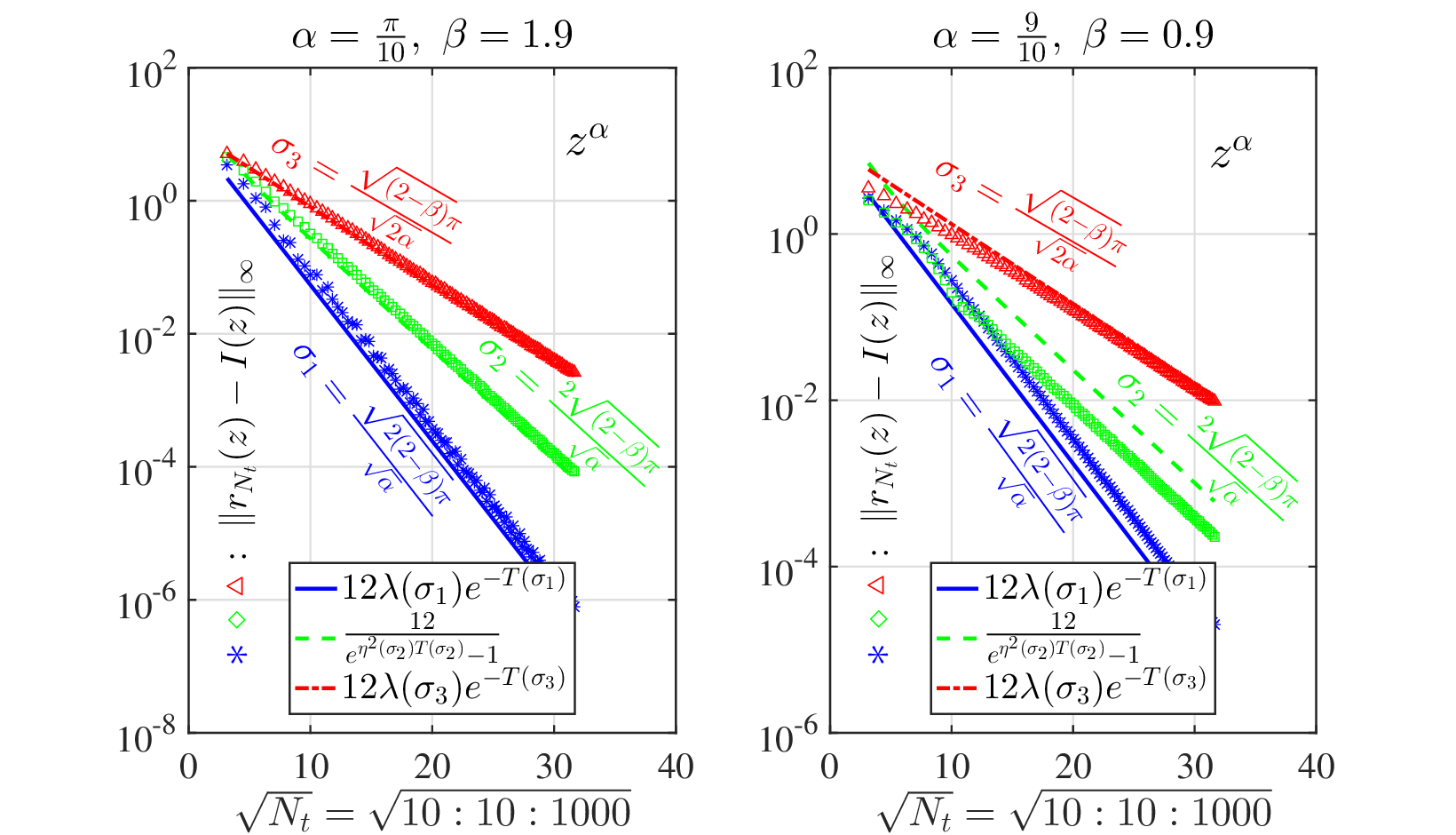}}
\centerline{\includegraphics[width=12cm]{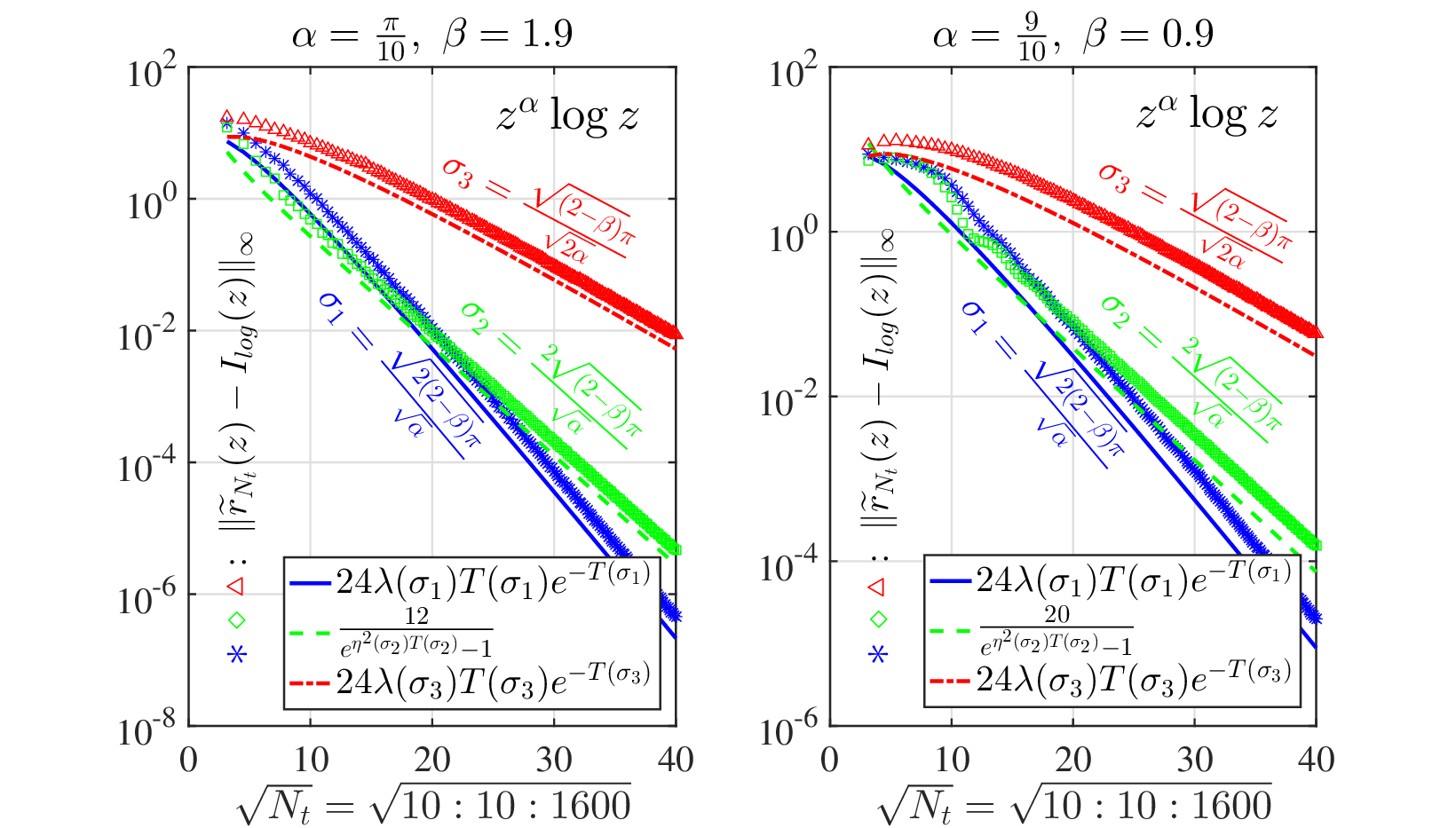}}
\caption{The decay behaviors of the quadrature errors $\|I-r_{N_t}\|_{\infty}$ for $z^{\alpha}$ (first row) and $\|I_{log}-\widetilde{r}_{N_t}\|_{\infty}$ for $z^{\alpha}\log{z}$ (second row), endowed with $T(\sigma_l)=\frac{\alpha\sigma_l\sqrt{N_t}}{\kappa+1}$, $\lambda(\sigma_l)=\frac{e^{\alpha\sigma_l/2}}{e^{\alpha\sigma_l\eta^2(\sigma_l)/2}-1}$ and $\eta(\sigma_l)=\frac{\sigma_{opt}}{\sigma_l}$ with parameters $\sigma_l,\ l=1,2,3$, which are  equivalent to, larger or smaller than the optimal $\sigma_{opt}=\frac{\sqrt{2(2-\beta)}\pi}{\sqrt{\alpha}}$, respectively. The infinite norm $\|\cdot\|_{\infty}$ is evaluated on the sector domain $S_{\beta}$ with $x=1$.}
\label{decay-behavior-of-Integral}
\end{figure}

\bigskip
{\bf Proof of  Theorems \ref{mainthm} and \ref{mainthm2}}: From Lemma \ref{eq:PalyC}, it is easy to verify that
\begin{align*}
\frac{x^\alpha}{e^{\frac{2\pi}{h}a_{0,\beta}}-1}=\left\{\begin{array}{ll}
\mathcal{O}(e^{-T}),& \sigma\le \sigma_{opt}\\
\mathcal{O}(e^{-\eta^2 T}),& \sigma> \sigma_{opt} \end{array}\right.
\end{align*}
uniformly for $z\in S_\beta$ with $x\in [x^*,1]$, which,  together with \Cref{la31} and \cref{errquad},  implies for $T=\sqrt{\frac{N_th}{(1+\kappa)^2}}$ and $N_1={\rm ceil}\big(\frac{N_t}{(\kappa+1)^2}\big)$ that
\begin{align*}
\|I-r_{N_t}\|_{\infty}=\left\{\begin{array}{ll}
\mathcal{O}(e^{-T})=\mathcal{O}\left(e^{-\sigma\alpha\sqrt{N_1}}\right),&\sigma\le \sigma_{opt}\\
\mathcal{O}(e^{-\eta^2 T})=\mathcal{O}\left(e^{-\eta\pi\sqrt{2(2-\beta)\alpha N_1}}\right),\ &\sigma> \sigma_{opt}.
 \end{array}\right. \end{align*}
Noting that
 $$
 \sqrt{N_1}=\sqrt{N-N_2}
 =\sqrt{N}\left(1+\mathcal{O}\left(\frac{N_2}{N}\right)\right)
 =\sqrt{N}+\mathcal{O}(1),
 $$
 it leads to Theorem \ref{mainthm} by \cref{eq:intC1} and  \cref{polynomial app1}.

Analogously, from \Cref{eq:thmlog} and \Cref{la31}  it establishes
 the desired result Theorem \ref{mainthm2}. These complete the proof.

\bigskip

 {\sc Fig}. \ref{rates} illustrates the optimal choices of parameter $\sigma$ and the sharpness of estimated convergence orders.
\begin{figure}[htbp]
\centerline{\includegraphics[height=6cm,width=12cm]{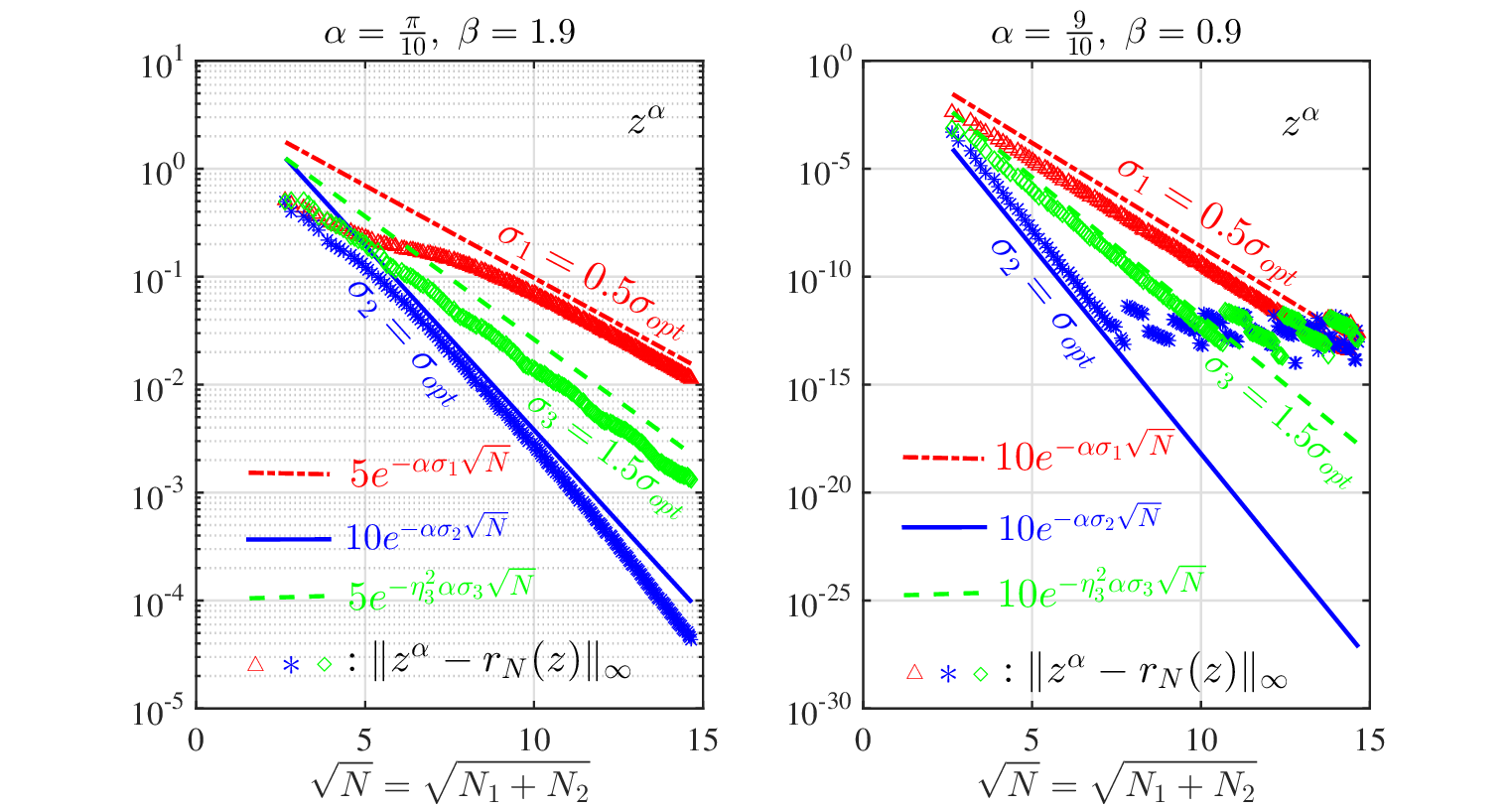}}
\centerline{\includegraphics[height=6cm,width=12cm]{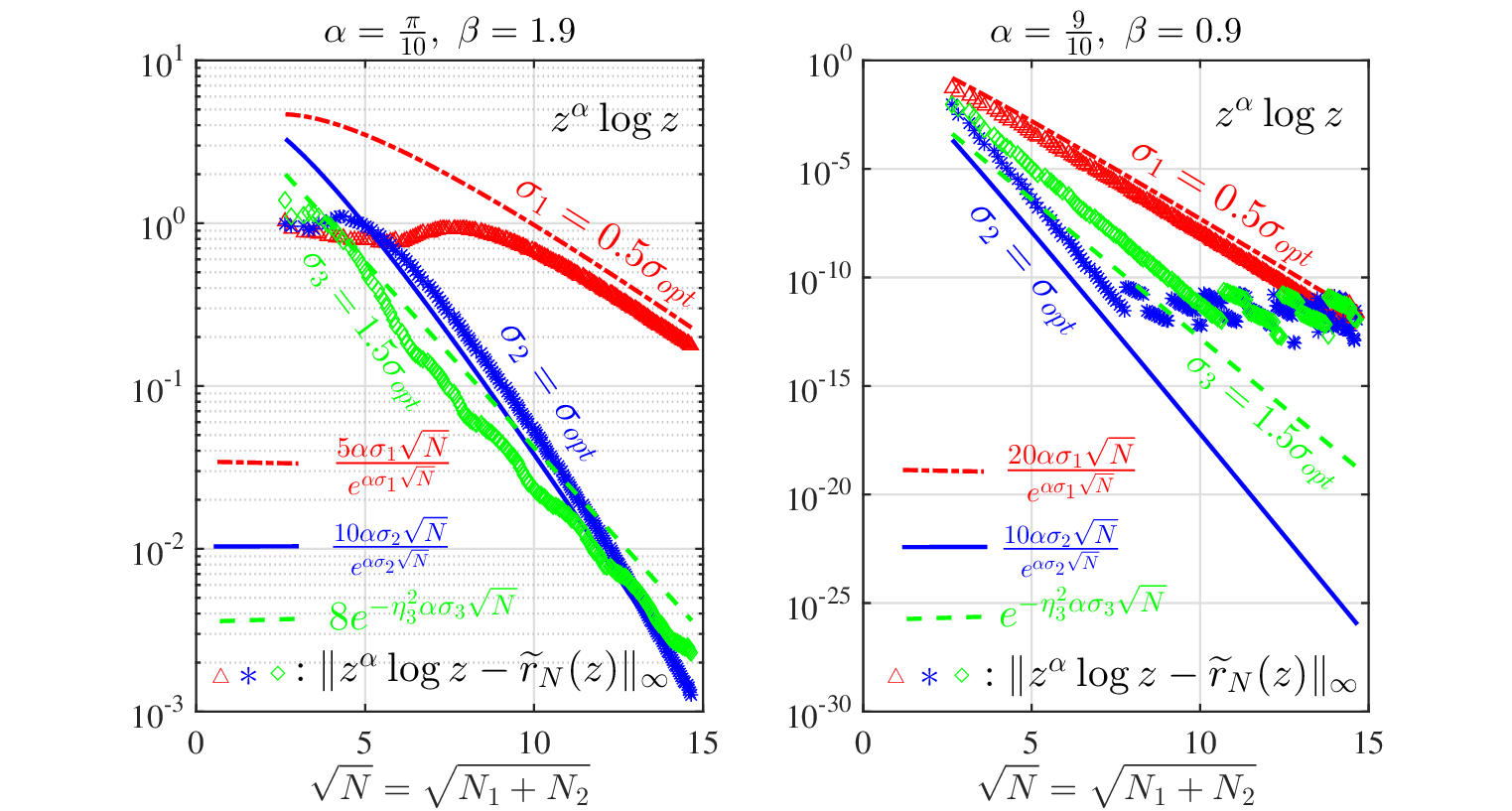}}
\caption{Convergence rates of the LPs for $z^{\alpha}$ and $z^{\alpha}\log{z}$ on $S_\beta$ with various values of $\alpha$, $\beta$ and
$\sigma=\sigma_l,\ l=1,2,3$, where $\sigma_{opt}=\frac{\sqrt{2(2-\beta)}\pi}{\sqrt{\alpha}}$ and $N=N_1+N_2$ with $N_1=4:100$ and $N_2={\rm ceil}(1.3N_1)$.}
\label{rates}
\end{figure}

\section{Approximations on corner domains}\label{app_cornerdomain}
From the decompositions by Cauchy integrals \cite[Theorem 2.3]{Gopal2019}, Theorems \ref{mainthm} and \ref{mainthm2} can be extended to the case in
which the domain $\Omega$ is a polygon (with each internal angle $<2\pi$), validated the presume ``in fact we believe convexity is not necessary'' \cite{Gopal2019}.

For Laplace PDEs, following \cite[Theorem 5]{Wasow},  the link between the type
of the corner singularity $z^\alpha$
and the angle of the corner $\beta$, $\beta\in (0,2)$ on a V-shaped domain is that the dominant asymptotic behaviour near the corner can be
described as $\mathcal{O}(z^{1/\beta})$ for $1/\beta$ non-integer and $\mathcal{O}(z^{1/\beta}\log z)$ for $1/\beta$ integer.
Then, we first give the following lemma for the Cauchy-type integrals.
\begin{lemma}\label{lemma_for_fk1}
Let $\mathcal{W}$ be a positive real number, $\alpha\in(0,1)$ a real number. Then there exist some power series $p_{s,k}(z)$ in $z$ and polynomials $P_{s}(\log{z})$ in $\log{z}$, $s=0,1$, $k=0,1,\cdots$,
such that
\begin{align}\label{fk_z_alpha}
\int_0^\mathcal{W}\frac{\zeta^{k+\alpha}}{\zeta-z}d\zeta
=z^{k+\alpha}P_{0}(\log{z})+p_{0,k}(z)
\end{align}
and
\begin{align}\label{fk_zalpha_log}
\int_0^\mathcal{W}\frac{\zeta^{k+\alpha}\log{\zeta}}{\zeta-z}d\zeta
=z^{k+\alpha}P_{1}(\log{z})+p_{1,k}(z),
\end{align}
where $p_{s,k}(z)$ $(s=0,1)$ converge for $|z|<\mathcal{W}$ and
\begin{align*}
P_{0}(\log{z})=&-\pi\cot{(\alpha\pi)}-i\pi,\\
P_{1}(\log{z})=&-\left[\pi\cot{(\alpha\pi)}
+i\pi\right]\log{z}+\pi^2\csc^2{(\alpha\pi)}.
\end{align*}
\end{lemma}
\begin{proof}
By integrating by parts it follows that
\begin{align}
\int_0^\mathcal{W}\frac{\zeta^{k+\alpha}}{\zeta-z}d\zeta
=&\zeta^{k+\alpha}\log{(\zeta-z)}\bigg|_0^\mathcal{W}
-(k+\alpha)\int_0^\mathcal{W}\zeta^{k+\alpha-1}\log{(\zeta-z)}d\zeta\notag\\
=&\mathcal{W}^{k+\alpha}\log{(\mathcal{W}-z)}
-(k+\alpha)\int_0^\mathcal{W}\zeta^{k+\alpha-1}
\log{\big(1-\frac{z}{\zeta}\big)}d\zeta
\label{first_term_integ}\\
&-(k+\alpha)\int_0^\mathcal{W}\zeta^{k+\alpha-1}\log{\zeta}d\zeta,\notag
\end{align}
and the first two terms in \cref{first_term_integ} can be represented as a series of $z$ convergent for $|z|<\mathcal{W}$ since wherein $\log{(\mathcal{W}-z)}$ is holomorphic.
Then by \cite[Theorems 4.1]{Lehman1954DevelopmentsIT} and \cite[Lemma 1]{Lehman1957} there exist power series $\mathcal{Q}_{s,k}(z)$, which converge for $|z|<\mathcal{W}$, and polynomials $\mathcal{P}_{s,k}(\log{z})
=\sum_{\ell=0}^{s}d_{\ell,k}\left(\log{z}\right)^{s-\ell}$, such that
\begin{align}\label{3rd_term}
\int_0^\mathcal{W}\zeta^{k+\alpha-1}(\log{\zeta})^s
\log{\left(1-\frac{z}{\zeta}\right)}d\zeta
=z^{k+\alpha}\mathcal{P}_{s,k}(\log{z})+\mathcal{Q}_{s,k}(z),
\end{align}
$s=0,1,\ k=0,1,\cdots.$
Additionally, the coefficients of $\mathcal{P}_{s,k}$ can be determined by the linear recurrence relation  \cite[proof of Theorem 4.1]{Lehman1954DevelopmentsIT}
\begin{align*}
\mathcal{P}_{s,k}(\log{z})-e^{-2(k+\alpha)\pi i}\mathcal{P}_{s,k}(\log{z-2\pi i})
=2\pi i\int_0^1\zeta^{k+\alpha-1}\left(\log{\zeta}+\log{z}\right)^sd\zeta.
\end{align*}
Then by setting $s=0$ and $s=1$, respectively,
it follows for
\begin{align}\label{Psk}
\mathcal{P}_{0,k}(\log{z})=d_{0,k},\hspace{.5cm}
\mathcal{P}_{1,k}(\log{z})=d_{0,k}\log{z}+d_{1,k},
\end{align}
that
\begin{align*}
d_{0,k}=&\frac{\pi}{k+\alpha}\frac{2i}{1-e^{-2(k+\alpha)\pi i}}=
\frac{\pi}{k+\alpha}\cot{(\alpha\pi)}+\frac{i\pi}{k+\alpha},\\
d_{1,k}=&-\frac{\pi}{(k+\alpha)^2}\frac{2i}{1-e^{-2(k+\alpha)\pi i}}
-\frac{\pi^2}{k+\alpha}
\left[\frac{2ie^{-(k+\alpha)\pi i}}{1-e^{-2(k+\alpha)\pi i}}\right]^2\\
=&-\frac{d_{0,k}}{k+\alpha}
-\frac{\pi^2}{k+\alpha}\csc^2{(\alpha\pi)}.
\end{align*}
By substituting $\mathcal{P}_{0,k}(\log{z})=d_{0,k}$ and \cref{3rd_term} (the case $s=0$) into \cref{first_term_integ}, we arrive at \cref{fk_z_alpha}.

For \cref{fk_zalpha_log}, in a similar way we have by \cref{3rd_term} that
\begin{align}
\int_0^\mathcal{W}\frac{\zeta^{k+\alpha}\log{\zeta}}{\zeta-z}d\zeta
=&\mathcal{W}^{k+\alpha}\log{\mathcal{W}}\log{(\mathcal{W}-z)}
-\int_0^\mathcal{W}\left[1+(k+\alpha)\log{\zeta}\right]\zeta^{k+\alpha-1}\log{\zeta}
d\zeta\notag\\
&-(k+\alpha)\int_0^\mathcal{W}
\zeta^{k+\alpha-1}\log{\zeta}\log{\left(1-\frac{z}{\zeta}\right)}d\zeta\notag\\
&-\int_0^\mathcal{W}
\zeta^{k+\alpha-1}\log{\left(1-\frac{z}{\zeta}\right)}d\zeta\notag\\
=&\mathcal{W}^{k+\alpha}\log{\mathcal{W}}\log{(\mathcal{W}-z)}
-\int_0^\mathcal{W}\left[1+(k+\alpha)\log{\zeta}\right]\zeta^{k+\alpha-1}\log{\zeta}
d\zeta\label{first_term_integ_log}\\
&-\mathcal{Q}_{0,k}(z)-(k+\alpha)\mathcal{Q}_{1,k}(z)\notag\\
&-z^{k+\alpha}\left[\mathcal{P}_{0,k}(\log{z})
+(k+\alpha)\mathcal{P}_{1,k}(\log{z})\right],\notag
\end{align}
then we complete the proof of \cref{fk_zalpha_log} by substituting \cref{Psk} into \cref{first_term_integ_log} and noticing the fact that both of $\mathcal{Q}_{s,k}(z)$ and $\log{\left(\mathcal{W}-z\right)}$ are holomorphic in $\{z:|z|<\mathcal{W}\}$.
\end{proof}

\begin{remark}
We use the special cases $s=0,1$ of \cite[Theorem 4.1]{Lehman1954DevelopmentsIT} and \cite[Lemma 1]{Lehman1957} for the proof of \cref{lemma_for_fk1}, wherein we found that the additional statement $\mathcal{P}_{s,k}(0)=0$ at the end of \cite[Theorem 4.1]{Lehman1954DevelopmentsIT} may be incorrect. However, this supplementary statement has been removed in the author's subsequent article \cite[Lemma 1]{Lehman1957}.
\end{remark}

By noticing the analyticity of $p_{s,k}(z)$, we have by the Weierstrass Theorem \cite[Theorem 4.1.10, Corollary 4.1.13]{Asmar2018ComplexAW} that
\begin{corollary}\label{coro:dec_sing}
Let the conditions of \cref{lemma_for_fk1} hold and $g(z)$ be holomorphic on
$\left\{z:|z|\le\mathcal{W}\right\}$.
Then there exist two functions $\mathcal{H}_0(z)$
and $\mathcal{H}_1(z)$ holomorphic for $|z|\le\mathcal{W}$, such that
\begin{align}
\int_0^{\mathcal{W}}\frac{g(\zeta)\zeta^{\alpha}}{\zeta-z}d\zeta
=&z^{\alpha}g(z)P_0(\log{z})
+\mathcal{H}_0(z),\label{eq:decom_za_h}\\
\int_0^{\mathcal{W}}\frac{g(\zeta)\zeta^{\alpha}\log{\zeta}}{\zeta-z}d\zeta
=&z^{\alpha}g(z)P_1(\log{z})+\mathcal{H}_1(z).\label{eq:decom_za_log_h}
\end{align}
\end{corollary}
\begin{proof}
From \cref{lemma_for_fk1}, we have that
\begin{align}
p_{0,k}(z)=&\int_0^\mathcal{W}\frac{\zeta^{k+\alpha}}{\zeta-z}d\zeta
-z^{k+\alpha}P_{0}(\log{z}),\label{psk}\\
p_{1,k}(z)=&\int_0^\mathcal{W}\frac{\zeta^{k+\alpha}\log{\zeta}}{\zeta-z}d\zeta
-z^{k+\alpha}P_{1}(\log{z})\label{psklog}
\end{align}
are holomorphic for $|z|<\mathcal{W}$. By the analyticity of $g(z)$, then it can be expressed as $g(z)=\sum_{k=0}^{\infty}\ell_kz^k$ uniformly convergent for $|z|\le\mathcal{W}$, which implies by \cref{fk_z_alpha} and \cref{fk_zalpha_log} that
\begin{align*}
\int_0^{\mathcal{W}}\frac{g(\zeta)\zeta^{\alpha}}{\zeta-z}d\zeta
&=\sum_{k=0}^{\infty}\ell_k\int_0^{\mathcal{W}}
\frac{\zeta^{k+\alpha}}{\zeta-z}d\zeta
=z^{\alpha}g(z)P_0(\log{z})
+\sum_{k=0}^{\infty}\ell_kp_{0,k}(z)
\end{align*}
and
\begin{align*}
\int_0^{\mathcal{W}}\frac{g(\zeta)\zeta^{\alpha}\log{\zeta}}{\zeta-z}d\zeta
&=\sum_{k=0}^{\infty}\ell_k\int_0^{\mathcal{W}}
\frac{\zeta^{k+\alpha}\log{\zeta}}{\zeta-z}d\zeta
=z^{\alpha}g(z)P_1(\log{z})
+\sum_{k=0}^{\infty}\ell_kp_{1,k}(z).
\end{align*}

We observe from \cref{psk} and \cref{psklog} that both of $\sum_{k=0}^{\infty}\ell_kp_{0,k}(z)$ and $\sum_{k=0}^{\infty}\ell_kp_{1,k}(z)$ uniformly converge for $|z|\le \mathcal{W}$, and with the help of Weierstrass Theorem \cite[Theorem 4.1.10, Corollary 4.1.13]{Asmar2018ComplexAW} it follows that, both of $$\mathcal{H}_0(z):=\sum_{k=0}^{\infty}\ell_kp_{0,k}(z)
\hspace{.3cm} \mbox{ and } \hspace{.3cm} \mathcal{H}_1(z):=\sum_{k=0}^{\infty}\ell_kp_{1,k}(z)$$
are holomorphic for $|z|<\mathcal{W}$. Thus, we complete the proof of \cref{eq:decom_za_h} and \cref{eq:decom_za_log_h}.
\end{proof}

%

\bigskip
We sketch
\textbf{the proof of \Cref{corner}} as follows.
\begin{proof}
From the proof of \cite[Theorem 2.3]{Gopal2019}, $f(z)$ can be written as a sum of $2m$ Cauchy-type integrals
\begin{align*}
f(z)=\frac{1}{2\pi i}\sum_{k=1}^m\int_{\Lambda_k}\frac{f(\zeta)}{\zeta-z}d\zeta
+\frac{1}{2\pi i}\sum_{k=1}^m\int_{\Gamma_k}\frac{f(\zeta)}{\zeta-z}d\zeta
=:\sum_{k=1}^mf_k(z)+\sum_{k=1}^mg_k(z),
\end{align*}
where  $\Lambda_k$ consists of the two sides of an exterior bisector at $w_k$, and $\Gamma_k$ connects the end of the slit contour at vertex $k$ to the beginning of the slit contour at vertex $k+1$ (denote $w_{m+1}=w_1$), see {\sc Fig.} \ref{decompose_path}, for example. Thus, every $g_k$ is holomorphic on a larger domain including $\Omega$, and $f_k$ holomorphic on a slit-disk region around $w_k$ with the slit line tilted and translated to lie around $\Lambda_k$, $k=1,\cdots, m$.

\begin{figure}[htbp]
\centerline{\includegraphics[width=8cm]{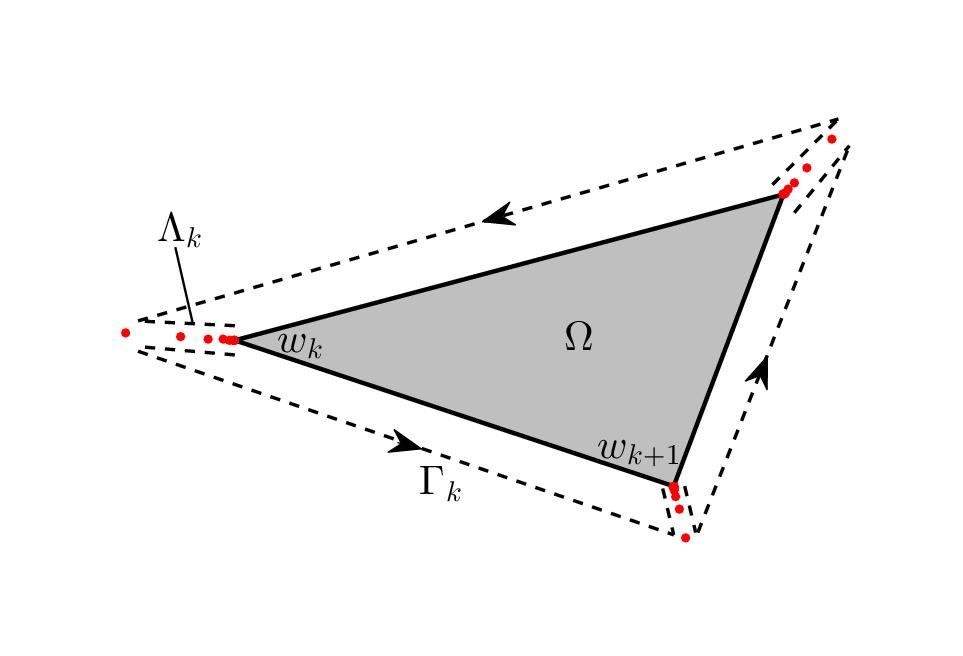}}
\caption{ \cite[{\sc Fig. 3}]{Gopal2019} A holomorphic function $f(z)$ defined on the corner domain $\Omega$ is decomposed as the sum of $2m$ Cauchy-type integrals: $\sum_{k=1}^mf_k(z)+\sum_{k=1}^mg_k(z)$, with $f_k(z)=\frac{1}{2\pi i}\int_{\Lambda_k}\frac{f(\zeta)}{\zeta-z}d\zeta$ along the two sides of an exterior bisector slit to each corner, and $g_k(z)=\frac{1}{2\pi i}\int_{\Gamma_k}\frac{f(\zeta)}{\zeta-z}d\zeta$ along each line segment connecting the ends of those slit contours.}
\label{decompose_path}
\end{figure}

Actually, both of $g_k$ and $f_k$ are holomorphic in $\mathbb{C}\setminus\Gamma_k$ and  $\mathbb{C}\setminus\Lambda_k$, respectively, according to the property of Cauchy-type integral \cite[Theorem 3.8.5]{1984Complex}. Hence, by the proof of Runge's theorem \cite[pp. 76-77]{Gaier1987} the summation $\sum_{k=1}^mg_k(z)$ can be approximated by a polynomial $\mathcal{T}(z)$ of degree of order $\mathcal{O}(\sqrt{N})$ on $\Omega$ with exponential convergence analogous to \cite{Gopal2019}.

Furthermore, from \cref{coro:dec_sing} it follows that $f_k(z)$ in a neighborhood of $w_k$ can be represented as $(z-\omega_k)^{\alpha_k}h_k(z)$ plus a holomorphic function.
Then from \cref{mainthm}, $f_k(z)$ can be approximated by LP $r^{(k)}_{N}(z)$ with a root-exponential rate
\begin{equation}
|r^{(k)}_N(z)-f_k(z)|=\left\{\begin{array}{ll}
\mathcal{O}(e^{-\sigma\alpha_k\sqrt{N}}),&\sigma\le \sigma^{(k)}_{opt},\\
\mathcal{O}(e^{-\pi\eta_k\sqrt{2(2-\beta_k)N\alpha_k}}),&\sigma> \sigma^{(k)}_{opt},
\end{array}\right.\quad \eta_k:=\frac{\sigma^{(k)}_{opt}}{\sigma}
\end{equation}
bounded by $\mathcal{O}(e^{-\sigma\alpha\sqrt{N}})
=\mathcal{O}(e^{-\pi\sqrt{2(2-\beta)N\alpha}})$ for $k=1,2,\cdots,m$,
where we used the fact that $\sigma\alpha\le \sigma \alpha_k$ and $\eta_k\sqrt{2(2-\beta_k)N\alpha_k}=\sqrt{\frac{2(2-\beta_k)^2}{2-\beta}N\alpha}\ge \sqrt{2(2-\beta)N\alpha}$, and $r^{(k)}_N(z)$ is taken to have finite poles exponentially clustered along the exterior bisectors at the corners, with the number of poles near each $w_k$ grows at least in proportion to $N$ that approaches to $\infty$.
Then we establish \cref{eq: corrate1} and complete the proof of \cref{corner}.
\end{proof}

\begin{remark}
From the proof of \cref{corner}, it is evident that one can approximate $f_k(z)$ using the LP $r_{N}^{(k)}(z)$ with $\sigma_{k}=\frac{\sqrt{2(2-\beta_k)}}{\sqrt{\alpha_k}}$. From \cref{mainthm}, it achieves
\begin{align}\label{eq:rcorner}
|r_n(z)-f(z)|=\mathcal{O}(e^{-\min_{1\le k\le m}\pi\sqrt{2(2-\beta_k)N\alpha_k}}).
\end{align}
In this case, the tapered exponential clustering of the poles at each $w_k$ with different $\sigma_k$.
\end{remark}

\begin{remark}
Moreover, from
Wasow \cite[Theorem 5]{Wasow} we see that  in a corner domain,  the dominant asymptotic behavior near the corner $w_k$ with  interior angle $\beta_k\pi$ ($\beta_k\in (0,2)$) can be described as $\mathcal{O}(z^{1/\beta_k})$ for non-integer values of $1/\beta_k$, while $\mathcal{O}(z^{1/\beta_k}\log z)$ for  $1/\beta_k$ being an integer. Then we can choose $\sigma=\frac{\sqrt{2(2-\beta)}\pi}{\sqrt{\alpha}}=\sqrt{2(2-\beta_{k_0})\beta_{k_0}}\pi$ in \cref{corner} and obtain the same rate as \cref{eq:rcorner}
$$
|r_n(z)-f(z)|=\mathcal{O}\left(e^{-\pi{\sqrt{\frac{2(2-\beta_{k_0})}{\beta_{k_0}}N}}}\right)
$$
for every $1/\beta_k$ non-integer, where $\beta_{k_0}=\max_{1\le k\le m}\beta_k$.

 If some values of  $1/\beta_k$  are integers, the rate of $|r_n(z)-f(z)|$ can be readily obtained from \cref{mainthm} and \cref{mainthm2}.

\end{remark}


{\sc Fig.}\ref{Lalace_0} illustrates the robustness of the LP on the corner domain $\Omega$ (the concave quadrilateral domain in {\sc Fig.} \ref{general_domain}) by a Laplace equation $\Delta u=0$ with Dirichlet condition $u=[\Re(z)]^2,\ z\in\partial\Omega$, which is the default boundary condition of the Matlab function \texttt{laplace} in \cite{Gopal2019}.
{\sc Fig.} \ref{Lalace_0} also shows that the globally optimal value $\sigma_{opt}=\sqrt{2(2-\beta)\beta}\pi$ ($\beta=\beta_3$ w.r.t. the largest interior angle) is slightly more efficient than $4$, which is often employed in the previous practical computations \cite{Brubeck2022,Gopal20191,Trefethen2021}.

Inspired by the weaker singularity at the corners with smaller internal angles, we often reduce appropriately the clustering poles there, with little effect on the rate of convergence (see {\sc Fig.} \ref{Lalace_1}). Additionally, an experiment of the same boundary value problem on the curvy L-shaped domain are also illustrated in {\sc Fig}. \ref{LalaceCL}, wherein the three decay behaviors exhibit very small differences, as their corresponding clustering parameters do not vary significantly ($\sigma=4$ and $\sigma_{opt}\approx4.30$). We also see that all of these domains are included in some sector domains centred at every vertex $w_k$ with interior angle $\beta_k\pi,\ 0\le\beta_k\le\beta\le2$.

\begin{figure}[htbp]
\centerline{\includegraphics[width=16cm]{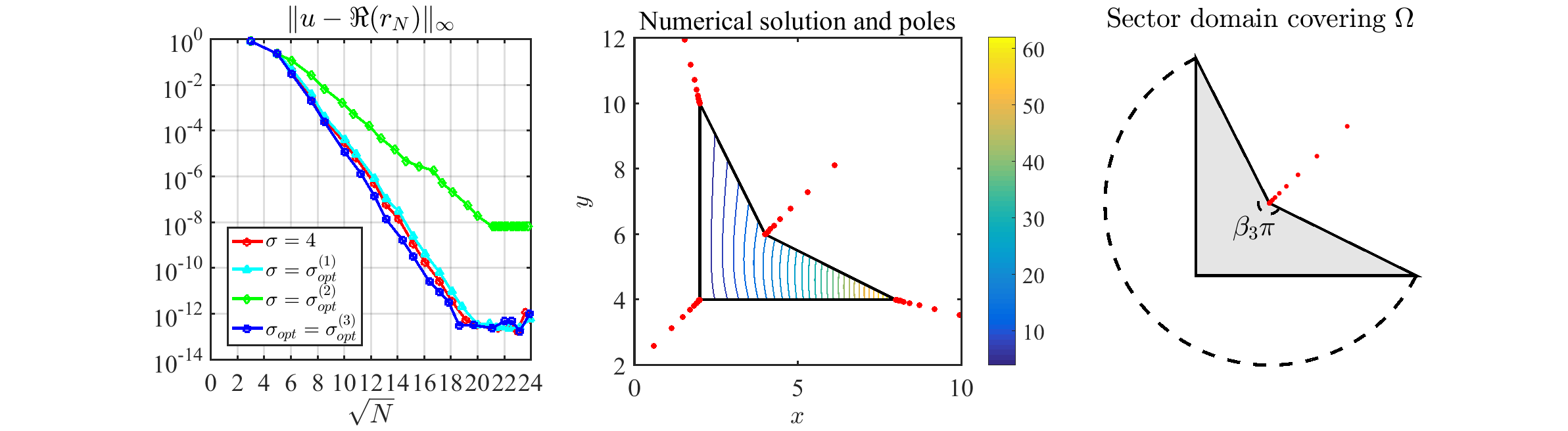}}
\caption{The decay rates (left) of errors of the numerical solutions for the Laplace equation on the concave quadrilateral domain $\Omega$ (whose vertices are: $w_1=2+4i,\ w_2=8+4i,\ w_3=4+6i,\ w_4=2+10i$) with various values of $\sigma$: $\sigma^{(k)}_{opt}=\sqrt{2(2-\beta_k)\beta_k}\pi$, which corresponds to $w_1,\ w_2,\ w_3$. Additionally, $\sigma_{opt}=\sigma^{(3)}_{opt}$ is the globally optimal clustering parameter, and $\sigma=4$ is often employed in the previous practical computations. The second subplot displays the contour plot of numerical solution and the distribution of clustering poles (red points) with respect to $\sigma_{opt}$. Obviously, the domain $\Omega$ is covered by a sector domain centred at $w_3$ with the largest interior angel $\beta_3\pi$, see the third subplot.}
\label{Lalace_0}
\end{figure}

\begin{figure}[htbp]
\centerline{\includegraphics[width=12cm]{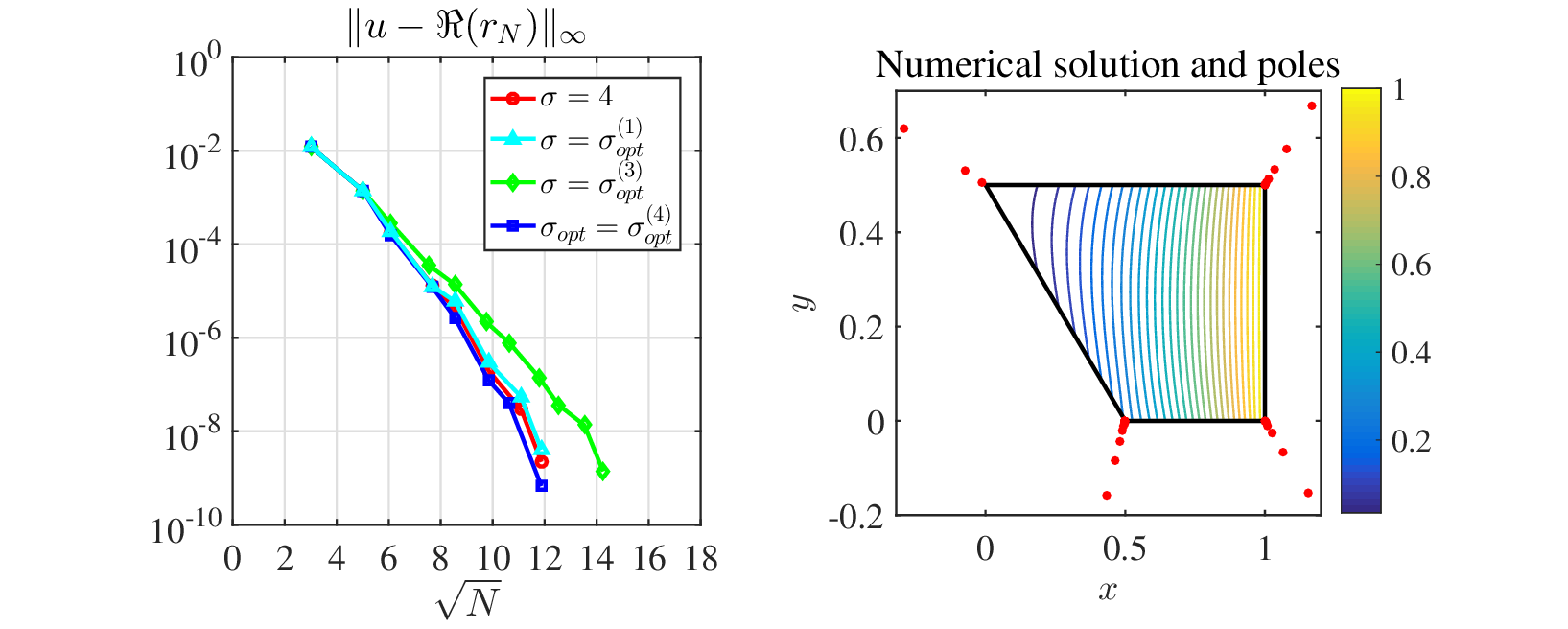}}
\caption{The decay rates (left) of errors of the numerical solutions for the Laplace equation on a quadrilateral domain $\Omega$ with vertices $w_1=1,\ w_2=1+0.5,\ w_3=0.5i,\ w_4=0.5$ with various values $\sigma_k=\sqrt{2(2-\beta_k)\beta_k}\pi$ for $\sigma$, where $\beta_1=\frac{1}{2},\ \beta_3=\frac{1}{4},\ \beta_4=\frac{3}{4}$ corresponding to $w_1,\ w_3,\ w_4$. The globally optimal clustering parameter is  $\sigma_{opt}=\sigma^{(4)}_{opt}$.  The second subplot displays the contour plot of numerical solution and the distribution of clustering poles (red points) with respect to $\sigma_{opt}$, whose clustering density is reduced appropriately at the corners with smaller internal angles.}
\label{Lalace_1}
\end{figure}

\begin{figure}[htbp]
\centerline{\includegraphics[width=16cm]{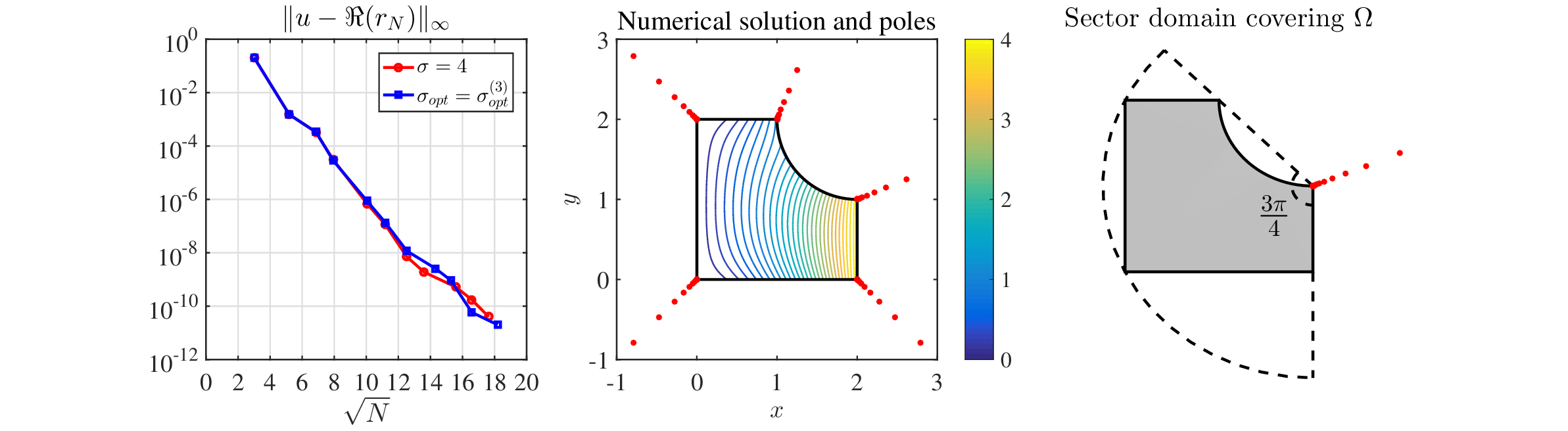}}
\caption{The decay rate (left) of errors of the numerical solution for the Laplace equation on the curvy L-shaped domain $\Omega$ determined by the vertices $w_1=0,\ w_2=2,\ w_3=2+i,\ w_4=1+2i,\ w_5=2i$. We choose the clustering parameter $\sigma=4$ and the globally optimal $\sigma_{opt}=\sigma^{(3)}_{opt}=\sqrt{2(2-\beta_3)\beta_3}\pi$, $\beta_3=\frac{3}{4}$ corresponding to $w_3$. The contour plot of numerical solution and distribution of clustering poles (red points) with respect to $\sigma_{opt}$ also are displayed in the second subplot. Here the domain $\Omega$ also can be covered by a sector domain centred at $w_3$ with the largest interior angel $\frac{3\pi}{4}$.}
\label{LalaceCL}
\end{figure}


\section{Conclusions}
\label{conclusion}
Utilizing  Poisson summation formula, Runge's approximation theorem and with the aid of Cauchy's integral theorem,  this paper rigorously proves  the proposed conjecture of the lightning $+$ polynomial rational approximation in a V-shaped domain and extends to general algebraical and logarithmatic singularities.
In addition, from Lehman and Wasow's contributions to the study of  corner singularities of solutions of
partial differential equations \cite{{Lehman1954DevelopmentsIT}, Wasow}, together with the decomposition of Gopal and Trefethen \cite{Gopal2019}, it leads to theoretical analysis of a root-exponential rate for efficient lightning plus polynomial schemes in corner domains \cite{Herremans2023}.
\begin{appendix}
\section{Proofs of \cref{la3}, \cref{lemma_inte_circ}, \cref{infty} and \cref{LemmalogA}}\label{AppendixA}
We present the lemmas for the case $n\ge1$, and those for $n\le-1$ can be proven in the same approach. Here we are concerned with the uniform bounds independent of $x\in [x^*,1]$ and $\theta\in [0,\beta]$ in a V-shaped domain different from \cite{XY2023}. Again, to avoid repetition, only the case $z=z^+=xe^{\frac{\theta\pi}{2}i}$ is proved here, and the other case $z=z^-=xe^{-\frac{\theta\pi}{2}i}$ can be checked in the exactly same manner.
\begin{lemma}\label{la3}
Let $A_0=a_0+a$ and $A_1=a_1+a$. Then
\begin{align*}
&\bigg{|}\int_{0}^{A_1}f(h+it,z)
e^{-\frac{2n\pi}{h}t}dt-\int_{0}^{A_0}f(h-it,z)
e^{-\frac{2n\pi}{h}t}dt\bigg{|}\\
=&\mathcal{O}(e^{-T})
\int_0^{+\infty}te^{\sqrt{t}}e^{-\frac{2n\pi}{h}t}dt
\end{align*}
holds uniformly and independent of $\theta\in [0,\beta]$ and $x\in [x^*,1]$ for $z\in \mathcal{A}^*_\theta$.
\end{lemma}
\begin{proof}
At first, we have
\begin{align*}
&f(h+it,z)-f(h-it,z)\\
=&e^{-T}\frac{\sin(\alpha\pi)}{2\alpha\pi}
\bigg{[}\frac{1}{\sqrt{h+it}}
\frac{zC^{\alpha}e^{\sqrt{h+it}}}{Ce^{\frac{1}{\alpha}(\sqrt{h+it}-T)}+z}
-\frac{1}{\sqrt{h-it}}
\frac{zC^{\alpha}e^{\sqrt{h-it}}}{Ce^{\frac{1}{\alpha}(\sqrt{h-it}-T)}+z}\bigg{]}\notag.
\end{align*}
Define
$$\phi(t,z)=\frac{1}{\sqrt{h+it}}
\frac{ze^{\sqrt{h+it}}}{Ce^{\frac{1}{\alpha}(\sqrt{h+it}-T)}+z},\quad t\in[-1,1],
$$
then $\phi(t,z)$ is analytic for $t\in [-1,1]$ and  $\partial_t\phi$ is continuous on $[-1,1]\times S_\beta$, and from \cite{Mcleod1965} it obtains
\begin{align*}
|\phi(t,z)-\phi(-t,z)|\le 2\|\partial_t\phi\|_{\infty}t,\quad t\in [-1,1]
\end{align*}
and
\begin{align*}
\bigg{|}\int_{0}^{1}(f(h+it,z)-(f(h-it,z))
e^{-\frac{2n\pi}{h}t}dt\bigg{|}\le e^{-T}\frac{C^\alpha\|\partial_t\phi\|_{\infty}\sin(\alpha\pi)}{\alpha\pi}\int_0^1te^{-\frac{2n\pi}{h}t}dt.
\end{align*}

Let
\begin{align*}
\varphi(t,z)=&\frac{z}
{Ce^{\frac{1}{\alpha}(\sqrt{h-it}-T)}+z},\quad t\in[1,A_0].
\end{align*}
Since $z\in\mathcal{A}^*_\theta$, from $\sqrt{M_0h}
\ge \big(\sqrt{(4+\beta)\alpha\pi/2}+\sqrt[4]{4h}\big)^2\ge 2\big(\sqrt{\alpha\pi}+\sqrt[4]{h}\big)^2$ it follows
\begin{align*}
\sqrt{\alpha\log{\frac{x}{C}}+T}=&\sqrt[4]{v_0+(2-\theta)^2\alpha^2\pi^2/4}
\ge\sqrt[4]{M_0h}
\ge\sqrt{2\alpha\pi}+\sqrt[4]{4h}\\
\ge&\frac{\sqrt{2\alpha\pi}+\sqrt{2\alpha\pi+8\sqrt{h}}}{2}
\end{align*}
and
\begin{align}\label{align}
\left(\sqrt{\alpha\log{\frac{x}{C}}+T}-\sqrt{\alpha\pi/2}\right)^2\ge \alpha\pi/2+2\sqrt{h}.
\end{align}

In addition, set $u=v+iw=re^{i\Theta}$ and $r=\sqrt{v^2+w^2}$. We have by $\cos\Theta=\frac{v}{r}$  and the half angle formula that
$$
\Re(\sqrt{u})=\sqrt{\frac{\sqrt{v^2+w^2}+v}{2}}.
$$
Then, together with $\Re\big(\sqrt{u_0}\big)=T+\alpha\log{\frac{x}{C}}\ge \frac{\sqrt{h}}{2}$ \cref{eq:realp}
and \cref{align}, we obtain  that
\begin{align*}
&\Re(\sqrt{h-it}-\sqrt{u_0})
=\sqrt{\frac{\sqrt{h^2+t^2}+h}{2}}-\Re(\sqrt{u_0})
\le\sqrt{\frac{\sqrt{h^2+A_0^2}+h}{2}}-\Re(\sqrt{u_0})\\
\le&\sqrt{A_0/2}+\sqrt{h}-\Re(\sqrt{u_0})
=\sqrt{(4-\theta)\alpha\pi/2}
\sqrt{\alpha\log{\frac{x}{C}}+T}+\sqrt{h}-\Re(\sqrt{u_0})\\
\le&\sqrt{2\alpha\pi}
\sqrt{\alpha\log{\frac{x}{C}}+T}+\sqrt{h}-(\alpha\log{\frac{x}{C}}+T)
\le-\sqrt{h}
=-\alpha\sigma,
\end{align*}
which yields
$\big|e^{\frac{1}{\alpha}(\sqrt{h-it}-\sqrt{u_0})}\big|\le e^{-\sigma}$ and
\begin{align}\label{eq:dom1}
\big{|}\varphi(t,z)\big{|}
=\frac{|z|}{|z+Ce^{\frac{1}{\alpha}(\sqrt{h-it}-T)}|}=\frac{1}{|e^{\frac{1}{\alpha}(\sqrt{h-it}-\sqrt{u_0})}-1|}
\le\frac{1}{1-e^{-\sigma}}.
\end{align}

Analogously, we have for $t\in [1,A_1]$ that
\begin{align}\label{eq:dom2}
\frac{|z|}{|z+Ce^{\frac{1}{\alpha}(\sqrt{h+it}-T)}|}=\frac{1}{|e^{\frac{1}{\alpha}(\sqrt{h+it}-\sqrt{u_0})}-1|}
\le\frac{1}{1-e^{-\sigma}}
\end{align}
and then for $t\in [0,A_0]$ or $t\in [0,A_1]$ respectively
\begin{align*}
\big{|}f(h\pm it,z)\big{|}\le e^{-T}\frac{C^\alpha\sin(\alpha\pi)}{2\alpha\pi}\frac{e^{\sqrt{t}+\sqrt{h}}}{\sqrt{t}(1-e^{-\sigma})}.
\end{align*}
Consequently, we get
\begin{align*}
&\bigg{|}\int_{0}^{A_1}f(h+it,z)
e^{-\frac{2n\pi}{h}t}dt-\int_{0}^{A_0}f(h-it,z)
e^{-\frac{2n\pi}{h}t}dt\bigg{|}\\
\le&\int_{0}^1\left|f(h+it,z)-f(h-it,z)\right|e^{-\frac{2n\pi}{h}t}dt\\
&+\int_{1}^{A_1}|f(h+it,z)|e^{-\frac{2n\pi}{h}t}dt
+\int_{1}^{A_0}|f(h-it,z)|e^{-\frac{2n\pi}{h}t}dt
\\
=&\mathcal{O}(e^{-T})\bigg[
\int_0^1te^{-\frac{2n\pi}{h}t}dt
+\left(\int_1^{A_1}+\int_1^{A_0}\right)\frac{e^{\sqrt{t}+\sqrt{h}}}{\sqrt{t}}
e^{-\frac{2n\pi}{h}t}dt\bigg]\\
=&\mathcal{O}(e^{-T})
\int_0^{+\infty}te^{\sqrt{t}}e^{-\frac{2n\pi}{h}t}dt
\end{align*}
independent of $\theta\in [0,\beta]$ and $x\in [x^*,1]$ for $z\in \mathcal{A}^*_\theta$, which leads to the desired result.
\end{proof}

\bigskip
\begin{lemma}\label{lemma_inte_circ}
Let $f(u,z)$ be defined in \cref{eq:func} with $z\in\mathcal{A}^*_\theta$. Suppose for some fixed sufficiently large  $N_0$ independent of $z$ and $T$ that
$$
C^{-}_{\rho}=\{z=u_0+\rho e^{i\vartheta},\vartheta:0\rightarrow-2\pi\},\ \
C^{+}_{\rho}=\{z=u_1+\rho e^{i\vartheta},\vartheta:0\rightarrow2\pi\}
$$
with $0<\rho=\frac{1}{2}\min\left\{\alpha^2\pi^2,a_{0,\beta},\frac{1}{N_0}\right\}$,
then
\begin{align}
\bigg{|}\int_{C^{-}_{\rho}}
f(u,z)e^{-i\frac{2n\pi}{h}u}du\bigg{|}
=&e^{-\frac{2n\pi}{h}a_0}x^{\alpha}\mathcal{O}(1),\label{55555}\\
\bigg{|}\int_{C^{+}_{\rho}}
f(u,z)e^{i\frac{2n\pi}{h}u}du\bigg{|}
=&e^{-\frac{2n\pi}{h}a_1}x^{\alpha}\mathcal{O}(1)\label{66666}
\end{align}
hold for all $T$ and the constants in terms $\mathcal{O}$ are independent of $\theta\in [0,\beta]$ and $x\in [x^*,1]$ for $z\in \mathcal{A}^*_\theta$.
\end{lemma}
\begin{proof}
Without loss of generality, we consider the case of $u\in C^{-}_{\rho}$, and the same argument can be developed for case $u\in C^{+}_{\rho}$.

We denote $u=u_0+\rho e^{i\vartheta}=v_0+\rho\cos\vartheta+i(\rho\sin\vartheta-a_0)$, and the integral in \cref{55555} can be rewritten as
\begin{align}\label{integral_bound_for_cir}
&\bigg{|}\int_{C^{-}_{\rho}}
f(u,z)e^{-i\frac{2n\pi}{h}u}du\bigg{|}\\
\le&e^{-\frac{2n\pi}{h}a_0}\int_{0}^{2\pi}
\big|f(u_0+\rho e^{i\vartheta},z)e^{-i\frac{2n\pi}{h}(v_0+\rho\cos{\vartheta})}
e^{\rho\sin{\vartheta}}ie^{i\vartheta}\big|\rho d\vartheta\notag\\
\le&e^{\rho-\frac{2n\pi}{h}a_0}\int_{0}^{2\pi}
\big|\rho e^{i\vartheta}f(u_0+\rho e^{i\vartheta},z)\big|d\vartheta\notag\\
\le&\frac{C^{\alpha}e^{\rho-\frac{2n\pi}{h}a_0}}{2}\int_0^{2\pi}\bigg{|}\frac{e^{\sqrt{u_0+\rho e^{i\vartheta}}-T}}{\sqrt{u_0+\rho e^{i\vartheta}}}\bigg{|}
\bigg{|}\frac{\rho e^{i\vartheta}}{e^{\frac{1}{\alpha}(\sqrt{u_0+\rho e^{i\vartheta}}-\sqrt{u_0})}-1}\bigg{|}d\vartheta\notag\\
=&\frac{\alpha C^{\alpha}e^{\rho-\frac{2n\pi}{h}a_0}}{2}\int_0^{2\pi}\bigg{|}\frac{e^{\sqrt{u_0+\rho e^{i\vartheta}}-T}(\sqrt{u_0+\rho e^{i\vartheta}}+\sqrt{u_0})}{\sqrt{u_0+\rho e^{i\vartheta}}}\bigg{|}
\bigg{|}\frac{\frac{\rho e^{i\vartheta}/\alpha}{\sqrt{u_0+\rho e^{i\vartheta}}+\sqrt{u_0}}}{e^{\frac{\rho e^{i\vartheta}/\alpha}{\sqrt{u_0+\rho e^{i\vartheta}}+\sqrt{u}}}-1}\bigg{|}d\vartheta\notag.
\end{align}

By using
$\frac{e^{\zeta}-1}{\zeta}=1+o(|\zeta|)$ as $\zeta\rightarrow0$,
we bound the last term in the integrand of the last identity as follows. Note that $\Re(\sqrt{u_0})\ge \frac{\sqrt{h}}{2}$. Then for sufficiently large $N_0$ there is a constant $C_0$ independent of $\theta$, such that
\begin{align}\label{eq:bound_circ1}
\bigg{|}\frac{\rho e^{i\vartheta}/\alpha}{\sqrt{u_0+\rho e^{i\vartheta}}+\sqrt{u_0}}\bigg{|}\bigg{/}\bigg{|}
e^{\frac{\rho e^{i\vartheta}/\alpha}{\sqrt{u_0+\rho e^{i\vartheta}}+\sqrt{u_0}}}-1\bigg{|}
\le C_0.
\end{align}

Next we estimate $\frac{\sqrt{u_0+\rho e^{i\vartheta}}+\sqrt{u_0}}{\sqrt{u_0+\rho e^{i\vartheta}}}$ from $|u_0|\ge |v_0|\ge M_0h\ge 2$ and $\rho\le\frac{1}{N_0}$:
$$
\bigg{|}\frac{u_0}{u}\bigg{|}
=\bigg{|}\frac{u_0}{u_0+\rho e^{i\vartheta}}\bigg{|}
\le\frac{|u_0|}{|u_0|-|\rho e^{i\vartheta}|}
\le\frac{|v_0|}{|v_0|-\frac{1}{N_0}}
\le\frac{2}{2-\frac{1}{N_0}}
\le 2
$$
which implies
\begin{align}\label{eq:bound_circ2}
\bigg{|}\frac{\sqrt{u_0+\rho e^{i\vartheta}}+\sqrt{u_0}}{\sqrt{u_0+\rho e^{i\vartheta}}}\bigg{|}\le 1+\sqrt{\bigg{|}\frac{u_0}{u_0+\rho e^{i\vartheta}}\bigg{|}}<3.
\end{align}

Finally, we consider $e^{\sqrt{u_0+\rho e^{i\vartheta}}-T}$: From $v_0>M_0h$ and \cref{eq:real}, and by  $v_0+\rho\ge v$, we have  with $c_0=\frac{1}{4}(2-\theta)\alpha\pi$ and
$a_0=(2-\theta)\pi\alpha\sqrt{v_0+(2-\theta)^2\alpha^2\pi^2/4}$ that
\begin{align*}
2\big{(}\sqrt{v_0+(2-\theta)^2\alpha^2\pi^2/4}+c_0\big{)}^2-v
=&2v_0+(2-\theta)^2\alpha^2\pi^2/2+2c_0^2+\frac{4c_0a_0}{(2-\theta)\alpha\pi}-v\\
\ge& v_0+(2-\theta)^2\alpha^2\pi^2/2+a_0+2c_0^2-\rho.
\end{align*}
Then by
$$
\sqrt{v^2+w^2}\le\sqrt{(v_0+\rho)^2+(a_0+\rho)^2}\le v_0+a_0+2\rho,
$$
we obtain from the definition of $\rho$ for sufficiently large $N_0$ that
$$
\sqrt{v^2+w^2}-
2\big{(}\sqrt{v_0+(2-\theta)^2\alpha^2\pi^2/4}+c_0\big{)}^2+v
\le3\rho-(2-\theta)^2\alpha^2\pi^2/2-2c_0^2\le0.
$$
which deduces
\begin{align*}
\sqrt{\frac{\sqrt{v^2+w^2}+v}{2}}
\le&\sqrt{v_0+(2-\theta)\alpha^2\pi^2/4}
+\frac{1}{4}(2-\theta)\alpha\pi,\\
\Re(\sqrt{u})-T
=&\sqrt{\frac{\sqrt{v^2+w^2}+v}{2}}-T\\
\le&\sqrt{v_0+(2-\theta)^2\alpha^2\pi^2/4}-T+\frac{1}{4}(2-\theta)\alpha\pi\\
=&\alpha\log{\frac{x}{C}}+\frac{1}{4}(2-\theta)\alpha\pi,
\end{align*}
and
\begin{align}\label{eq:bound_circ3}
\big{|}e^{\sqrt{u_0+\rho e^{i\vartheta}}-T}\big{|}=e^{\Re(u)-T}\le\left(\frac{x}{C}\right)^\alpha e^{\frac{1}{4}(2-\theta)\alpha\pi}\le\left(\frac{x}{C}\right)^\alpha e^{2\alpha\pi}.
\end{align}

Substitute \cref{eq:bound_circ1}, \cref{eq:bound_circ2} and \cref{eq:bound_circ3} into \cref{integral_bound_for_cir}, we have that
\begin{align*}
\bigg{|}\int_{C^{-}_{\rho}}
f(u,z)e^{-i\frac{2n\pi}{h}u}du\bigg{|}
=e^{-\frac{2n\pi}{h}a_0}x^\alpha \mathcal{O}(1),
\end{align*}
where the constant $\mathcal{O}(1)$ is independent of $\theta\in [0,\beta]$ and $x\in [x^*,1]$ for $z\in \mathcal{A}^*_\theta$.
\end{proof}

\bigskip
Now we turn to prove the boundedness of
\begin{align*}
\bigg{|}
\int_{h-iA_0}^{+\infty-iA_0}f(u,z)e^{-i\frac{2n\pi}{h}u}du\bigg{|}
=&e^{-\frac{2n\pi}{h} A_0}\bigg{|}\int_{h}^{+\infty}f(t-iA_0,z)
e^{-i\frac{2n\pi}{h}t}dt\bigg{|},\\
\bigg{|}
\int_{h+iA_1}^{+\infty+iA_1}f(u,z)e^{i\frac{2n\pi}{h}u}du\bigg{|}
=&e^{-\frac{2n\pi}{h} A_1}\bigg{|}\int_{h}^{+\infty}f(t+iA_1,z)
e^{i\frac{2n\pi}{h}t}dt\bigg{|}.
\end{align*}
Our strategy is to divide the integral interval $[h,+\infty)$ into three
subintervals
$$[h,v_L],\ [v_L,v_R],\ [v_R,+\infty),$$
on which the integrals of $|f(t-iA_0,z)|$ will be bounded by $x^{\alpha}\mathcal{O}(1)$ independent of $\theta\in [0,\beta]$, where the dividing points satisfies $h<v_L<v_0<v_R<+\infty$. The proof can be directly applied to the integrals of $|f(t+iA_1,z)|$.

We seek two points $u_L=v_L-iA_0$ and $u_R=v_R-iA_0$ locating on the left and right sides of $u_0-ia$, respectively, such that (see {\sc Fig}. \ref{integral_contour})
\begin{align}\label{bounds of Im(t-2ia)}
-\frac{1}{2}(5-\theta)\alpha\pi=\Im(\sqrt{u_L})
\le\Im(\sqrt{u})\le\Im(\sqrt{u_R})
=-\frac{1}{2}(3-\theta)\alpha\pi
\end{align}
for $u=t-2iA_0,\ v_L\le t\le v_R.$

The following observations are much important for the choosing of $u_L$ and $u_R$.
Denote $u=t-iA_0$ with $t\in[h,+\infty)$, then it follows
\begin{align}\label{presention of sqrt u}
\sqrt{u}=\sqrt{\frac{\sqrt{t^2+A_0^2}+t}{2}}-i\sqrt{\frac{\frac{1}{2}A_0^2}{\sqrt{t^2+A_0^2}+t}}
=:\omega-i\varpi_u,
\end{align}
which implies that both of its real part $\Re(\sqrt{u})$ and imaginary part $\Im(\sqrt{u})$ are strictly monotonically increasing with respect to $t\in[h,+\infty)$, $\Re(\sqrt{u})$ is a positive function, and $\Im(\sqrt{u})$ is a negative one. In particular, we see from \cref{eq:all_poles_fux_C} that $\Re(\sqrt{u_0})=\alpha\log{\frac{x}{C}}+T-i(2-\theta)\alpha\pi/2$.

At first, we show that
\begin{align}\label{bounds of Im(u0-ia)}
-\frac{1}{2}(5-\theta)\alpha\pi<\Im\big(\sqrt{v_0-iA_0}\big)
=\Im\big(\sqrt{u_0-ia}\big)<-\frac{1}{2}(3-\theta)\alpha\pi.
\end{align}

Set $\hat{u}=\hat{v}-iA_0$ such that $\Re(\sqrt{\hat{u}})=\Re(\sqrt{u_0})$. By noticing the definitions of $a$ and $a_0$,
it is easy to check that
$\Re(\sqrt{\hat{u}})=\Re(\sqrt{u_0})$ is equivalent to $\varpi_{\hat{u}}=\frac{1}{2}(4-\theta)\alpha\pi$
since $\Re(\sqrt{u_0})=\frac{a_0}{(2-\theta)\alpha\pi}$ and from \eqref{presention of sqrt u} $\Re(\sqrt{\hat{u}})=\frac{A_0}{2\varpi_{\hat{u}}}$.
Then we have
\begin{align}
&\Re(\sqrt{u}-\sqrt{u_0})<0\hspace{.2cm}\text{ for }\hspace{.2cm}u=t-iA_0,\ t\in[h,\hat{v}),\label{monotonicity on h-baru}\\
&\Re(\sqrt{u}-\sqrt{u_0})>0\hspace{.2cm}\text{ for }\hspace{.2cm}u=t-iA_0,\ t\in(\hat{v},+\infty],\label{monotonicity on baru-infty}
\end{align}
which together with
$\Re(\sqrt{\hat{u}})=\Re(\sqrt{u_0})<\Re(\sqrt{v_0-iA_0})$ implies that
\begin{align*}
-\frac{1}{2}(5-\theta)\alpha\pi<-\frac{1}{2}(4-\theta)\alpha\pi=\Im\big(\sqrt{\hat u}\big)<
\Im\big(\sqrt{v_0-iA_0}\big)
\end{align*}
according to the monotonicity of $\Re(\sqrt{u})$ and $\Im(\sqrt{u})$.

By some elementary arithmetic, we can verify by letting $y=\alpha\log{\frac{x}{C}}+T$ that
\begin{align}\label{ratio of a v0}
\frac{A_0}{v_0}=&\frac{4-\theta}{2}\frac{a}{v_0}
=\frac{4-\theta}{2}
\frac{2\alpha\pi(\alpha\log{\frac{x}{C}}+T)}{(\alpha\log{\frac{x}{C}}+T)^2
-\frac{1}{4}(2-\theta)^2\alpha^2\pi^2}\\
=&\frac{4-\theta}{2}\frac{2\alpha\pi y}{y^2-\frac{1}{4}(2-\theta)^2\alpha^2\pi^2}
\le1\notag
\end{align}
holds for
$y\ge(5-\frac{3\theta}{2})\alpha\pi$, and it is sufficient that
$v_0=(\alpha\log{\frac{x}{C}}+T)^2-(2-\beta)^2\alpha^2\pi^2/4>M_0h\ge 24\pi^2\alpha^2\ge2(3-\theta)(4-\theta)\alpha\pi$. From \cref{presention of sqrt u} and \cref{ratio of a v0} it follows that
\begin{align*}
&\frac{(4-\theta)^2}{4}-\frac{\Im^2(\sqrt{v_0-iA_0})}{\alpha^2\pi^2}
=\frac{(4-\theta)^2}{4}\left[1-
\frac{\sqrt{v_0^2+a^2}+v_0}{\sqrt{v_0^2+A_0^2}+v_0}\right]\\
=&\frac{(4-\theta)^2}{4}\frac{\sqrt{v_0^2+A_0^2}
-\sqrt{v_0^2+a^2}}{\sqrt{v_0^2+A_0^2}+v_0}\hspace{1.5cm}(>0)\\
=&\frac{(4-\theta)^2\left[A_0^2-a^2\right]}
{4\left[\sqrt{v_0^2+A_0^2}+v_0\right]
\left[\sqrt{v_0^2+A_0^2}+\sqrt{v_0^2+a^2}\right]}
\hspace{.2cm}\left(\frac{A_0}{v_0}\le1\right)\\
<&\frac{(4-\theta)^2}{4}\frac{\left[\frac{(4-\theta)^2}{4}-1\right]a^2}
{2(1+\sqrt{2})\frac{(4-\theta)^2}{4}a^2}
\le\frac{\frac{(4-\theta)^2}{4}-1}{4},\hspace{.5cm}(<\frac{3}{4})
\end{align*}
and then
\begin{align*}
\frac{\Im^2(\sqrt{v_0-iA_0})}{\alpha^2\pi^2}
>\frac{1}{4}\left[1+\frac{3}{4}(4-\theta)^2\right]\ge\frac{(3-\theta)^2}{4},
\end{align*}
which implies that
\begin{align*}
\Im(\sqrt{v_0-iA_0})
\le-\frac{1}{2}(3-\theta)\alpha\pi.
\end{align*}

Inspirited by \cref{bounds of Im(u0-ia)}, we choose $u_L:=v_L-iA_0$ satisfied $\Im(\sqrt{u_L})=-\frac{1}{2}(5-\theta)\alpha\pi$. Then we have
from \cref{presention of sqrt u} and $\Im(\sqrt{u_0})=-\frac{1}{2}(2-\theta)\alpha\pi$ that
\begin{align*}
-\Im(\sqrt{u_L})=\sqrt{\frac{\sqrt{v_L^2+A_0^2}-v_L}{2}}
=\frac{5-\theta}{2-\theta}\sqrt{\frac{\sqrt{v_0^2+a_0^2}-v_0}{2}}
=-\frac{5-\theta}{2-\theta}\Im(\sqrt{u_0}),
\end{align*}
and
\begin{align}\label{bound for su-su0 on h-v1}
\frac{\Re(\sqrt{u_L})}{\Re(\sqrt{u_0})}
=\frac{\sqrt{\frac{\sqrt{v_L^2+A_0^2}+v_L}{2}}}
{\sqrt{\frac{\sqrt{v_0^2+a_0^2}+v_0}{2}}}
=\frac{4-\theta}{2-\theta}\frac{\Im(\sqrt{u_0})}{\Im(\sqrt{u_L})}
=\frac{4-\theta}{5-\theta}.
\end{align}
Subsequently, we have from \cref{monotonicity on h-baru} and \cref{bound for su-su0 on h-v1} that
\begin{align}\label{real_thefirstsegment}
\Re(\sqrt{u}-\sqrt{u_0})\le
\Re(\sqrt{u_L}-\sqrt{u_0})=-\frac{\Re(\sqrt{u_0})}{5-\theta}
\le\frac{-1}{5-\theta}\le-\frac{1}{5}
\end{align}
for $u=t-iA_0$ and $t\in[h,v_L]$.

Similarly,
by choosing $u_R:=v_R-iA_0$ satisfied $\Im(\sqrt{u_R})=-\frac{1}{2}(3-\theta)\alpha\pi$, we have
\begin{align}\label{real_thesecondsegment}
\Re(\sqrt{u}-\sqrt{u_0})\ge
\Re(\sqrt{u_R}-\sqrt{u_0})=\frac{1}{3-\theta}\Re(\sqrt{u_0})
\ge\frac{1}{3-\theta}>\frac{1}{3}
\end{align}
for $u=t-iA_0$ and $t\in[v_R,+\infty)$.

Thus now we can choose $v_L$ and $v_R$ as the dividing points, which are the real parts of $u_L$ and $u_R$, respectively. Furthermore, $u_L$ and $u_R$ satisfy well the condition \cref{bounds of Im(t-2ia)}.

\bigskip
\begin{lemma}\label{infty}
Let $f(u,z)$ be defined in \eqref{eq:func} with $x\in\mathcal{A}^*_\theta$. Then
\begin{align}
\bigg{|}
\int_{h-iA_0}^{+\infty-iA_0}f(u,z)e^{-i\frac{2n\pi}{h}u}du\bigg{|}
=&e^{-\frac{2n\pi}{h}A_0}x^{\alpha}\mathcal{O}(1),\label{bound_path_h-inf}\\
\bigg{|}
\int_{h+iA_1}^{+\infty+iA_1}f(u,z)e^{i\frac{2n\pi}{h}u}du\bigg{|}
=&e^{-\frac{2n\pi}{h}A_1}x^{\alpha}\mathcal{O}(1)\label{bound_path_h-inf111}
\end{align}
hold for all $T$ and the constants in $\mathcal{O}(1)$ are  independent of $\theta\in [0,\beta]$ for $z\in\mathcal{A}^*_\theta$.
\end{lemma}
\begin{proof}
We only prove the case \cref{bound_path_h-inf}, and \cref{bound_path_h-inf111} can be proved in the same way.

At first, we estimate  the integrand of \cref{bound_path_h-inf} on the subinterval $[h,v_L]$
\begin{align}\label{integrand}
\big|f(u,z)\big|
=\frac{\sin{(\alpha\pi)}}{2\alpha\pi}
\frac{C^{\alpha}\big|ze^{\sqrt{u}-T}\big|}
{\big|\sqrt{u}\big|\big|Ce^{\frac{1}{\alpha}(\sqrt{u}-T)}+z\big|}
=\frac{\sin{(\alpha\pi)}}{2\alpha\pi}
\frac{\big|e^{\sqrt{u}-T}\big|}
{\big|\sqrt{u}\big|\big|e^{\frac{1}{\alpha}(\sqrt{u}-\sqrt{u_0})}-1\big|},
\end{align}
where $u=t-iA_0,t\in[h,v_L]$.

From $\sqrt{u_0}=\alpha\log{\frac{x}{C}}+T-\frac{i}{2}(2-\theta)\alpha\pi$, we have
\begin{align}\label{e^{sqrt_u-T}}
\big{|}e^{\sqrt{u}-T}\big{|}
=\big{|}e^{(\sqrt{u_0}-T)+(\sqrt{u}-\sqrt{u_0})}\big{|}
=\left(\frac{x}{C}\right)^{\alpha}e^{\Re(\sqrt{u}-\sqrt{u_0})},
\end{align}
and the exponent can be estimated by $v_0\ge A_0$ and $v_0\ge24\alpha^2\pi^2$
as follows
\begin{align}\label{boundsfor_realpar_sqru-u0}
&\Re(\sqrt{u}-\sqrt{u_0})
=\sqrt{\frac{\sqrt{t^2+A_0^2}+t}{2}}
-\sqrt{\frac{\sqrt{v^2_0+a_0^2}+v_0}{2}}\\
=&\frac{1}{2}\frac{\sqrt{t^2+A_0^2}-\sqrt{v_0^2+a_0^2}+(t-v_0)}
{\sqrt{\frac{\sqrt{t^2+A_0^2}+t}{2}}
+\sqrt{\frac{\sqrt{v_0^2+a_0^2}+v_0}{2}}}\notag\hspace{.5cm}(t-v_0<0)\notag\\
=&\frac{1}{2}\frac{(t-v_0)
\big(\frac{t+v_0}{\sqrt{t^2+A_0^2}+\sqrt{v_0^2+a_0^2}}+1\big)
+\frac{A_0^2-a_0^2}{\sqrt{t^2+A_0^2}+\sqrt{v_0^2+a_0^2}}}
{\sqrt{\frac{\sqrt{t^2+A_0^2}+t}{2}}
+\sqrt{\frac{\sqrt{v_0^2+a_0^2}+v_0}{2}}}\notag\\
\le&\frac{\frac{t-v_0}{2}
\big(\frac{t+v_0}{t+a+v_0+2a_0}+1\big)}
{\sqrt{t+\frac{A_0}{2}}+\sqrt{v_0+\frac{a_0}{2}}}
+\frac{A_0^2-a_0^2}{2(\sqrt{t}+\sqrt{v_0})
\big[\frac{a_0^2}{(2-\theta)^2\alpha^2\pi^2}+\frac{1}{4}(2-\theta)^2\alpha^2\pi^2\big]}
\notag\\
\le&\frac{t-v_0}{6(\sqrt{t}+\sqrt{v_0})}
\bigg[\frac{t+v_0}{3(t+v_0)}+1\bigg]
+\frac{\frac{4(3-\theta)}{(2-\theta)^2}a_0^2}{
\frac{2a_0^2}{(2-\theta)^2\alpha^2\pi^2}\sqrt{v_0}}\notag\\
\le&\frac{2}{9}\big(\sqrt{t}-\sqrt{v_0}\big)
+\frac{1}{2}(3-\theta)\alpha\pi<\frac{2}{9}\big(\sqrt{t}-\sqrt{v_0}\big)
+2\alpha\pi\notag
\end{align}
for $u=t-iA_0,\ t\in[h,v_0](\supseteq[h,v_L])$.

Based on the estimations \cref{real_thefirstsegment} and \cref{boundsfor_realpar_sqru-u0} and by noticing $\big|\sqrt{t-iA_0}\big|>\sqrt{t}$, the integral of $\big|f(t-iA_0,z)\big|$ on the first subinterval $[h,v_L]$ satisfies that
\begin{align}\label{bound for first interval}
\int_{h}^{v_L}\big|f(t-iA_0,z)\big|dt
\le&\frac{x^{\alpha}e^{2\alpha\pi}\sin{(\alpha\pi)}}{\big(1-e^{-\frac{1}{5\alpha}}\big)\alpha\pi}
\int_{h}^{v_L}\frac{e^{\frac{2}{9}(\sqrt{t}-\sqrt{v_0})}}{2\sqrt{t}}dt\\
=&\frac{x^{\alpha}e^{2\alpha\pi}\sin{(\alpha\pi)}}{\big(1-e^{-\frac{1}{5\alpha}}\big)\alpha\pi}
\int_{0}^{+\infty}e^{-\frac{2}{9}s}ds\notag\\
=&x^{\alpha}\mathcal{O}(1)\notag
\end{align}
holds for all $T$ and $\mathcal{O}(1)$ is independent of $\theta\in[0,\beta]$ for $z\in\mathcal{A}^*_\theta$ due to $h<v_L<v_0$.

\bigskip
In order to estimate the integral on the third subinterval $[v_R,+\infty)$, we rewrite the integrand
$|f(u,z)|$ as
\begin{align}\label{estimation of integrand}
|f(u,z)|
=&\frac{\sin{(\alpha\pi)}}{\alpha\pi}
\frac{x^{\alpha}}{2\big|\sqrt{t-iA_0}\big|}
\frac{\big|e^{\sqrt{u}-\sqrt{u_0}}\big|}
{\big|e^{\frac{1}{\alpha}(\sqrt{u}-\sqrt{u_0})}\big|
\big{|}e^{-\frac{1}{\alpha}(\sqrt{u}-\sqrt{u_0})}-1\big{|}}\\
=&\frac{\sin{(\alpha\pi)}}{\alpha\pi}
\frac{x^{\alpha}}{2\sqrt[4]{t^2+A_0^2}
e^{\frac{1}{\kappa}\Re(\sqrt{u}-\sqrt{u_0})}
\big{|}1-e^{-\frac{1}{\alpha}(\sqrt{u}-\sqrt{u_0})}\big{|}},\notag
\end{align}
where $u=t-iA_0,t\in[v_R,+\infty)$.

Analogous to \cref{boundsfor_realpar_sqru-u0}, by using
$t-v_0>0$ and $\frac{A_0^2-a_0^2}{\sqrt{t^2+A_0^2}+\sqrt{v_0^2+a_0^2}}>0$ we have for the exponent $\Re(\sqrt{u}-\sqrt{u_0})$ that
\begin{align}\label{boundsfor_realpar_sqru-u0_v1-inf}
\frac{2}{9}(\sqrt{t}-\sqrt{v_0})\le\Re(\sqrt{u}-\sqrt{u_0}),
\end{align}
where $u=t-iA_0,\ t\in[v_0,+\infty)\left(\supseteq[v_R,+\infty)\right)$.

With the bounds \cref{real_thesecondsegment} and \cref{boundsfor_realpar_sqru-u0_v1-inf} in hand, we have
\begin{align*}
\big{|}1-e^{-\frac{1}{\alpha}(\sqrt{u}-\sqrt{u_0})}\big{|}
&\ge 1-e^{-\frac{1}{\alpha}\Re(\sqrt{u}-\sqrt{u_0})}
\ge  1-e^{-\frac{1}{3\alpha}},
\end{align*}
and
\begin{align}\label{bound for third interval}
\int_{v_R}^{+\infty}\big|f(t-iA_0,z)\big|dt
\le&\frac{x^{\alpha}\sin{(\alpha\pi)}}
{\big(1-e^{-\frac{1}{3\alpha}}\big)\alpha\pi}
\int_{v_R}^{+\infty}
\frac{e^{-\frac{2}{9\kappa}(\sqrt{t}-\sqrt{v_0})}}{2\sqrt{t}}dt\notag\\
\le&\frac{x^{\alpha}\sin{(\alpha\pi)}}
{\big(1-e^{-\frac{1}{3\alpha}}\big)\alpha\pi}
\int_{0}^{+\infty}
e^{-\frac{2}{9\kappa}s}ds\\
=&x^{\alpha}\mathcal{O}(1)\notag
\end{align}
holds for all $T$ and is independent of $\theta\in[0,\beta]$ for $z\in\mathcal{A}^*_\theta$ by $v_0<v_R$.

\bigskip
Now, we turn to the middle subinterval $[v_L,v_R]$. Since $\Im(\sqrt{u_0})=-\frac{1}{2}(2-\theta)\alpha\pi$, it is easy to check by \cref{bounds of Im(t-2ia)}, \cref{real_thefirstsegment} and \cref{real_thesecondsegment} that
\begin{align*}
-\frac{3\pi}{2}\le\frac{1}{\alpha}\Im(\sqrt{u}-\sqrt{u_0})
\le-\frac{\pi}{2},
\hspace{0.5cm}
-\frac{1}{5}\le\Re(\sqrt{u}-\sqrt{u_0})\le\frac{1}{3}
\end{align*}
hold for $u=t-iA_0,t\in[v_L,v_R]$,
which implies that
\begin{align}\label{bound for sqrt-sign1}
\big|e^{\frac{1}{\alpha}(\sqrt{u}-\sqrt{u_0})}-1\big|
=&\sqrt{e^{\frac{2}{\alpha}\Re(\sqrt{u}-\sqrt{u_0})}
-2e^{\frac{1}{\alpha}\Re(\sqrt{u}-\sqrt{u_0})}
\cos{\bigg(\frac{\varpi_d}{\alpha}\bigg)}+1}\\
\ge&\sqrt{1+e^{\frac{2}{\alpha}\Re(\sqrt{u}-\sqrt{u_0})}}
\ge\sqrt{1+e^{-\frac{2}{5\alpha}}}>1,\notag
\end{align}
for $u=t-iA_0,\ t\in[v_L,\hat v]$,
where $\varpi_d:=\Im(\sqrt{u}-\sqrt{u_0})$.
Similarly, we have
\begin{align}\label{bound for sqrt-sign2}
\big|1-e^{-\frac{1}{\alpha}(\sqrt{u}-\sqrt{u_0})}\big|
\ge\sqrt{1+e^{-\frac{2}{3\alpha}}}>1
\end{align}
for $u=t-iA_0,\ t\in[\hat v,v_R]$.

Furthermore,
we have according to \cref{monotonicity on h-baru} and \cref{monotonicity on baru-infty}
that
\begin{align}\label{bounds middle interval}
&\int_{v_L}^{v_R}\big|f(t-iA_0,z)\big|dt
=\bigg(\int_{v_L}^{\hat v}+\int_{\hat v}^{v_R}\bigg)\big|f(t-iA_0,z)\big|dt\\
=&\frac{\sin{\alpha\pi}}{\alpha\pi}
\int_{v_L}^{\hat v}\frac{x^{\alpha}e^{\Re(\sqrt{u}-\sqrt{u_0})}}
{2\big|\sqrt{u}\big|\big|e^{\frac{1}{\alpha}(\sqrt{u}-\sqrt{u_0})}-1\big|}dt\notag\\
&+\frac{\sin{\alpha\pi}}{\alpha\pi}
\int_{\hat v}^{v_R}\frac{x^{\alpha}e^{-\frac{1}{\kappa}\Re(\sqrt{u}-\sqrt{u_0})}}
{2\big|\sqrt{u}\big|
\big|1-e^{-\frac{1}{\alpha}\Re(\sqrt{u}-\sqrt{u_0})}\big|}dt
\notag\\
\le&\frac{x^{\alpha}e^{2\alpha\pi}\sin{(\alpha\pi)}}{\alpha\pi}
\int_{v_L}^{\hat v}\frac{e^{\frac{2}{9}(\sqrt{t}-\sqrt{v_0})}}
{2\sqrt{t}}dt
\hspace{1.2cm}(\text{by \cref{boundsfor_realpar_sqru-u0}, \cref{bound for sqrt-sign1}})\notag\\
&+\frac{x^{\alpha}\sin{(\alpha\pi)}}
{\alpha\pi}
\int_{\hat v}^{v_R}\frac{e^{-\frac{2}{9\kappa}(\sqrt{t}-\sqrt{v_0})}}
{2\sqrt{t}}dt\notag\hspace{.9cm}(\text{by \cref{boundsfor_realpar_sqru-u0_v1-inf}, \cref{bound for sqrt-sign2}})\\
=&x^{\alpha}\mathcal{O}(1)\notag
\end{align}
holds for all $T$ and $\mathcal{O}(1)$ is independent of $\theta\in[0,\beta]$ for $z\in\mathcal{A}^*_\theta$.

Adding \cref{bound for first interval}, \cref{bound for third interval} and \cref{bounds middle interval} all up, we prove the case of
\begin{align*}
\bigg{|}
\int_{h-iA_0}^{+\infty-iA_0}f(u,z)e^{-i\frac{2n\pi}{h}u}du\bigg{|}
=&e^{-\frac{2n\pi}{h}A_0}\bigg{|}\int_{h}^{+\infty}f(t-iA_0,z)e^{- i\frac{2n\pi}{h}t}dt\bigg{|}\\
\le&e^{-\frac{2n\pi}{h}A_0}\int_{h}^{+\infty}\big{|}f(t-iA_0,z)\big{|}dt\\
=&e^{-\frac{2n\pi}{h}A_0}x^{\alpha}\mathcal{O}(1).
\end{align*}
\end{proof}

\begin{lemma}\label{LemmalogA}
Let $f_{log}(u,z)=\frac{1}{\alpha}(\sqrt{u}-T)f(u,z)$  with $z\in \mathcal{A}^*_\theta$ and the conditions of Lemmas \ref{la3}, \ref{lemma_inte_circ} and \ref{infty} hold, respectively. Then we have that
\begin{align}\label{eq:line}
&\bigg{|}\int_{0}^{A_1}f_{log}(h+it,z)
e^{-\frac{2n\pi}{h}t}dt-\int_{0}^{A_0}f_{log}(h-it,z)
e^{-\frac{2n\pi}{h}t}dt\bigg{|}
=\mathcal{O}(Te^{-T}),
\end{align}
\begin{align}
\bigg{|}\int_{C^{-}_{\rho}}
f_{log}(u,z)e^{-i\frac{2n\pi}{h}u}du\bigg{|}
=&e^{-\frac{2n\pi}{h}a_0}x^{\alpha}\mathcal{O}(1),\label{eq:circle}\\
\bigg{|}
\int_{C^{+}_{\rho}}
f_{log}(u,z)e^{i\frac{2n\pi}{h}u}du\bigg{|}
=&e^{-\frac{2n\pi}{h}a_1}x^{\alpha}\mathcal{O}(1),\label{eq:circle1}
\end{align}
and
\begin{align}
\bigg{|}
\int_{h-iA_0}^{+\infty-iA_0}f_{log}(u,z)e^{-i\frac{2n\pi}{h}u}du\bigg{|}
=&Te^{-\frac{2n\pi}{h}A_0}x^{\alpha}\mathcal{O}(1)
=e^{-\frac{2n\pi}{h}a_0}x^{\alpha}\mathcal{O}(1),\label{bound_path_h-inflog}\\
\bigg{|}
\int_{h+iA_1}^{+\infty+iA_1}f_{log}(u,z)e^{i\frac{2n\pi}{h}u}du\bigg{|}
=&Te^{-\frac{2n\pi}{h}A_1}x^{\alpha}\mathcal{O}(1)
=e^{-\frac{2n\pi}{h}a_1}x^{\alpha}\mathcal{O}(1)\label{bound_path_h-inf111log}
\end{align}
%
%
%
hold for sufficiently large $T$ and the constants in $\mathcal{O}$ are
independent of $n$, $T$, $\theta\in [0,\beta]$ and $x\in [x^*,1]$.
\end{lemma}
\begin{proof}
For the proof of \cref{eq:line},  \cref{bound_path_h-inflog} and \cref{bound_path_h-inf111log}, we rewrite $f_{log}$ as
\begin{align*}
f_{log}(u,z)=\frac{\sin{(\alpha\pi)}}{2\alpha^2\pi}\frac{zC^{\alpha}e^{\sqrt{u}-T}}
{Ce^{\frac{1}{\alpha}(\sqrt{u}-T)}+z}-\frac{1}{\alpha}Tf(u,z),
\end{align*}
which, together with \cref{eq:dom1}, \cref{eq:dom2} and \Cref{la3}, yields
\begin{align*}
&\bigg{|}\int_{0}^{A_1}f_{log}(h+it,z)
e^{-\frac{2n\pi}{h}t}dt-\int_{0}^{A_0}f_{log}(h-it,z)
e^{-\frac{2n\pi}{h}t}dt\bigg{|}\\
\le&\mathcal{O}(e^{-T})\left\{\int_{0}^{A_1}
+\int_{0}^{A_0}\right\}e^{\sqrt{t}+\sqrt{h}}
e^{-\frac{2n\pi}{h}t}dt+\mathcal{O}(Te^{-T})\\
=&\mathcal{O}(e^{-T})
\int_0^{+\infty}e^{\sqrt{t}+\sqrt{h}}e^{-\frac{2n\pi}{h}t}dt+\mathcal{O}(Te^{-T})
\end{align*}
and then leads to the desired result \cref{eq:line}.

Estimates \cref{bound_path_h-inflog} and \cref{bound_path_h-inf111log} follow from  \cref{infty} and a similar proof without the integrand $\frac{1}{2\sqrt{t}}$ and multiplied by a factor $\frac{1}{\alpha}$ in
\cref{bound for first interval}, \cref{bound for third interval} and \cref{bounds middle interval}  on the integral of $\frac{\sin{(\alpha\pi)}}{2\alpha^2\pi}\frac{zC^{\alpha}e^{\sqrt{u}-T}}
{Ce^{\frac{1}{\alpha}(\sqrt{u}-T)}+z}$, respectively.

Finally, it easily to check by \cref{eq:real_imag_part_v0a0} and \cref{eq:real_imag_part_v1a1} that
\begin{align*}
\sqrt{u}-T=\frac{u-T^2}{\sqrt{u}+T}=\mathcal{O}(1)
\end{align*}
for $u=u_l+\rho e^{i\vartheta},\ \vartheta:0\rightarrow\pm2\pi,\ l=0,1$, as $T$ approaches infinity. Then by $f_{log}(u,z)=\frac{\sin{(\alpha\pi)}}{2\alpha^2\pi}\frac{\sqrt{u}-T}{2\sqrt{u}}\frac{zC^{\alpha}e^{\sqrt{u}-T}}
{Ce^{\frac{1}{\alpha}(\sqrt{u}-T)}+z}$, with the same argument of \cref{lemma_inte_circ} we arrive at \cref{eq:circle} and \cref{eq:circle1}.


\begin{remark}
It is worthy of noticing that all the constants in $\mathcal{O}$ of the statements in Lemmas \ref{la3}-\ref{LemmalogA} are independent of $\theta\in [0,\beta]$ and $x\in [x^*,1]$ for $z\in\mathcal{A}^*_\theta$, which is important for deriving the uniform convergence rates of quadrature errors of $I(z)$ and $I_{\log}$ for $z\in S_\beta$.
\end{remark}

\end{proof}

\end{appendix}

\section*{Acknowledgement}
The authors would like to thank Guidong Liu, Yanghao Wu and Yuee Zhong  for their helpful discussions.

\bibliographystyle{siamplain}
\bibliography{references}
\end{document}

%% file: ex_shared.tex
\usepackage{lipsum}
\usepackage{amsfonts,amsmath, amssymb}
\usepackage{graphicx}
\usepackage[caption=false,farskip=0pt]{subfig}
\usepackage{epstopdf}
\usepackage{algorithmic,arydshln}
\ifpdf
  \DeclareGraphicsExtensions{.eps,.pdf,.png,.jpg}
\else
  \DeclareGraphicsExtensions{.eps}
\fi

\newsiamremark{remark}{Remark}
\newsiamremark{case}{Case}
\newsiamremark{hypothesis}{Hypothesis}
\crefname{hypothesis}{Hypothesis}{Hypotheses}
\newsiamthm{claim}{Claim}

\headers{Root-exponential convergence of LP approximation}{Shuhuang Xiang and Shunfeng Yang}

\title{The root-exponential convergence of lightning plus polynomial approximation on corner domains\thanks{Submitted to the editors DATE.
\funding{This work was funded by National Science Foundation of China (No. 12271528).}}}

\author{Shuhuang Xiang\thanks{School of Mathematics and Statistics, Central South University, Changsha 410083, Hunan, People's Republic of China
  (\email{xiangsh@csu.edu.cn}).}
\and Shunfeng Yang\thanks{School of Mathematics and Statistics, Central South University, Changsha 410083, Hunan, People's Republic of China
  (\email{yangshunfeng@163.com}), corresponding author.}}

\usepackage{amsopn}


%% file: LPV20240106V2.bbl
\begin{thebibliography}{10}

\bibitem{Asmar2018ComplexAW}
{\sc N.~H. Asmar and L.~Grafakos}, {\em Complex Analysis with Applications},
  Springer Nature Switzerland AG, 2018,
  \url{https://api.semanticscholar.org/CorpusID:187303207}.

\bibitem{Brubeck2022}
{\sc P.~D. Brubeck and L.~N. Trefethen}, {\em Lightning {S}tokes solver}, SIAM
  J. Sci. Comput., 44 (2022), pp.~A1205--A1226,
  \url{https://doi.org/10.1137/21M1408579}.

\bibitem{Gaier1987}
{\sc D.~Gaier}, {\em Lectures on complex approximation},
  Birkh${\rm\ddot{a}}$user, Basel, 1987.

\bibitem{GopTre2019}
{\sc A.~Gopal and L.~N. Trefethen}, {\em New {L}aplace and {H}elmholtz
  solvers}, Proc. Nat. Acad. Sci. USA, 116 (2019), pp.~10223--10225,
  \url{https://doi.org/10.1073/pnas.1904139116}.

\bibitem{Gopal2019}
{\sc A.~Gopal and L.~N. Trefethen}, {\em Representation of conformal maps by
  rational functions}, Numer. Math., 142 (2019), pp.~359--382,
  \url{https://doi.org/10.1007/s00211-019-01023-z}.

\bibitem{Gopal20191}
{\sc A.~Gopal and L.~N. Trefethen}, {\em Solving {L}aplace problems with corner
  singularities via rational functions}, SIAM J. Numer. Anal., 57 (2019),
  pp.~2074--2094, \url{https://doi.org/10.1137/19M125947X}.

\bibitem{GR2014}
{\sc I.~Gradshteyn and I.~Ryzhik}, {\em Table of Integrals, Series, and
  Products}, Academic Press, Cambridge, 5th~ed., 2014.

\bibitem{Henrici}
{\sc P.~Henrici}, {\em Applied and Computational Complex Analysis, Volume 2,
  Special Functions-Integral Transforms-Asymptotics-Continued Fractions}, John
  Wiley \& Sons, London, 1977.

\bibitem{Herremans2023}
{\sc A.~Herremans, D.~Huybrechs, and L.~N. Trefethen}, {\em Resolution of
  singularities by rational functions}, SIAM Journal on Numerical Analysis, 61
  (2023), pp.~2580--2600, \url{https://doi.org/10.1137/23M1551821}.

\bibitem{Lehman1954DevelopmentsIT}
{\sc R.~S. Lehman}, {\em Developments in the neighborhood of the beach of
  surface waves over an inclined bottom}, Communications on Pure and Applied
  Mathematics, 7 (1954), pp.~393--439,
  \url{https://api.semanticscholar.org/CorpusID:121096733}.

\bibitem{Lehman1957}
{\sc R.~S. Lehman}, {\em Development of the mapping function at an analytic
  corner.}, Pacific Journal of Mathematics, 7 (1957), pp.~1437--1449,
  \url{https://api.semanticscholar.org/CorpusID:123120719}.

\bibitem{Mcleod1965}
{\sc R.~M. Mcleod}, {\em Mean value theorems for vector valued functions},
  Proc. Edinburgh Math. Soc., 14 (1965), pp.~197--209,
  \url{https://doi.org/10.1017/S0013091500008786}.

\bibitem{Nakatsukasa2021}
{\sc Y.~Nakatsukasa and L.~N. Trefethen}, {\em Reciprocal-log approximation and
  planar {PDE} solvers}, SIAM J. Numer. Anal., 59 (2021), pp.~2801--2822,
  \url{https://doi.org/10.1137/20M1369555}.

\bibitem{1984Complex}
{\sc R.~A. Silverman}, {\em Complex Analysis with Applications}, Dover
  Publications, 1984.

\bibitem{Stahl2003}
{\sc H.~R. Stahl}, {\em Best uniform rational approximation of {$x^\alpha$} on
  {$[0,1]$}}, Acta Math., 190 (2003), pp.~241--306,
  \url{https://doi.org/10.1007/BF02392691}.

\bibitem{Trefethen2021}
{\sc L.~N. Trefethen, Y.~Nakatsukasa, and J.~A.~C. Weideman}, {\em Exponential
  node clustering at singularities for rational approximation, quadrature, and
  {PDE}s}, Numer. Math., 147 (2021), pp.~227--254,
  \url{https://doi.org/10.1007/s00211-020-01168-2}.

\bibitem{Trefethen2014SIREV}
{\sc L.~N. Trefethen and J.~A.~C. Weideman}, {\em The exponential convergent
  trapezoidal rule}, SIAM Rev., 56 (2014), pp.~385--458,
  \url{https://doi.org/10.1137/130932132}.

\bibitem{Walsh1965}
{\sc J.~L. Walsh}, {\em Interpolation and approximation by rational functions
  in the complex domain}, American Mathematical Society, Providence, 5th~ed.,
  1969.

\bibitem{Wasow}
{\sc W.~Wasow}, {\em Asymptotic development of the solution of dirichlet’s
  problem at analytic corners}, Duke Mathematical Journal, 24 (1957),
  pp.~47--56, \url{https://api.semanticscholar.org/CorpusID:117990401}.

\bibitem{XY2023}
{\sc S.~Xiang, S.~Yang, and Y.~Wu}, {\em On the best convergence rate of
  lightning plus polynomial approximation for $x^{\alpha}$}, preprint,  (2023),
  \url{https://arxiv.org/abs/2312.16116}.

\end{thebibliography}
